\documentclass[11pt]{amsart}
\usepackage[utf8]{inputenc}
\usepackage[T1]{fontenc}
\usepackage[english]{babel}
\usepackage{graphicx}
\usepackage{amsthm}
\usepackage{amssymb}
\usepackage{stmaryrd}
\usepackage{csquotes}
\usepackage{amsmath,mleftright}
\usepackage{xparse}
\usepackage{braket}
\usepackage{hyperref}
\hypersetup{colorlinks,linkcolor={violet},citecolor={magenta},urlcolor={green}}
\usepackage{subfig}
\usepackage{amsfonts}
\usepackage{caption}
\usepackage{subcaption}
\usepackage{dsfont}
\usepackage{enumerate}
\usepackage{enumitem}
\usepackage{tikz-cd}
\usepackage{tikz}
\usetikzlibrary{scopes}
\usetikzlibrary{backgrounds}
\usetikzlibrary{arrows.meta}
\usetikzlibrary{math}
\usetikzlibrary {datavisualization}
\usepackage{tabularx}
\usepackage{siunitx}
\usepackage[top=1 in,bottom=1in, left=2cm, right=2cm]{geometry}
\usepackage{placeins}
\usepackage{xcolor}
\usepackage{todonotes}
\usepackage{url}
\usepackage{comment}
\usepackage{caption}
\usepackage{setspace}
\usepackage[bottom]{footmisc}
\usepackage[cmtip,all]{xy}

\newtheorem{proposition}{Proposition}[section]
\newtheorem{lemma}[proposition]{Lemma}
\newtheorem{theorem}[proposition]{Theorem}
\newtheorem{corollary}[proposition]{Corollary}
\newtheorem{conjecture}{Conjecture}[section]
\theoremstyle{definition}
\newtheorem{remark}[proposition]{Remark}
\newtheorem{definition}[proposition]{Definition}
\newtheorem{example}[proposition]{Example}

\DeclareMathOperator{\Id}{Id}

\DeclareMathOperator{\Aut}{Aut}

\newcommand{\w}{\omega} %kahler metric
\newcommand{\C}{\mathbb{C}} %complex numbers
\newcommand{\R}{\mathbb{R}} %real numbers
\newcommand{\proj}{\mathbb{P}} %projective space
\newcommand{\Z}{\mathbb{Z}}
\newcommand{\Q}{\mathbb{Q}}

\newcommand{\Fuk}{\mathcal{F}uk}
\newcommand{\Cob}{\mathcal{C}ob}
\newcommand{\flux}{\text{Flux}}
\newcommand{\D}{\mathcal{D}}
\newcommand{\Gr}{\text{Griff}}
\newcommand{\Symp}{Symp}
\newcommand{\ham}{Ham}

\newcommand{\algcob}{algCob}
\newcommand{\longrightsquigarrow}{\xymatrix{{}\ar@{~>}[r]&{}}}

\title{Algebraic Lagrangian cobordisms, flux and the Lagrangian Ceresa cycle }
\author{Alexia Corradini}
\email{ac2478@cam.ac.uk}

\numberwithin{equation}{section}

\usepackage[backend=biber,style=alphabetic, sorting=nyt]{biblatex}
\begin{document}

\begin{abstract}
	We introduce an equivalence relation for Lagrangians in a symplectic manifold known as \textit{algebraic Lagrangian cobordism}, which is meant to mirror algebraic equivalence of cycles. From this we prove a symplectic, mirror-symmetric analogue of the statement \enquote{the Ceresa cycle is non-torsion in the Griffiths group of the Jacobian of a generic genus $3$ curve}. Namely, we show that for a family of tropical curves, the \textit{Lagrangian Ceresa cycle}, which is the Lagrangian lift of their tropical Ceresa cycle to the corresponding Lagrangian torus fibration, is non-torsion in its oriented algebraic Lagrangian cobordism group. We proceed by developing the notions of tropical (resp. symplectic) flux, which are morphisms from the tropical Griffiths (resp. algebraic Lagrangian cobordism) groups. 
\end{abstract}
	
\maketitle

\section{Introduction}\label{introduction}

In a smooth projective variety, studying and comparing the different equivalence relations between algebraic cycles can shed light on a wealth of interesting geometric phenonema. A chief example is that of the Ceresa cycle of a curve, which is a $1$-cycle living in its Jacobian, and provided one of the first instances of a cycle which was shown to be homologically but not algebraically trivial \cite{Ceresa}. This discrepancy between homological and algebraic equivalence is encoded in the \textit{Griffiths group} of the variety. Our aim is to propose a symplectic counterpart to this group and use it to exhibit, through an $SYZ$ picture, a symplectic manifestation of the fact that the Ceresa cycle is not algebraically trivial.

We survey below how the (cylindrical) Lagrangian cobordism group of a symplectic manifold $(M,\w)$ is known to be closely related to the Chow groups of its mirror. However, a candidate mirror to its Griffiths groups (that is, the quotient of the Chow groups by algebraic equivalence) has, to the best of our knowledge, not yet appeared in the literature. We propose such a mirror by introducing an equivalence relation which we call \textit{algebraic Lagrangian cobordism} on the set of suitable Lagrangians in a symplectic manifold, and from this we define the \textit{algebraic Lagrangian cobordism group} of $(M,\w)$, $\algcob(M,\w)$. 

While rational equivalence is defined in terms of flat families of cycles over $\proj^1$, algebraic equivalence is defined in terms of flat families of cycles over curves of \textit{arbitrary} genus. These are known to have vastly more complicated mirrors, making a translation into symplectic geometry more subtle. The insight we use to circumvent this difficulty is that higher genus curves can be embedded in their Jacobian torus, which are abelian varieties. These have SYZ mirrors which we call \textit{polarised symplectic tori}. From here, we define algebraic Lagrangian cobordisms in terms of \enquote{Lagrangian families} over a polarised symplectic torus. More explicitly, an algebraic Lagrangian cobordism between two Lagrangians $L$ and $L'$ in a symplectic manifold $(M,\w)$ is a Lagrangian correspondence $\Gamma\subset (X(B)\times M, -\w_{X(B)}\oplus\w)$ such that $\Gamma(F_p)=L$ and $\Gamma(F_q)=L'$, where $(X(B),\w_{X(B)})$ is a symplectic polarised torus and $F_p,F_q\subset X(B)$ are two Lagrangian tori (Definition \ref{alg_lag_cob}). The algebraic Lagrangian cobordism group of $(M,\w)$, $\algcob(M,\w)$, is then defined as the quotient of integer combinations of Lagrangians by the equivalence relation generated by algebraic Lagrangian cobordism.

The predicted relation between the Lagrangian cobordism group of a symplectic manifold $M$, denoted $\Cob(M,\w)$, and the Chow groups of its mirror $Y$, was first brought to light in the context of SYZ fibrations by Sheridan--Smith in \cite{nick_ivan1}. Their constructions rely on the following observation. Let $B$ be a tropical torus, $Z_1$ and $Z_2$ be tropical cycles in $B$, and consider the Lagrangian torus fibration  $$X(B):=T^*B/T^*_\Z B\longrightarrow B.$$ A rational equivalence between $Z_1$ and $Z_2$, if it admits a Lagrangian lift to $X(B\times\R)\cong X(B)\times\C^*$, defines a (cylindrical) Lagrangian cobordism between the Lagrangian lifts $L_1$ and $L_2$ of $Z_1$ and $Z_2$.

From the perspective of homological mirror symmetry, this correspondence can be formulated as follows. An equivalence of triangulated categories $$\D^\pi\Fuk(M)\cong D^bCoh(Y)$$ would in particular imply an isomorphism $$K_0(D^\pi\Fuk(M))\cong K_0(D^bCoh(Y))$$ between their Grothendieck groups. On the right-hand side, assuming $Y$ is smooth projective, it is known that $$K_0(D^bCoh(Y))_\Q\cong K_0(Coh(Y))_\Q\cong CH_*(Y)_\Q,$$ where $CH_*(Y)$ is the Chow ring of $Y$ and the second isomorphism is given by the Chern character map $ch$; see \cite[Section 2.3]{gillet_Ktheory_intersection_theory}. On the left-hand side, Biran$-$Cornea have shown in \cite{bc1,Biran_Cornea_2} that $K_0(D\Fuk(M))$ is closely related to (the appropriate flavour of) the Lagrangian cobordism group $\Cob(M)$, which we define in section \ref{symplectic_flux}. More precisely, Lagrangian cobordisms yield triangle decompositions in the Fukaya category, therefore there is a surjective morphism $$\Cob(M)\longrightarrow K_0(\Fuk(M)).$$ In some cases, we know this to be an isomorphism \cite{haug_T2,alvaro_bielliptic,unob_lag_cob_gps_surfaces}, and it is interesting to note that the first two of these examples rely on a proof of homological mirror symmetry hence on understanding the mirror Chow groups. Purely symplectically, this line of questions amounts to uncovering which triangle decompositions in $\Fuk(M)$ can be realised geometrically by Lagrangian cobordisms.

For symplectic topologists, one of the motivations for studying the groups $\Cob(M)$ associated to a symplectic manifold $M$ is that Lagrangian cobordism provides a coarser equivalence relation than Hamiltonian isotopy (see Example \ref{suspension_cobordism}). Therefore understanding $\Cob(M)$ ought to be simpler than the famously hard problem of classifying Lagrangians in $(M,\w)$ up to Hamiltonian isotopy. Nevertheless, currently there are only few instances in which these cobordism groups are fully understood. In dimension two, Arnold computed $\Cob(T^*S^1)$ using flux arguments \cite{Arnold_Lag_Leg_cobordisms} (see Section \ref{symplectic_flux}), then Haug used similar techniques to compute $\Cob(T^2)$ \cite{haug_T2}. Rathael-Fournier computed Lagrangian cobordisms groups for higher-genus surfaces through a different approach \cite{unob_lag_cob_gps_surfaces}, namely using the action of the mapping class group to obtain generators following work of Abouzaid \cite{abouzaid_fukaya_cat_higher_genus}. The first computation for compact symplectic manifolds in dimension four was carried out by Muñiz-Brea in \cite{alvaro_bielliptic} for symplectic bielliptic surfaces, restricting to tropical Lagrangians. In the non-compact case, Bosshard studies $\Cob(M)$ when $M$ is a Liouville manifold, and computes these for non-compact Riemann surfaces of finite type \cite{bosshard_Lag_cob_Liouville}. Here again, we have omitted information about the exact flavour of cobordism group considered in each of these cases.

In light of the observations above, it is natural to wonder whether one could further our understanding of Lagrangian cobordism groups by translating interesting facts about Chow groups to the symplectic realm. More ambitiously, after having introduced algebraic Lagrangian cobordisms as an equivalence relation on Lagrangians \enquote{mirror} to algebraic equivalence of cycles, one could wonder whether information about Griffiths groups can also be accessed symplectically. With these objectives in mind, the story of the Ceresa cycle is a natural candidate to turn one's attention to.

Our main Theorem builds upon a \textit{tropical} version of the Ceresa cycle story, worked out by Zharkov in \cite{zharkov_tropical_ceresa}. To a certain type of genus $3$ tropical curve $C$, we associate a nullhomologous Lagrangian $3$-manifold $L_C-L_C^-$ called the \textit{Lagrangian Ceresa cycle} of $C$. Its ambient symplectic manifold is a symplectic $6$-torus $X(J(C))$, also built out of $C$. We prove the following:
\begin{theorem}\label{intro_main_theorem}
	For generic tropical curves $C$ of a certain type, the Lagrangian Ceresa cycle $L_C-L_C^-$ has infinite order in $\algcob^{or}(X(J(C)))$. 
\end{theorem}
Here $\algcob^{or}(X(J(C)))$ is the oriented algebraic Lagrangian cobordism group of $X(J(C))$.
 
Our constructions for $L_C-L_C^-$ as well as $X(J(C))$ build upon those of Zharkov, as does the proof of \ref{intro_main_theorem}. Therefore we begin with an outline of the tropical story. More details, as well as the relevant background in tropical geometry can be found in Section \ref{tropical_background}. 

In \cite{zharkov_tropical_ceresa}, Zharkov considers a genus $3$ tropical curve $C$ \textit{of type $K4$}, represented below in Figure \ref{intro_K4}. Its data is that of the underlying graph, and the choice of real positive lengths $a,b,c,d,e,f$ for the edges $A,B,C,D,E,F$. 

\begin{figure}[htb]
	\centering
	\scalebox{0.6}{
	\begin{tikzpicture}
		\pgfdeclarelayer{nodelayer} 
		\pgfdeclarelayer{edgelayer}
		\pgfsetlayers{main,nodelayer,edgelayer}
		\begin{pgfonlayer}{nodelayer}
			\node [fill,circle,label={above right:$p$}] (0) at (0, 0) {};
			\node [fill,circle,color=red] (1) at (0, 3.75) {};
			\node [fill,circle,color=lime] (2) at (-3.5, -2) {};
			\node [fill,circle,color=gray] (3) at (3.5, -2) {};
			
			\node[label={left: E}] (4) at (0.1,1.875) {};
			\node[label={above:F}] (5) at (-1.75,-1.7) {};
			\node[label={above:D}] (6)  at (1.75,-1.7) {};
			\node[label={below left:A}] (7) at (-3.0,1.75) {};
			\node[label={below left:C}] (8) at (3.8,1.75) {};
			\node[label={above:B}] (9) at (0,-4.15) {};
		\end{pgfonlayer}
		\begin{pgfonlayer}{edgelayer}
			\draw[very thick,color=magenta] (1.center) to (0.center);
			\draw[very thick,color=green] (0.center) to (2.center);
			\draw[very thick,color=violet](0.center) to (3.center);
			\draw [very thick, color=cyan,bend right=45] (1.center) to (2.center);
			\draw [very thick, color=blue, bend left=45] (1.center) to (3.center);
			\draw [very thick, color=orange,bend right=45] (2.center) to (3.center);
		\end{pgfonlayer}
	\end{tikzpicture}}
	\caption{A tropical curve of type $K4$.}
	\label{intro_K4}
\end{figure}

 After an arbitrary choice of basepoint, $C$ can be embedded in its \textit{tropical Jacobian} $J(C)$, which in this case is a $3$-dimensional tropical torus. The image of this embedding is represented in Figure \ref{K4_in_Jacobian}. This Jacobian has a natural group structure, therefore an involution by taking inverses for the group law. The \textit{Ceresa cycle} of $C$ is a nullhomologous $1$-cycle in $J(C)$: it is the difference between the image of the curve in its Jacobian, and its image composed with this involution. We write it as $C-C^-\in \Gr^{trop}_1(J(C))$, where $\Gr^{trop}_1(J(C))$ denotes the $1$-dimensional tropical Griffiths group of $J(C)$. Zharkov proves the following:

\begin{theorem}[\cite{zharkov_tropical_ceresa,Ceresa}]\label{intro_Ceresa_thm}
	For a generic genus $3$ tropical curve of type $K4$, $C-C^-$ is not trivial in $\Gr^{trop}_1(J(C))$.
\end{theorem} 
This is a tropical version of a famous Theorem of Ceresa \cite{Ceresa}. 

The word \textit{generic} here refers to the choice of edge lengths of $C$, which are $6$ positive real numbers, being generic. The proof of Theorem \ref{intro_Ceresa_thm} is sketched as follows. Zharkov shows that if two $1$-cycles in $J(C)$ are algebraically equivalent (in particular they are homologous), then certain integrals over $2$-chains realising their homological equivalence are constrained. More precisely, he defines a \textit{determinantal form} $\Omega_0$: this is a $2$-form with coefficients encoding information about the tropical structure of $C$. He shows that if $C$ and $C^-$ were algebraically equivalent in $J(C)$, then for any $2$-chain $\gamma$ with $\partial\gamma=C-C^-$ we would have $$\int_\gamma\Omega_0\in \mathcal{P}(\Omega_0),$$ where $\mathcal{P}(\Omega_0)$ is the set of periods of $\Omega_0$. He then constructs an explicit such $2$-chain $\gamma_0$ and computes $$\int_{\gamma_0}\Omega_0=-ad.$$  After a straightforward verification, one observes that $-ad$ is generically not in $\mathcal{P}(\Omega_0)$, and concludes.

We now return to the statement of Theorem \ref{intro_main_theorem}. To the tropical torus $J(C)$, one can associate a symplectic manifold $$X(B):=T^*J(C)/T^*_\Z J(C),$$ where $T^*_\Z J(C)\subset T^*J(C)$ is a full rank lattice determining the tropical structure of $J(C)$. This is a $6$-dimensional symplectic torus constructed solely from the data of the tropical curve $C$. In fact, the cohomology class of the symplectic form is determined by its nine non-vanishing periods, which are the entries of the \textit{polarisation matrix} of the curve $C$:
\begin{equation}
	Q=	\begin{pmatrix}
		a+e+f & -f & -e \\ -f & b+d+f & -d \\ -e & -d & c+d+e
	\end{pmatrix}.\notag
\end{equation}

The Lagrangian Ceresa cycle $L_C-L_C^-$ associated to $C$ is the Lagrangian lift of $C-C^-$ to the Lagrangian torus fibration $$X(J(C))\longrightarrow J(C). $$ Each of its two embedded pieces is an oriented graph $3$-manifold built out of the $K4$ graph from Figure \ref{intro_K4}, with first and second Betti numbers $6$.

The proof of Theorem \ref{intro_main_theorem} can be sketched as follows. We show in Proposition \ref{alg_lag_cob} that if $L_C$ and $L_C^-$ were algebraic Lagrangian cobordant (in particular homologous), then for any $4$-chain $\gamma$ in $X(J(C))$ with $\partial\gamma=L_C-L_C^-$, we would have $$\int_\gamma\w^2\in\mathcal{P}(\w^2),$$ where $\mathcal{P}(\w^2)$ denotes the set of periods of $\w^2$.

Because $L_C-L_C^-$ is a Lagrangian lift to $X(J(C))$ of $C-C^-\subset J(C)$, we can build a $4$-chain $\hat{\phi}(\gamma_0)$ in $X(J(C))$ satisfying $\partial\hat{\phi}(\gamma_0)=L_C-L_C^-$ by \enquote{lifting} Zharkov's $2$-chain $\gamma_0$ in $J(C)$. 
This involves a conormal-type construction which has the virtue of satisfying
\begin{equation}\label{intro_flux_equality}
	\int_{\hat{\phi}(\gamma_0)}\w^2=\int_{\gamma_0}\Omega_0.
\end{equation} 
This equality stems from the fact that the $4$-form $\w^2$ on $X(J(C))$ can be obtained from the $2$-form $\Omega_0$ on $J(C)$ through a construction \enquote{dual} to the one taking $\gamma_0$ to $\hat{\phi}(\gamma_0)$. 
In particular, we have $$\mathcal{P}(\w^2)=\mathcal{P}(\Omega_0).$$ Putting this together with Zharkov's computation yields $$\int_{\hat{\phi}(\gamma_0)}\w^2=-ad\notin\mathcal{P}(\w^2),$$ from which we conclude. 

In is important to emphasize that the theory we build in this paper is much more general than what is strictly necessary for this proof. Equation \eqref{intro_flux_equality} above is merely an instance of a much more general correspondence we display between what we call \textit{tropical fluxes} and \textit{symplectic fluxes}.

Tropical fluxes are morphisms $$\Theta^{trop}:\Gr_k^{trop}(B)_{hom}\longrightarrow \R/\mathcal{P}_k,$$ where $B$ is any tropical torus, $\mathcal{P}_k\subset\R$ are finitely generated groups of periods, and $\Gr_k^{trop}(B)_{hom}$ denotes the homologically trivial part of the tropical Griffiths group of $B$. These morphisms are given by integrals over $(k+1)$-chains bounding the homologically trivial cycles. They are constructed in Section \ref{tropical_flux_section}, using the language of tropical (co)homology.

Similarly, symplectic fluxes are morphisms $$\Theta^{symp}:\algcob^{or}(M,\w)_{hom}\longrightarrow\R/\mathcal{P},$$ where $\algcob^{or}(M,\w)$ is the homologically trivial part of the oriented algebraic Lagrangian cobordism group of $(M,\w)$, and $\mathcal{P}\subset\R$ is a finitely generated group of periods. These are given by integrating $\w^2\wedge\beta$ over $(n+1)$-chains bounding homologically trivial Lagrangians, for any $(n-3)$-form $\beta$ on $M$. This is developed in Section \ref{flux_from_alglagcob}. Cobordism groups $\Cob(M,\w)$ carry a stronger version of these flux morphisms, see Lemma \ref{cob_implies_periods}.

The equality from Equation \eqref{intro_flux_equality} is then interpreted as an instance of a more general correspondence between tropical fluxes in a tropical torus $B$, and symplectic fluxes in the corresponding Lagrangian torus fibration $X(B)$. More concretely: if a nullhomologous tropical $k$-cycle $Z$ admits a Lagrangian lift $L_Z$, this correspondence gives $\Theta^{trop}(Z)=\Theta^{symp}(L_Z)$\footnote{Note that this equality should be made much more precise. In reality, we define \textit{multiple} maps $\Theta^{trop}$ and $\Theta^{symp}$, corresponding to different forms being integrated on the bounding chains. We should really say that given a \textit{particular} tropical flux (i.e., a particular choice of tropical form on $B$), can be equated to a \textit{particular} symplectic flux (i.e., a corresponding choice of $(n-3)$-form $\beta$) on $X(B)$.}. 
 
In these more general terms, our proof of Theorem \ref{intro_main_theorem} simply consists of exhibiting an instance of the following commutative diagram:
\begin{center}
	\begin{tikzcd}
		\algcob^{or}(X(B))_{hom} \arrow[rrdd, "symplectic \hspace{0.5mm}flux"]                           &  &       \\
		&  &       \\
		\Gr_k(B)_{hom} \arrow[uu, "Lagrangian\hspace{1mm}lift", dashed] \arrow[rr, "tropical \hspace{0.2mm}flux"] &  & \R/\mathcal{P},
	\end{tikzcd}
\end{center}
and concluding from Zharkov's computation.

\begin{remark}
	Note that we show (Remark \ref{cob_implies_algcob}) that a Lagrangian cobordism between $L$ and $L'$ gives rise to an algebraic Lagrangian cobordism between $L$ and $L'$ (which is what we could expect from the fact that algebraic equivalence is weaker than rational equivalence). Therefore, an immediate Corollary of Theorem \ref{intro_main_theorem} is the following:
	\begin{theorem}
		For generic tropical curve of type $K4$, the Lagrangian Ceresa cycle $L_C-L_C^-$ has infinite order in $\Cob^{or}(X(J(C)))$. 
	\end{theorem}
\end{remark}

\subsection*{Structure of the paper}

Section \ref{tropical_section} contains two parts. The first introduces the relevant background in tropical geometry before going through the construction of the tropical Ceresa cycle associated to a tropical curve. Special attention is given to the case of the $K4$ curve, which is of main interest to us.

The second part defines tropical flux morphisms using the language of tropical (co)homology. These are given by integrals of tropical forms, which are generalisations of Zharkov's determinantal form. 

In Section \ref{symplectic_flux} we give definitions for the various flavours of Lagrangian cobordism groups. We introduce symplectic flux morphisms from $\Cob^{or}(M,\w)$, and include  an example application to the cotangent bundle of a torus.

Section \ref{algebraic_Lagrangian_cobordisms} introduces the notion of algebraic Lagrangian cobordism.  We then construct symplectic flux maps from algebraic Lagrangian cobordism groups $\algcob(M,\w)$. We end by showing that these fluxes can be encoded as a map to a \textit{Lefschetz Jacobian}, a naturally tropical object which plays the role symplectically of intermediate Jacobians. 

The aim of Section \ref{relating_trop_symp_flux} is to formalise the relation between tropical fluxes in a tropical torus $B$, and symplectic fluxes in the corresponding Lagrangian torus fibration $X(B)$. 
Because both are given by integrating certain differential forms over chains, this involves constructing maps of (co)chains from tropical (co)chains in $B$ to singular (co)chains in $X(B)$ in a way that preserves the integration pairing. The underlying maps on (co)homology are a Künneth-type decomposition for $X(B)\cong B\times T^n$. 

We prove our main Theorem \ref{intro_main_theorem} in Section \ref{Lag_Ceresa_section} by bringing together elements of the previous Sections. This is done after an explicit description of the Lagrangian Ceresa cycle as a $3$-manifold inside the symplectic $6$-torus $X(J(C))$. We also speculate on a higher-dimensional extension of this result. 

We end with two additional remarks in Section \ref{further_remarks}. The first turns to the case of genus $3$ tropical curves of hyperelliptic type, whose Ceresa cycles provide one with a rich source of interesting questions and have inspired recent work \cite{ceresa_period_tropical,Ceresa_class_corey_et_al}. The second relates our work to a construction by Reznikov \cite{Reznikov_characteristic_classes_in_symplectic_topology} building characters of symplectic Torelli-type groups.

\subsection*{Notation}
We abuse notation and simply denote by $M$ a symplectic manifold $(M,\w)$. Similarly, $(M,-\w)$ will be written $\overline{M}$.

\subsection*{Acknowledgements}
	I am greatly endebted to Ivan Smith for his guidance and support, and for carefully reading previous drafts of this manuscript. I also thank an anonymous referee for providing constructive feedback which has improved the exposition of the paper. I was supported by EPSRC grant EP/X030660/1 for the duration of this work.

\section{The tropical Ceresa cycle and tropical flux}\label{tropical_section}

This Section consists of two parts. The first starts by introducing the relevant background in tropical geometry, with a view to constructing the tropical Ceresa cycle associated to a tropical curve. It ends with a detailed study of the case which is of interest to us, namely that of the $K4$ curve.

The second introduces the notion of \textit{tropical flux} using the language tropical (co)homology. These are morphisms from the homologically trivial part of the tropical Griffiths groups given by integrals of a generalised version of Zharkov's determinantal form \cite{zharkov_tropical_ceresa}. They formalise and generalise Zharkov's argument for proving Theorem \ref{intro_Ceresa_thm}, where he exhibits some non-trivial integral over a chain bounding the tropical Ceresa cycle.

\subsection{Tropical background}\label{tropical_background}

\subsubsection{Tropical tori}\label{tropical_tori}

We will be concerned with tropical tori, which are a particularly simple example of tropical manifold.

\begin{definition}\label{def_tropical_manifold}
	Let $B$ be a smooth manifold. A \textit{tropical affine structure} on $B$ is a set of coordinate charts $\{(U_i,\varphi_i)\}$ on $B$ such that for all $i,j$, the transition maps $\varphi_i\circ\varphi_j^{-1}$ lie in $\R^n\rtimes GL(n;\Z)$. That is, they are of the form $x\mapsto Ax+b$ with $A\in GL(n;\Z)$ and $b\in\R^n$. We call $B$ endowed with such a structure a \textit{tropical manifold}.
\end{definition}

\begin{remark}
	The base $B$ of a complete integrable system $M\rightarrow B$ carries a canonical tropical affine structure given by the action-angle coordinates \cite[Section 2.3]{johnnyevans_ltf}.
\end{remark}

From such a structure one can define a subsheaf $\textit{Aff}_B\subset C^\infty(B)$ of the sheaf of smooth functions, consisting of those functions which in any coordinate chart $(U_i,\varphi_i)$ are affine-linear with integer slope. Their differentials form a local system of rank $n$ lattices $T^{*}_\Z B$ inside  $T^{*}B$. Equivalently, there is a short exact sequence of sheaves 
\begin{equation}
	0\longrightarrow \underline{\R}\longrightarrow \textit{Aff}_B\longrightarrow T_\Z^*B\longrightarrow 0.  \notag
\end{equation}
Elements of the group of global sections $H^0(T_\Z^*B)$ are called \textit{tropical $1$-forms}. 
Furthermore, there is a dual lattice $T_\Z B\subset TB$ of those vectors on which $T_\Z^*B$ evaluate to integers, and elements of $H^0(T_\Z B)$ are called \textit{integral vectors}.

In the case where $B:=\R^n/\Gamma$ is a torus, the linear part of the monodromy on the sheaf of affine functions vanishes, and the local systems $T^{(*)}_\Z B$ are trivial. Therefore a tropical torus $B$ is simply the data of two rank $n$ lattices $\Gamma_1$ and $\Gamma_2$ in $\R^n$, such that $B=\R^n/\Gamma_1$ and the tropical structure on $B$ is given by $T_\Z B=\Gamma_2\subset\R^n$. Up to a global coordinate transformation one can take $\Gamma_2$ to be the standard lattice $\Z^n\subset\R^n$, then a tropical torus is determined by a matrix $Q$ in $GL(n;\R)$ such that $B=B(Q):=\R^n/Q\cdot\Z^n$ (note that left or right multiplication of $Q$ by elements of $GL(n;\Z)$ yield isomorphic tropical tori, see e.g. \cite[Lemma 2.4]{nick_ivan1}). 

A \textit{polarisation} on a tropical torus $B$ is a class in $H^1(T^*_\Z B)$ in the image of the tropical Chern class map $$H^1(\textit{Aff}_B)\longrightarrow H^1(T_\Z^*B).$$ From the isomorphism $H^1(T^*_\Z B)\cong \Gamma_1^*\otimes H^0(T^*_\Z B)\cong \hom(H_1(B;\Z),H^0(T_\Z^*B))$ this can be shown \cite[Section 5.1]{Mikhalkin_Zharkov_trop_curves_jac_theta_funct} to be equivalent to a map $$c:H_1(B;\Z)\longrightarrow H^0(T_\Z^*B)$$ such that the induced pairing on $H_1(B;\Z)$ given by $\int_{\gamma_i}c(\gamma_j)$ is positive-definite and symmetric. A polarisation is \textit{principal} if this map is an isomorphism.

\subsubsection{Tropical curves, their Jacobian and their Ceresa cycle}\label{tropical_curves_Jac_Ceresa}

We summarize the tropical version of the classical construction of the Jacobian of a curve and its Abel-Jacobi embedding, and give the definition of the tropical Ceresa cycle. 

\begin{definition}
	A \textit{tropical curve} consists of a finite connected graph $G$ together with a positive function $l:E(G)\longrightarrow \R_{\geq0}$ on its edge set $E(G)$. Given $e\in E(G)$, $l(e)$ is called the \textit{length} of $e$. 
\end{definition}

\begin{remark}
	As such, $C$ is not quite a tropical manifold of dimension $1$ as in Definition \ref{def_tropical_manifold}, but falls into the more general notion of \textit{tropical space} \cite[Definition 7.1.8]{tropical_book}. The two can be shown to be equivalent \cite[Section 8.1]{tropical_book} if one additionally requires $C$ to be smooth, regular at infinity, and of finite type \cite[Definition 7.4.1]{tropical_book}. For our purposes these assumptions are not restrictive, and we extend our constructions on tropical manifolds to tropical curves, where the tropical structure on the latter can be obtained by embedding edges of $C$ in $\R$ via metric-preserving charts (for more details, consult e.g. \cite[Proposition 3.6]{Mikhalkin_Zharkov_trop_curves_jac_theta_funct}).
\end{remark}

If the underlying graph of $C$ has genus $g$, then $H^0(T_\Z^*C)$ is a free abelian group of rank $g$; in particular $H^0(T_\R^*C):=H^0(T_\Z^*C)\otimes\R\cong\R^g$. We define the tropical Jacobian of $C$ as the torus
\begin{align}
	J(C):=H^0(T_\R^*C)^*/H_1(C;\Z),	\notag
\end{align}
where $H_1(C;\Z)$ is identified with a lattice inside $H^0(T_\R^*C)^*$ by setting $\gamma(\alpha):=\int_\gamma\alpha$ for any $\alpha\in H^0(T_\R^*C)$.

Given a choice of basepoint $p\in C$, the tropical Abel-Jacobi map
\begin{align}
	AJ_p:\hspace{2mm}& C\longrightarrow J(C) \notag \\
	& q\mapsto \left(\alpha\in\Omega^1(C)\mapsto \int_{\gamma:p\rightarrow q}\alpha\right) \notag
\end{align}
embeds $C$ into $J(C)$, where $\gamma:p\rightarrow q$ is any path in $C$ between $p$ and $q$. 

The Jacobian torus $J(C)$ carries both a canonical tropical structure given by the lattice $\Gamma_2\subset H^0(T_\R^*C)^*$ of those maps which evaluate to integers on global sections of $T_\Z C$, and a canonical principal polarisation which we describe following \cite[Section 6.1]{Mikhalkin_Zharkov_trop_curves_jac_theta_funct}. The bilinear form on the space of paths in $C$ defined by $$Q(E,E):=l(E)$$ on any edge $E$ of $C$ can be extended bilinearly. This induces a symmetric, positive-definite bilinear form on $H_1(C;\Z)\cong H_1(J(C);\Z)$, therefore a polarisation $$c(C):H_1(J(C);\Z)\longrightarrow H^0(T^*_\Z J(C))$$ by $\int_{\gamma_i}c(\gamma_j):=Q(\gamma_i,\gamma_j)$.

One can show \cite[Lemma 6.3]{Mikhalkin_Zharkov_trop_curves_jac_theta_funct} that the tropical structure on $J(C)$ makes $AJ_p$ into a tropical map, that is $(AJ_p)^*\textit{Aff}_{J(C)}=\textit{Aff}_C$. In the language of tropical Chow groups \cite{trop_intersection}, this implies that $AJ_p(C)$ is a \textit{tropical $1$-cycle} inside $J(C)$ which we write $C\in CH_1(J(C))$.

Furthermore, $J(C)$ inherits a group structure from the one on $H^0(T_\R^*C)^*$, hence carries a natural involution $$(-1):J(C)\longrightarrow J(C)$$ given by inverse for the group law. The \textit{Ceresa cycle} of $C$ is
\begin{equation}
	AJ_p(C)-(-1)_*(AJ_p(C))\in CH_1(J(C)). \notag
\end{equation}
We will often abuse notation and write $C$ instead of $AJ_p(C)$, $C^-$ instead of $(-1)_*(AJ_p(C))$, and the Ceresa cycle simply as $C-C^-$.

\subsubsection{An important example: the genus $3$ curve of type $K4$}\label{important_example}

Here we explore the case where $C$ is a genus $3$ curve of type $K4$ represented in Figure \ref{K4}; which is of primary interest to us. It is one of five types of generic genus $3$ tropical curves, and distinguishes itself by being the only \textit{non-hyperelliptic} type. A tropical curve is of \textit{hyperelliptic type} if its Jacobian is isomorphic to the Jacobian of a hyperelliptic tropical curve. 

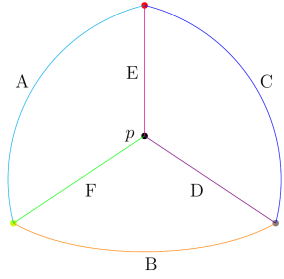
\begin{figure}
	\centering
	\scalebox{0.6}{
		\begin{tikzpicture}
			\pgfdeclarelayer{nodelayer} 
			\pgfdeclarelayer{edgelayer}
			\pgfsetlayers{main,nodelayer,edgelayer}
			\begin{pgfonlayer}{nodelayer}
				\node [fill,circle,label={above right:$p$}] (0) at (0, 0) {};
				\node [fill,circle,color=red] (1) at (0, 3.75) {};
				\node [fill,circle,color=lime] (2) at (-3.5, -2) {};
				\node [fill,circle,color=gray] (3) at (3.5, -2) {};
				
				\node[label={left: E}] (4) at (0.1,1.875) {};
				\node[label={above:F}] (5) at (-1.75,-1.7) {};
				\node[label={above:D}] (6)  at (1.75,-1.7) {};
				\node[label={below left:A}] (7) at (-3.0,1.75) {};
				\node[label={below left:C}] (8) at (3.8,1.75) {};
				\node[label={above:B}] (9) at (0,-4.15) {};
			\end{pgfonlayer}
			\begin{pgfonlayer}{edgelayer}
				\draw[very thick,color=magenta] (1.center) to (0.center);
				\draw[very thick,color=green] (0.center) to (2.center);
				\draw[very thick,color=violet](0.center) to (3.center);
				\draw [very thick, color=cyan,bend right=45] (1.center) to (2.center);
				\draw [very thick, color=blue, bend left=45] (1.center) to (3.center);
				\draw [very thick, color=orange,bend right=45] (2.center) to (3.center);
			\end{pgfonlayer}
	\end{tikzpicture}}
	\caption{A tropical curve of type $K4$.}
	\label{K4}
\end{figure}

When a curve is hyperelliptic, its Ceresa cycle is trivial: this can be seen by embedding it in its Jacobian using as a basepoint a Weierstrass point $p\in C$. Then one finds that $AJ_p(C)=(-1)_*(AJ_p(C))$. However, as soon as $g(C)\geq 3$, the Ceresa cycle (whether in its classical or tropical version) has been the object of extensive study for being an example of a cycle which is homologically, but not algebraically, trivial. The following result is originally due to Ceresa in the classical setting \cite{Ceresa}, but its tropical version was proved by Zharkov:
\begin{theorem}\cite[Theorem 3]{zharkov_tropical_ceresa}\label{Zharkov_theorem}
	Let $C$ be a generic genus $3$ tropical curve of type $K4$. Then $C$ is not algebraically equivalent to $C^-$ in its Jacobian $J(C)$.
\end{theorem}
The fact that Zharkov's argument does not apply to other non-hyperelliptic curves which are of hyperelliptic type, can be understood as a failure of the tropical Torelli Theorem; this is discussed in Section \ref{hyperelliptic_types_section}. 

\begin{remark}\label{Ceresa_nullhomologous}
	Theorem \ref{Zharkov_theorem} is only non-trivial because the Ceresa cycle is nullhomologous. In the classical case,  this follows from the fact that $(-1)$ acts trivially on $H_j(J(C);\Z)$ for $j$ even. 
	Here, it follows immediately from the fact that Zharkov constructs an explicit $2$-chain $\gamma_0$ with $\partial\gamma_0=C-C^-$. 
\end{remark}
\begin{remark}
	The stronger statement that $C-C^-$ has \textit{infinite order} in $CH_1(J(C))$ modulo algebraic equivalence follows from Zharkov's proof, and is proved by Ceresa in the classical setting \cite{Ceresa}. We address this again in Remark \ref{remark_torsion_case}.
\end{remark}

Let us now study this curve and its Jacobian more explicitly. The basis $\{\gamma_1,\gamma_2,\gamma_3\}$ of $H_1(C;\Z)=\Gamma_1$ is represented in Figure \ref{K4_basis_oh}. Recall that $\Gamma_2\subset H^0(T^*_\R C)^*\cong\R^3$ is the lattice of those maps which evaluate to integers on elements of $H^0(T_\Z^*C)$. This determines the tropical structure of $J(C)$. A basis $\{e_1,e_2,e_3\}$ of $\Gamma_2$ is also represented by the vectors in Figure \ref{K4_basis_oh}.
These two bases are related through
\begin{align}
	&\gamma_1=(a+f+e,-f,-e)& &\gamma_2=(-f,b+d+f,-d)& &\gamma_3=(-e,-d,c+d+e), \notag
\end{align}
where lowercase letters denote the length of the edge labelled by the corresponding uppercase letter. 

\begin{figure}[htb]
	\centering
	\includegraphics[width=2in]{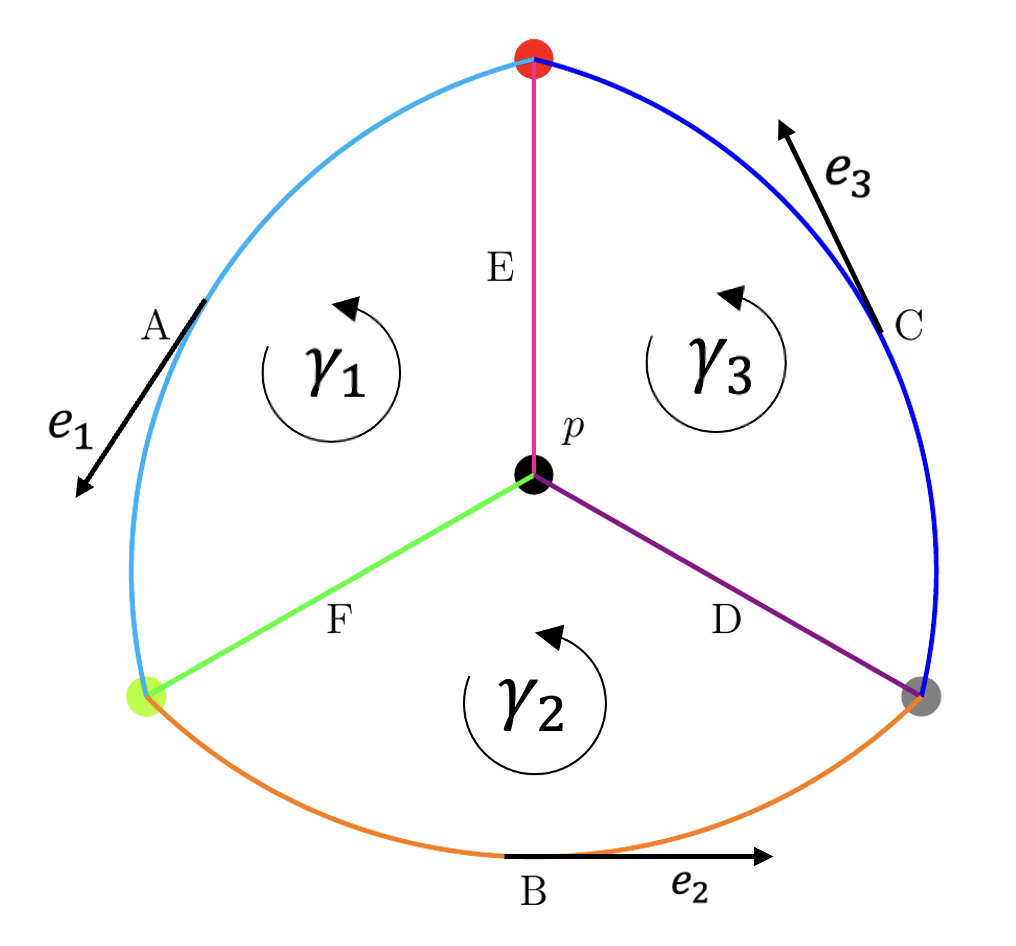} 
	\caption{Bases for $\Gamma_1$ and $\Gamma_2$ represented in $C$.}
	\label{K4_basis_oh}
\end{figure}

In the colour-coded figure \ref{K4_in_Jacobian}, we represent the image of $C$ in its Jacobian by the Abel-Jacobi map $AJ_p$, with the arbitrary choices of edge lengths
\begin{align}
	&a=1&&b=0.5&&c=1.5&&d=1&&e=0.5&&f=1.& \notag
\end{align}

\begin{figure}
	\centering
	\begin{tikzpicture}[x=1.8cm,y=1.8cm,z=1.0cm,>=stealth]
		
		% defining edge lengths
		\tikzmath{\x1=1;\x2=0.5;\x3=1.5;\x4=1;\x5=0.5;\x6=1;}
		
		% setting the axes for visualisation
		
		\draw[->](-3.4,-2,1.2)--(-2.4,-2,1.2); % e1		
		\draw[->](-3.4,-2,1.2)--(-3.4,-1,1.2); % e2
		\draw[->](-3.4,-2,1.2)--(-3.4,-2,2.2); % e3
		\node[label={below:$e_1$}] at (-2.4,-2,1.2) {};
		\node[label={left:$e_2$}] at (-3.4,-1,1.2) {};
		\node[label={above:$e_3$}] at (-3.4,-2,2.2) {};

		%% The edges of the cubes
		% the edges along \gamma_1
		\draw[] (0,0,0)--node[midway]{$\blacktriangle$} (\x1+\x6+\x5,-\x6,-\x5);
		\draw[] (-\x6,\x2+\x4+\x6,-\x4)--node[midway]{$\blacktriangle$} (\x1+\x5,\x2+\x4,-\x5-\x4);
		\draw[] (-\x5,-\x4,\x3+\x4+\x5)--node[midway]{$\blacktriangle$} (\x1+\x6,-\x4-\x6,\x3+\x4);
		\draw[] (-\x6-\x5,\x2+\x6,\x3+\x5)--node[midway]{$\blacktriangle$} (\x1,\x2,\x3);
		% the edges along \gamma_2
		\draw[] (0,0,0)--node[midway]{$\bigstar$} (-\x6,\x2+\x4+\x6,-\x4);
		\draw[] (\x1+\x6+\x5,-\x6,-\x5)--node[midway]{$\bigstar$}(\x1+\x5,\x2+\x4,-\x5-\x4) ;
		\draw[] (-\x5,-\x4,\x3+\x4+\x5)--node[midway]{$\bigstar$} (-\x5-\x6,\x2+\x6,\x3+\x5);
		\draw[] (\x1+\x6,-\x4-\x6,\x3+\x4)--node[midway]{$\bigstar$} (\x1,\x2,\x3);
		% the edges along \gamma_3
		\draw[] (0,0,0)--node[midway]{$\blacksquare$} (-\x5,-\x4,\x3+\x4+\x5);
		\draw[] (\x1+\x6+\x5,-\x6,-\x5)--node[midway]{$\blacksquare$} (\x1+\x6,-\x4-\x6,\x3+\x4);
		\draw[] (-\x6,\x2+\x4+\x6,-\x4)--node[midway]{$\blacksquare$} (-\x5-\x6,\x2+\x6,\x3+\x5);
		\draw[] (\x1+\x5,\x2+\x4,-\x5-\x4)--node[midway]{$\blacksquare$} (\x1,\x2,\x3);
		
		%% now for a second
		\draw[] (\x1+\x5+\x6,-\x6,-\x5)--node[midway]{$\blacksquare$} (\x1+2*\x5+\x6,-\x6+\x4,-\x3-\x4-2*\x5); %from \gamma_1 translated by -\gamma_3
		\draw[] (0,0,0)--node[midway]{$\blacksquare$} (\x5,\x4,-\x3-\x4-\x5); %from (0,0,0) translated by -\gamma_3
		\draw[] (-\x6,\x2+\x4+\x6,-\x4)--node[midway]{$\blacksquare$} (-\x6+\x5,\x2+2*\x4+\x6,-\x3-2*\x4-\x5); %from \gamma_2 translated by -\gamma_3
		\draw[] (\x1+\x5,\x2+\x4,-\x5-\x4)--node[midway]{$\blacksquare$} (\x1+2*\x5,\x2+2*\x4,-\x3-2*\x5-2*\x4); %from \gamma_2+\gamma_1 translated by -\gamma_3
		\draw[] (\x1+2*\x5,\x2+2*\x4,-\x3-2*\x5-2*\x4)--node[midway]{$\blacktriangle$} (-\x6+\x5,\x2+2*\x4+\x6,-\x3-2*\x4-\x5);
		\draw[] (\x1+2*\x5,\x2+2*\x4,-\x3-2*\x5-2*\x4)--node[midway]{$\bigstar$} (\x1+2*\x5+\x6,-\x6+\x4,-\x3-\x4-2*\x5);
		\draw[] (\x5,\x4,-\x3-\x4-\x5)--node[midway]{$\blacktriangle$} (\x1+2*\x5+\x6,-\x6+\x4,-\x3-\x4-2*\x5);
		\draw[] (\x5,\x4,-\x3-\x4-\x5)--node[midway]{$\bigstar$} (-\x6+\x5,\x2+2*\x4+\x6,-\x3-2*\x4-\x5);
		
		%black node
		\node[fill,circle,inner sep=1.5pt,label={below left:$p$}] at (0,0,0) {};
		\node[fill,circle,inner sep=1.5pt,label={above:$\gamma_2$}] at (-\x6,\x2+\x4+\x6,-\x4){};
		\node[fill,circle,inner sep=1.5pt,label={above:$\gamma_1$}] at (\x1+\x5+\x6,-\x6,-\x5){};
		\node[fill,circle,inner sep=1.5pt,label={above:$\gamma_3$}] at (-\x5,-\x4,\x3+\x4+\x5){};
		\node[fill,circle,inner sep=1.5pt,label] at (-\x5-\x6,\x2+\x6,\x3+\x5){}; %\gamma_2+\gamma_3
		\node[fill,circle,inner sep=1.5pt,label] at (\x1+\x5,\x2+\x4,-\x5-\x4){}; %\gamma_1+\gamma_2
		\node[fill,circle,inner sep=1.5pt,label] at (\x1+\x6,-\x4-\x6,\x3+\x4){}; %\gamma_1+\gamma_3
		\node[fill,circle,inner sep=1.5pt,label] at (\x1,\x2,\x3){}; %\gamma_1+\gamma_2+\gamma_3
		\node[fill,circle,inner sep=1.5pt,label={above:$-\gamma_3$}] at 	(\x5,\x4,-\x3-\x4-\x5){};
		\node[fill,circle,inner sep=1.5pt,label] at (-\x6+\x5,\x2+2*\x4+\x6,-\x3-2*\x4-\x5){};
		\node[fill,circle,inner sep=1.5pt,label] at (\x1+2*\x5+\x6,-\x6+\x4,-\x3-\x4-2*\x5){};
		\node[fill,circle,inner sep=1.5pt,label] at (\x1+2*\x5,\x2+2*\x4,-\x3-2*\x5-2*\x4){}; %\gamma_1+\gamma_2-\gamma_3
		
		%	lines for edges of curve
		\draw[very thick,color=magenta] (0,0,0)--(\x5,0,-\x5);
		\draw[very thick,color=magenta] (0,-\x4,\x4+\x3)--(-\x5,-\x4,\x4+\x3+\x5);
		\draw[very thick,color=orange] (-\x6,\x6,0)--(-\x6,\x6+\x2,0);			
		\draw[very thick,color=violet] (0,0,0)--(0,-\x4,\x4);		
		\draw[very thick,color=violet] (-\x6,\x6+\x2,0)--(-\x6,\x6+\x2+\x4,-\x4);			
		\draw[very thick,color=blue] (0,-\x4,\x4)--(0,-\x4,\x4+\x3);
		\draw[very thick,color=cyan](\x5,0,-\x5)--(\x5+\x1,0,-\x5);
		\draw[very thick,color=green] (0,0,0)--(-\x6,\x6,0);
		\draw[very thick,color=green] (\x5+\x1,0,-\x5)--(\x5+\x1+\x6,-\x6,-\x5);
	\end{tikzpicture}	
	\caption{A representation of the embedding of the type $K4$ tropical curve in its Jacobian, with the edges colour-coded as in Figure \ref{K4_basis}, and the two bases $\{\gamma_1,\gamma_2,\gamma_3\}$ and $\{e_1,e_2,e_3\}$ represented as well. The two black parallelepipeds are two copies of $J(C)$ inside $\Gamma_2\otimes\R\cong\R^3$, and edges with the same symbols $\blacksquare,\blacktriangle$, or $\bigstar$ are identified.}
	\label{K4_in_Jacobian}
\end{figure}

\subsection{Tropical flux}\label{tropical_flux_section}

Recall from the Introduction the main ideas of Zharkov's proof of Theorem \ref{Zharkov_theorem}. He starts by introducing a $2$-form $\Omega_0$ on $J(C)$ called \textit{determinantal form}. He then shows that, if two $1$-cycles are algebraically equivalent (in particular they are homologically equivalent), then integrating $\Omega_0$ over any $2$-chain realising this homological equivalence must yield a period of $\Omega_0$. He proceeds to constructing an explicit $2$-chain $\gamma_0$ in $J(C)$ with $\partial\gamma_0=C-C^-$, computes $\int_{\gamma_0}\Omega_0$, and shows it is \textit{not} a period of $\Omega_0$ for generic $C$. 
The aim of the current Section is to formalise and generalise this idea of using non-trivial integrals to detect obstructions to algebraic equivalence. The result are morphisms $$\Theta_k:\Gr_k(B)\longrightarrow \R/\mathcal{P}_k,$$ where $\Gr_k(B)$ is the $k$-th Griffiths group of $B$ of homologically trivial $k$-cycles modulo algebraic equivalence, and $\mathcal{P}_k$ is a corresponding set of periods. This involves generalising the determinantal form $\Omega_0$, and the construction of certain $(k+1)$-chains bounding tropical $k$-cycles. The objects involved are \textit{tropical} chains and cochains; their coefficients encode information about the tropical structure of the underlying manifold. 

Later, in Section \ref{relating_trop_symp_flux}, we will show that for a torus $B$, these tropical fluxes between tropical cycles are related to symplectic fluxes between their Lagrangian lifts in the symplectic manifold $X(B):=T^*B/T^*_\Z B$ in the sense of Section \ref{symplectic_flux}. 

\subsubsection{Tropical (co)homology}\label{tropical_(co)homology}

Tropical (co)homology was originally introduced in \cite{tropical_homology_ikmz}, with the idea of encoding information about the (co)homology of an algebraic variety in the tropical (co)homology of its tropicalization. This is done by enriching singular (co)chains with coefficients in the exterior powers of the integral lattices $T_\Z^{(*)}B$ of a tropical manifold $B$, as a way of recording data from the tropical structure. In our case $B$ is a tropical torus, and we have seen that this means the local systems of lattices $T_\Z^{(*)}B$ are constant, therefore these are global coefficients.

For any $k\geq0$, we consider the chain complexes $C_j(B;\Lambda^k T_\Z B),$ whose objects are $j$-simplices $\sigma:\Delta\rightarrow B$ with coefficients in $\Lambda^kT_\Z B\cong \Lambda^k\Gamma_2$.
The boundary map is given by restricting coefficients to the usual boundary map on singular $j$-chains. The resulting homology groups are denoted by $H_j(B;\Lambda^k T_\Z B)$. Similarly, we define the cohomology groups $H^j(B;\Lambda^kT^*_\Z B)$ for any $k$, by taking the cohomology of the dual complex $C^j(B;\Lambda^kT_\Z^*B)$. 
The universal coefficient theorem yields canonical isomorphisms:
\begin{align}
	H_j(B;\Lambda^k T_\Z B)\cong \Lambda^j\Gamma_1\otimes\Lambda^k \Gamma_2 \notag \\
	H^j(B;\Lambda^kT_\Z^*B)\cong\Lambda^j\Gamma^*_1\otimes\Lambda^k\Gamma^*_2. \notag
\end{align}

\subsubsection{Tropical chains from algebraic equivalences}\label{construction_tropical_chain_equivalence}

This section does not contain any original ideas, but extends constructions of Zharkov \cite[Section 3.1]{zharkov_tropical_ceresa} from the case of $1$-cycles to that of arbitrary $k$-cycles. We borrow some of the notation and terminology for the sake of consistency. Note that here $B$ could be any tropical manifold, not just a tropical torus.

Our first aim is to translate algebraic equivalence of cycles into the language of tropical (co)chains, in order to obtain tropical (co)homological obstructions to two cycles being algebraic equivalent. This requires two steps. The first is to associate to a tropical cycle a \textit{tautological tropical cycle}, which is a cycle in the tropical cohomology (i.e., it has coefficients in exterior powers of $T^*_\Z B$). The second is, given an algebraic equivalence between two cycles, to construct a tropical chain realising the tropical homological equivalence between the corresponding tautological tropical cycles.

\begin{definition}
	A \textit{tropical $k$-cycle} in $B$ is a weighted balanced $k$-dimensional polyhedral complex with $\Gamma_2$-rational slopes.
\end{definition}

Two such cycles $Z_1$ and $Z_2$ are \textit{algebraically equivalent} in $B$ if there exists two points $c_1$ and $c_2$ in some tropical curve $C$, and a tropical $(k+1)$-cycle $W$ in $B\times C$ satisfying $$\pi_*\left(W\cap(B\times\{c_1\})-W\cap(B\times\{c_2\})\right)=Z_1-Z_2,$$ where $\pi:B\times C\rightarrow B$ is the projection onto the first factor, and $\cap$ denotes tropical intersection (also called stable interesection, \cite[Section 4.3]{tropical_book}). This defines an equivalence relation. 

To any tropical $k$-cycle $Y$, one can associate a canonical element of $C_k(B;\Lambda^kT_\Z B)$ called a \textit{tautological tropical cycle} $\overline{Y}$ as follows. Every oriented $k$-face $f$ of $Y$ with weight $w_f$ defines a $k$-cell $w_f\sigma_f$, as well as $k$ primitive tangent vectors $v_1,\dots,v_k$ up to a transformation in  $SL(n;\Z)$. The data of an orientation on $f$ yields a preferred element $\beta(f)=v_{i_1}\wedge\dots\wedge v_{i_k}$ in $\Lambda^kT_\Z B$. Then the tautological cycle associated to $Y$ is $$\overline{Y}=\sum_{f\in Y}\beta(f)w_f\cdot \sigma_f.$$ 
Reversing the orientation on $f$ has the effect of taking $\beta(f)$ to $-\beta(f)$ as well as $\sigma_f$ to $-\sigma_f$, and these cancel out in $C_j(B;\Lambda^k T_\Z B)$, making $\overline{Y}$ well-defined. The fact that $\overline{Y}$ is closed follows from the fact that $Y$ satisfies a balancing condition: an $(n-1)$-face in the boundary of $s$ distinct $n$-faces $f_1,\dots,f_s$ will appear $s$ times with coefficients $\beta(f_1)\w_{f_1},\dots,\beta(f_s)\w_{f_s}$, whose sum vanishes by the balancing condition. 

The existence of an algebraic equivalence $W\subset B\times C$ between two $k$-cycles $Y_1$ and $Y_2$ in $C$ implies that $[\overline{Y_1}]=[\overline{Y_2}]$ in $H_k(B;\Lambda^{k}T_\Z B)$. Still following \cite[Section 3.1]{zharkov_tropical_ceresa}, we construct some $\Tilde{W}$ in $C_{k+1}(B;\Lambda^kT_\Z B)$ satisfying $$\partial\Tilde{W}=\overline{Y_1}-\overline{Y_2}.$$
We proceed by choosing a path $P$ in $C$ from $c_1$ to $c_2$, and letting $$Y_c:=\pi_*(W\cap(B\times\{c\}))$$ for any $c\in C$. Now define the weighted polyhedral complex $W_P\subset B\times C$ as the restriction of $W$ to $(\pi')^{-1}(P)$, where $\pi':B\times C\longrightarrow C$ is the projection onto the second factor. For any $(k+1)$-cell $\sigma$ in the support of $W_P$, two scenarios arise. Either $\pi'$ is not submersive on $\sigma$, in which case we set $\beta(\sigma)=0$, or it is submersive on $\sigma$. In this latter case, above each $c\in\pi'(\sigma)$, $\sigma\cap (Y_c\times\{c\})$ carries a tautological framing $\beta_c(\sigma)\in\Lambda^kT_\Z (B\times\{c\})$ as decribed above, which can be pulled back by $\pi$ to $\Lambda^kT_\Z (B\times C)$. Clearly, on a fixed $(k+1)$-cell $\sigma$, this is independent of $c$ because both $\pi$ and $\pi'$ are linear; denote this by $\beta(\sigma)$. We can now define $$\Tilde{W}:=\sum_{\sigma\in W_P}\beta(\sigma)w_\sigma\cdot\pi(\sigma)\in C_{k+1}(B;\Lambda^kT_\Z B).$$

\begin{lemma}\cite[Lemma 4]{zharkov_tropical_ceresa}
	Let $W$ be an algebraic equivalence between two algebraic $k$-cycles $Y_1$ and $Y_2$. Then the tropical $(k+1)$-chain $\Tilde{W}$ defined above satisfies $$\partial\Tilde{W}=\overline{Y_1}-\overline{Y_2}.$$
\end{lemma}

\subsubsection{Determinantal forms}\label{determinantal_forms_section}

The notion of a determinantal form on a tropical torus $B$ was originally introduced by Zharkov \cite[Section 3.2]{zharkov_tropical_ceresa} as a cocycle in $C^2(B;T^*_\Z B)$. We extend this construction to higher dimensions. Although in practice we will only be interested in determinantal forms in $C^{k+1}(B;\Lambda^kT^*B)$, we construct them in any $C^j(B;\Lambda^kT^*B)$. 
In theory, one could expect a version of these determinantal forms to make sense in any tropical manifold, however our construction is for tropical tori.

Let $e_1,\dots,e_n$ be a basis of $\Gamma_2$, and $x_1,\dots,x_n$ an associated choice of coordinates (well-defined up to translation) on $B$.

\begin{definition}\label{determinantal_forms}
	On a tropical $n$ torus, we define for any $j,k\leq n$ the \textit{determinantal form} $$\Omega^n_{j,k}:=\sum_{(i_1<\dots<i_{j+k})\in\{1,\dots,n\}^{j+k}\atop(j_1<\dots<j_k)\in\{i_1,\dots,i_{j+k}\}^{k}  } (-1)^{sgn(\sigma)}e^*_{j_1}\wedge\dots\wedge e^*_{j_k}\otimes dx_{s_1}\wedge\dots\wedge dx_{s_j},$$ where $s_1<\dots<s_j$ runs over $\{i_1,\dots,i_{j+k}\}\backslash\{j_1,\dots,j_k\}$, and $\sigma$ is the permutation taking $\{i_1,\dots,i_{j+k}\}$ to $\{j_1,\dots,j_k,s_1,\dots,s_j\}$.
\end{definition}

\begin{remark}
	Taking $n=3$, $k=1$ and $j=k+1=2$, we recover Zharkov's original determinantal form \cite{zharkov_tropical_ceresa} $$\Omega^3_{1,2}=e_1^*\otimes dx_2\wedge dx_3-e_2^*\otimes dx_1\wedge dx_3+ e_3^*\otimes dx_1\wedge dx_2.$$
\end{remark}

The period lattices of the forms $\Omega^n_{k,j}$ are nicely packaged in the matrix $Q\in GL(n;\R)$ determining the tropical torus $B=B(Q)$. We denote them by $\mathcal{P}^n_{k,j}\subset\R$.

\begin{proposition}\label{periods_computation}
	The periods $\mathcal{P}^n_{j,k} $of the determinantal form $\Omega^n_{j,k}$ on a tropical torus $B(Q)$ are given by integer combinations of the real numbers 
	\begin{align}
		\sum_{(s_1<\dots<s_j)\in(\{1,\dots,n\}\backslash\{i_1,\dots,i_k\})^j} (-1)^{sgn(\sigma)}Q_{(s_1,\dots,s_j),(r_1,\dots,r_j)} \notag 
	\end{align}
	for any tuples $(i_1<\dots<i_k)\in\{1,\dots,n\}^k$ and $(r_1<\dots<r_j)\in\{1,\dots,n\}^j$, where $\sigma$ is the permutation taking $(i_1,\dots,i_k,s_1,\dots,s_j)$ to its corresponding ordered $(k+j)$-tuple, and $Q_{(s_1,\dots,s_j),(r_1,\dots,r_j)}$ is the $j\times j$ minor of the $n\times n$ matrix $Q$ obtained by conserving only the rows $\{s_1,\dots,s_j\}$ and the columns $\{r_1,\dots,r_j\}$.
\end{proposition}
\begin{proof}
	The $\Z$-generators correspond to integrals of $\Omega^n_{k,j}$ over the homology classes $$e_{i_1}\wedge\dots\wedge e_{i_k}\otimes \gamma_{r_1}\wedge\dots\wedge \gamma_{r_j}\in H_j(B;\Lambda^kT_\Z B)\cong \Lambda^k\Gamma_2\otimes\Lambda^j\Gamma_1.$$
	The pairing of each term of the form $$e_{j_1}^*\wedge\dots\wedge e_{j_k}^*\otimes dx_{s_1}\wedge\dots\wedge dx_{s_j}$$ in the determinantal form with such a class is non-vanishing if and only if $(j_1,\dots,j_k)=(i_1,\dots,i_k)$, and in this case it is equal to 
	\begin{align}
		dx_{s_1}\wedge\dots\wedge dx_{s_j}(\gamma_{r_1}\wedge\dots\wedge \gamma_{r_j})=dx_{s_1}\wedge\dots\wedge dx_{s_j}((Qe_{r_1})\wedge\dots\wedge(Qe_{r_j})). \notag
	\end{align}
	This is given by $$Q_{(s_1,\dots,s_j),(r_1,\dots,r_j)}.$$ Therefore the periods of the determinantal forms are generated by appropriate linear combinations of the minors of $Q$. 
\end{proof}

\begin{example}
	In the case of the non-hyperelliptic genus $3$ curve from Section \ref{important_example}, these periods of $\Omega^3_{2,1}$ were computed by Zharkov \cite[Section 3.2]{zharkov_tropical_ceresa}. We use the bases represented in Figure \ref{K4_basis}. In particular $H_2(J(C);T_\Z J(C))\cong \Gamma_2\otimes \Lambda^2\Gamma_1$ has $9$ generators
	\begin{align}
		&e_1\otimes \gamma_1\wedge\gamma_2& 
		&e_2\otimes \gamma_1\wedge\gamma_2&
		&e_3\otimes \gamma_1\wedge\gamma_2 \notag \\
		&e_1\otimes \gamma_1\wedge\gamma_3& 
		&e_2\otimes \gamma_1\wedge\gamma_3& 
		&e_3\otimes \gamma_1\wedge\gamma_3 \notag \\
		&e_1\otimes \gamma_2\wedge\gamma_3& &e_2\otimes\gamma_2\wedge\gamma_3& 
		&e_3\otimes \gamma_2\wedge\gamma_3, \notag
	\end{align}
	and the matrix $Q$ is given by
	\begin{equation}
		Q=	\begin{pmatrix}
			a+e+f & -f & -e \\ -f & b+d+f & -d \\ -e & -d & c+d+e
		\end{pmatrix}.\notag
	\end{equation}
	Using Proposition \ref{periods_computation}, these yield $6$ different periods, corresponding to the different minors of $Q$:
	\begin{align}
		&ab+ad+af+be+de+ef+bf+df& &ad+ed+df+ef
		\notag \\
		&ac+ad+ae+ec+ed+cf+df+ef& &df+be+de+ef \notag \\ 
		&bc+bd+be+cd+de+cf+df+ef& &cf+df+ef+de. \notag 
	\end{align}
	This coincides with those obtained by Zharkov.    
\end{example}

\begin{remark}\label{periods_jacobians_polarisation_form}
	In the case where $B=J(C)$ is a tropical Jacobian, the periods can be expressed in terms of the canonical polarisation of $B$ for the following reason. In the period computation above, we have used the fact that the map \begin{align}
		\phi_Q:\hspace{1mm}&\Z^n\longrightarrow \R^n \notag \\
		&x\mapsto Q\cdot x \notag
	\end{align}
	viewed as a map from $H_1(B;\Z)$ to $\Gamma_1\otimes\R\cong H^0(T_\R B(Q))$ is precisely the integration pairing, that is for any $\eta\in H^0(T_\R^*B(Q))$ and $\gamma\in H_1(B(Q);\Z)$, $$\int_\gamma\eta=\eta(\phi_Q(\gamma)).$$ Integration over $H_k(B(Q);\Z)$ is obtained from exterior powers $\phi_Q$. In the case of a Jacobian $B=J(C)$, the image of this map is contained in the lattice $H^0(T_\Z J(C))\subset H^0(T_\R J(C))$.
	Very generally, as described in \cite[Lemma 3.9]{nick_ivan1}, any polarisation $$c:H_1(B;\Z)\longrightarrow H^0(T_\Z^*B)$$ on a tropical torus can be related to its integration pairing $$\phi:H_1(B;\Z)\longrightarrow H^0(T_\R B)$$ through an isomorphism $T_\R B\cong T_\R^*B$ whose image lives in $T_\Z^*B(Q)$. In the case of the canonical principal polarisation on a Jacobian $J(C)$, one can verify that this isomorphism is simply the canonical duality isomorphism $T_\Z J(C)\cong T_\Z^*J(C)$. 
\end{remark}

\subsubsection{Tropical flux}\label{tropical_flux}

Recall that the $k$-th tropical Griffiths group $\Gr_k(B)$ is defined as the set of homologically trivial tropical $k$-cycles in $B$, modulo the equivalence relation generated by algebraic equivalence.

If $dim(B)=n$, we define for all $k\geq1$ the \textit{tropical flux maps} 
\begin{align}
	\Theta_k:\hspace{1mm}\Gr&_k(B)\longrightarrow \R/\mathcal{P}^n_{k+1,k}\notag \\
	& \hspace{1mm}Z\hspace{5mm} \mapsto\hspace{2mm} \int_\gamma\Omega^n_{k+1,k}, \notag
\end{align}
where $\gamma\in C_{k+1}(B;\Lambda^kT_\Z B)$ satisfies $\partial\gamma=\overline{Z}$, and $\mathcal{P}^n_{k+1,k}$ is the period lattice of $\Omega^n_{k,k+1}$ as computed in Proposition \ref{periods_computation}.

The following Proposition and its subsequent Corollary \ref{determinantal_form_vanishes} constitute the proof that this map is well-defined: while it is clear that a different choice of $(k+1)$-chain $\gamma$ would only change $\int_\gamma\Omega^n_{k+1,k}$ by a period of $\Omega^n_{k+1,k}$, it remains to show that for $Z$ algebraically trivial, this integral vanishes.

\begin{proposition}\label{pre_determinantal_form_vanishes}
	Given any integers $j$, $k$, and $n$ with $j+k\leq n$, 	a point $b$ in $B$, and vectors $v_1,\dots,v_{j+k}$ in $T_bB$, we have $$(\Omega^n_{j,k})_b(v_1,\dots,v_{j+k})=0$$ whenever the collection $v_1,\dots,v_{j+k}$ is not linearly independent. 
\end{proposition}
\begin{proof}
	Let $J$ denote the vector space generated by the last $j$ vectors. Assume one of the last $j$ vectors $v_{k+l}$ is a linear combination of the other $j+k-1$ vectors. By the antisymmetry properties of $\Omega^n_{j,k}$, $$(\Omega^n_{j,k})_b(v_1,\dots,v_{j+l},\dots, v_{j+k})=(\Omega^n_{j,k})_b(v_1,\dots,v_{j+l}-\pi_{J}(v_{j+l}),\dots, v_{j+k}).$$
	This means that, without loss of generality, we can assume $v_{k+l}$ is written as a linear combination of the first $k$ vectors: $$v_{k+l}=\sum_{r=1}^k a_r v_r$$ for some real numbers $a_1,\dots,a_k$. For the same reason, by replacing say $v_1$ by $\frac{1}{a_1}v_{k+l}$, one can assume $v_{k+l}=a_1v_1$ is a multiple of $v_1$.
	To show the vanishing of $(\Omega^n_{j,k})_b(v_1,\dots,v_{j+k})$, one writes down explicitly
	\begin{align}
		(\Omega^n_{j,k})_b(v_1,\dots,v_{j+k})=\sum_{(i_1<\dots<i_{j+k})\in\{1,\dots,n\}^{j+k}\atop(j_1<\dots<j_k)\in\{i_1,\dots,i_{j+k}\}^{k}}& (-1)^{sgn(\sigma)} \left(\sum_{\mu\in S_k}(-1)^{sgn(\mu)}v_{1,\mu(j_1)}\dots v_{k,\mu(j_k)}\right)\times \notag \\
		&\left(\sum_{\nu\in S_j}(-1)^{sgn(\nu)}v_{k+1,\nu(s_1)}\dots v_{k+j,\nu(s_j)}\right). \notag
	\end{align} 
	Recall that $\sigma$ appears in the definition of $\Omega^n_{j,k}$ and is the permutation taking $\{i_1,\dots,i_{j+k}\}$ to $\{j_1,\dots,j_k,s_1,\dots,s_j\}$. When replacing $v_{k+l}$ by $a_1v_1$, one can see that for every term, composing $\sigma$ with the transposition $(j_1 s_l)$ yields a term equal in absolute value but with opposite sign, hence they cancel out in pairs.
	The same argument can be carried out by taking one of the first $k$ and assuming it is written as a linear combination of the last $j$ vectors.
\end{proof}

It follows that, in the special case where $n=j+k$, given any point $b\in B$ and $n$ tangent vectors $v_1,\dots,v_n\in T_bB$, $$(\Omega^n_{j,k})_b(v_1,\dots,v_n)=vol(v_1,\dots,v_n).$$ 

\begin{corollary}\label{determinantal_form_vanishes}
	Assume $Z$ is an algebraically trivial $k$-cycle for $k\geq 1$. Then  $$\int_{\gamma}\Omega^n_{k+1,k}\in \mathcal{P}^n_{k+1,k}$$ for any $\gamma\in C_{k+1}(B;\Lambda^k T_\Z B)$ satisfying $\partial\gamma=\overline{Z}$.
\end{corollary}
\begin{proof}
	Let $W$ be an algebraic equivalence of $k$-cycles $Z_1$ and $Z_2$ for $k\geq1$, and $\tilde{W}$ its associated $(k+1)$-chain in $C_{k+1}(B;\Lambda^k T_\Z B)$ constructed in Section \ref{construction_tropical_chain_equivalence}. Then it suffices to show $$\int_{\tilde{W}}\Omega^n_{k+1,k}=0.$$
	To compute this integral we locally pair $\Omega^n_{j,k}$ on each $k+1$-cell $\sigma$ in the support of $\tilde{W}$ with $k$ vectors coming from the framing $\beta(\sigma)$ and $k+1$ vectors tangent to $\sigma$. Notice that by construction the framing $\beta(\sigma)\in\Lambda^k T_\Z B$ is the pullback of a tautological framing on a $k$-cell. Of course a tautological framing is written as a wedge product of vectors living inside the tangent space to the $k$-cell, therefore they are clearly not linearly independent, and the integral vanishes.
	Notice that, if $k=0$, the $\beta(\sigma)$ are simply integers hence the linear dependence argument cannot be carried out.  
\end{proof}

\begin{remark}\label{determinantal_form_vanishes_locally}
	We have actually proven a statement stronger than $$\int_{\tilde{W}}\Omega^n_{k+1,k}=0,$$  namely that \enquote{$\Omega^n_{k+1,k}\vert_{\tilde{W}}=0$}, in the sense that $\Omega^n_{k+1,k}$ evaluates to zero at each point of $\tilde{W}$. 
\end{remark}

\begin{remark}\label{no_trop_flux_0_cycles}
	We have mentioned that this result does not hold for $k=0$; the argument of linear dependence used for the vanishing result (Corollary \ref{determinantal_form_vanishes}) cannot be used for points. As a counterexample, consider the case $n=1$ where $B\cong S^1$ with two points $p$ and $q$ as shown below. In this setting $\Omega^1_{1,0}=dx \in C^0(B;T_\Z^*B)$. 
	\begin{figure}[h!]
		\centering
		\includegraphics[width=1.2in]{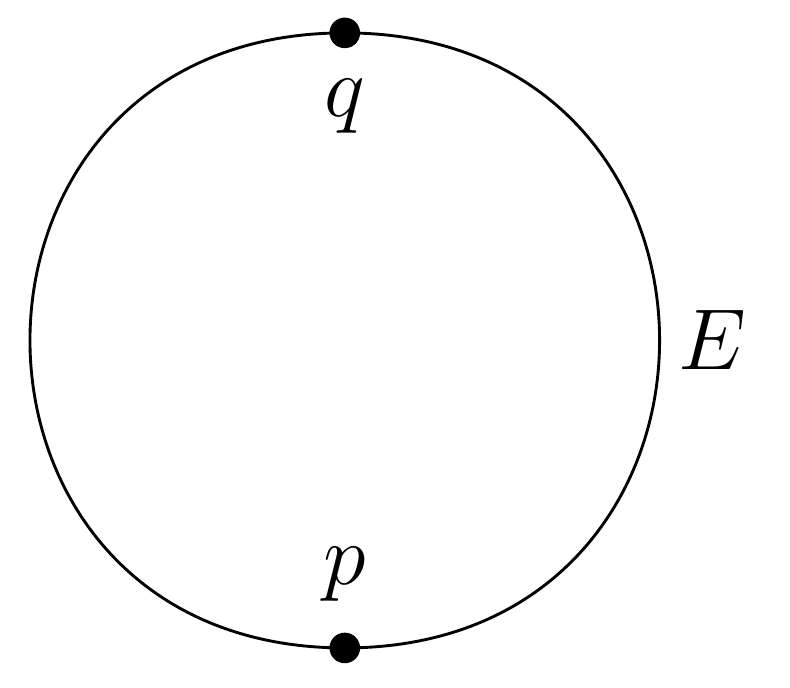} 
	\end{figure}
	These points are algebraically equivalent by the diagonal $\Delta$ in $S^1\times S^1$, and the associated chain $\tilde{\Delta}\in C_1(S^1;\Z)$ is the edge $E$, therefore $$\int_{\tilde{\Delta}}\Omega^1_{1,0}=l(E)\neq 0.$$
\end{remark}

\section{(Oriented) Lagrangian cobordisms and flux}\label{symplectic_flux}

The notion of Lagrangian cobordism was originally introduced by Arnold in \cite{Arnold_Lag_Leg_cobordisms}.

\begin{definition}\label{lag_cobordism}
	Two finite sets of Lagrangians $(L_1,\dots,L_{k^+})$ and $(L'_1,\dots,L'_{k^-})$ in $M$ are \textit{Lagrangian cobordant} if there exists a Lagrangian $V$ in $M\times\C$ and a compact subset $K\subset M\times\C$ such that
	$$V\backslash K \cong \bigsqcup_{i=1}^{k^+}L_i\times(-\infty,-1]\times\{x_i\}\sqcup \bigsqcup_{j=1}^{k^-}L'_j\times[1,+\infty)\times\{x'_j\}$$ for pairwise distinct real numbers $\{x_i\}$ and $\{x'_j\}$.
	The \textit{Lagrangian cobordism group} of $M$, $\Cob(M)$, is then $$\Cob(M):=\bigoplus_{L\in M}\Z\cdot L/\sim,$$ where $\sim$ denotes the equivalence relations defined by $$L_1+\dots+L_{k^+}\sim L'_1+\dots+L'_{k^-}$$ if there is a Lagrangian cobordism between the finite sets $(L_1,\dots,L_{k^+})$ and $(L'_1,\dots,L'_{k^-})$. 
\end{definition}
This definition can be modified to fit the desired set-up by adding assumptions on the finite sets of Lagrangians and the cobordisms themselves with a compatibility condition on the ends; in this way one can upgrade Lagrangians to Lagrangian \textit{branes}, require exactness, monotonicity, etc. Most relevant to our considerations is the \textit{oriented} Lagrangian cobordism group, which we denote by $\Cob^{or}(M)$. 

A slight variation of Lagrangian cobordisms are \textit{cylindrical} Lagrangian cobordisms, introduced by Sheridan--Smith in \cite{nick_ivan1}. These differ from Definition \ref{lag_cobordism} in that $V$ is a Lagrangian in $M\times\C^*$, which outside of a compact subset in $\C^*$ carries trivially fibered finite sets of Lagrangians in $M$ over radial rays. This variation appears naturally from mirror symmetry, where cylindrical cobordisms provide the more natural mirror to rational equivalence. 

\begin{example}\label{suspension_cobordism}
	Let $L\subset M$ be a Lagrangian, and $h\in\mathcal{C}^\infty(M\times[0,1];\R)$ a Hamiltonian on $M$. Then $L$ and $\phi^1(L)$ are Lagrangian cobordant, where $\phi^t$ is the time $t$-flow of the Hamiltonian vector field associated to $h$. This can be realised by the \textit{Hamiltonian suspension cobordism}
	\begin{align}
		V: &\hspace{1mm}L\times[0,1]\longrightarrow M\times\C \notag \\
			&\hspace{4mm}(x,t)\mapsto (\phi^t(x),t+ih(x,t)), \notag
	\end{align}
	which is naturally oriented if $L$ is. 
\end{example}

\begin{example}\label{surgery_trace_cobordism}
	Let $L$ and $L'$ be two Lagrangians intersecting transversally at a point $p= L\pitchfork L'$, and denote by $L\#_pL'$ the resulting Polterovich surgery. There exists a \textit{surgery trace cobordism} $$V:(L,L')\rightsquigarrow L\#_p L',$$ which is constructed in detail in \cite[Section 6.1]{bc1}. More involved cobordisms can be constructed from general Lagrangian $k$-(anti)surgeries, see \cite{haug_antisurgery}.
\end{example}

If $dim(M)=2n$ and when one restricts to \textit{oriented} Lagrangians, there is a \textit{cycle class map} \begin{align}
	cyc: \hspace{2mm}\Cob^{or}(M)\longrightarrow& H_n(M;\Z) \notag \\
	L_1+\dots+L_k\mapsto& [L_1]+\dots+[L_k], \notag
\end{align}
from which one can define the \textit{homologically trivial subgroup} $\Cob^{or}(M;\Z)_0$  appearing in the exact sequence
\begin{equation}\label{ses_cyc_class_map}
	\Cob(M)^{or}_0\hookrightarrow\Cob^{or}(M)\longrightarrow H_n(M;\Z). \notag
\end{equation}
While proving the \textit{existence} of a cobordism usually relies on exhibiting Lagrangian surgeries or Hamiltonian isotopies between Lagrangians, \textit{obstructions} to their existence in the oriented case can be obtained by the presence of some \textit{symplectic flux} between homologically equivalent Lagrangians. This is the strategy employed by Arnold and Haug in \cite{Arnold_Lag_Leg_cobordisms,haug_T2} to compute $\Cob(T^*S^1)^{or}$ and $\Cob(T^2)^{or}$, respectively.

Classically, one can associate a flux $\flux(\iota)\in H^1(L;\R)$ to a Lagrangian isotopy $\iota:L\times[0,1]\rightarrow M$ $$\flux(\iota)([\gamma]):=\int_{S^1\times[0,1]}(\iota\circ(\gamma\times\Id_{[0,1]}))^*\w,$$ which is geometrically the area swept by $\gamma$ under the isotopy. In our setting, we refer to a \textit{flux} between homologically equivalent Lagrangians $L$ and $L'$ as integrals of certain closed differential forms over an $(n+1)$ chain $\gamma$ with $\partial\gamma=L-L'$, modulo the periods of the form. They yield obstructions to the existence of oriented Lagrangian cobordisms between $L$ and $L'$, as the following Lemma illustrates: 

\begin{lemma}\label{cob_implies_periods}
	Let $L$ and $L'$ be oriented Lagrangians in $M^{2n}$ with $L=L'$ in $\Cob^{or}(M)$. Consider a smooth $(n+1)$-chain $\gamma$ in $M$ satisfying $\partial\gamma=L-L'$. Then, for any closed $(n-1)$-form $\alpha$ on $M$, $$\int_\gamma\alpha\wedge\w$$ lies in the periods of $\alpha\wedge\w$.
\end{lemma}
\begin{proof}
	The equality $L=L'$ in $\Cob(M)$ means there exists Lagrangians $L_1,\dots,L_k$ and a Lagrangian cobordism  $V\subset M\times\C$ betweeen $(L_1,\dots,L_j,L,L_{j+1},\dots,L_k)$ and $(L_1,\dots,L_i,L',L_{i+1},\dots,L_k)$. At the cost of making $V$ immersed, one can modify its cylindrical ends so that they terminate over a single fibre $\{z\}\in\C$. Denote this modified Lagrangian by $\tilde{V}$. Then $\tilde{V}\cup(\gamma\times\{z\})$ is an $(n+1)$-cycle in $M\times\C$. Let $\pi_M:M\times\C\longrightarrow M$ and $\pi_\C:M\times\C\longrightarrow\C$ be projections onto the first and second factor respectively. Notice that, if $r:\C\rightarrow\{z\}$ is a strong deformation retraction, then
	\begin{align}
		\int_{\tilde{V}\cup(\gamma\times\{z\}))}\pi_M^*\alpha\wedge(\pi_M^*\w+\pi^*_\C\w_\C)&=\int_{(id_M\times r)_*(\tilde{V}\cup(\gamma\times\{z\}))}\pi_M^*\alpha\wedge(\pi_M^*\w+\pi^*_\C\w_\C) \notag \\
		&=\int_{(id_M\times r)_*(\tilde{V}\cup(\gamma\times\{z\}))}\pi_M^*(\alpha\wedge\w) \notag \\
		&=\int_{(\pi_M\circ(id_M\times r))_*\tilde{V}\cup\gamma}\alpha\wedge\w, \notag
	\end{align}
	which lies in the periods of $\alpha\wedge\w$. On the other hand, because $\tilde{V}$ is Lagrangian, 
	\begin{align}
		\int_{\tilde{V}\cup(\gamma\times\{z\}))}\pi_M^*\alpha\wedge(\pi_M^*\w+\pi^*_\C\w_\C)&=	\int_{\gamma\times\{z\}}\pi_M^*\alpha\wedge(\pi_M^*\w+\pi^*_\C\w_\C) \notag \\
		&=\int_{\gamma\times\{z\}}\pi_M^*(\alpha\wedge\w) \notag \\
		&=\int_{\gamma}\alpha\wedge\w. \notag
	\end{align}
\end{proof}

One can rephrase this Lemma by saying that, for any closed $(n-1)$-form $\alpha$ on $M$, there exists a \textit{flux-type morphism} 
\begin{align}\label{flux_type_morphism}
	\Theta_\alpha: & \hspace{3mm}\Cob^{or}(M)_0\longrightarrow \R/(\alpha\wedge\w)( H_{n+1}(M;\Z)) \\ &\hspace{10mm} L \hspace{8mm}\mapsto \hspace{10mm}\int_\gamma\alpha\wedge\w, \notag
\end{align}
where $\gamma$ is any smooth $(n+1)$-chain realising the homological equivalence $\partial\gamma=L$.

\begin{remark}
	While Lemma \ref{cob_implies_periods} does not hold for cylindrical Lagrangian cobordism groups because $\C^*$ is not contractible, flux  provides an obstruction in this case as well. If $L$ and $L'$ are cylindrically cobordant, $\int_\gamma\alpha$ is necessarily a period of $\pi_M^*\alpha\wedge(\pi_M^*\w+\pi^*_{\C^*}\w_{\C^*})$, another discrete lattice in $\R$.
\end{remark}

As an example, we apply Lemma \ref{cob_implies_periods} to obstruct the existence of oriented Lagrangian cobordisms in cotangent bundles of tori: 

\begin{proposition}\label{cotangent_torus_flux}
	Let $T$ be an $n$-torus, and $\eta$ be a closed $1$-form on $T$. Let $T_0$ be the zero section, and $T_\eta$ the graph of $\eta$. Then $T_0=T_\eta$ in $\Cob^{or}(M)$ if and only if $\eta$ is exact.
\end{proposition}
\begin{proof}
	Notice that if $\eta$ is exact, the Hamiltonian suspension cobordism from Example \ref{suspension_cobordism} between $T$ and $T_\eta$ is oriented.
	On the other hand, $\eta$ being non-exact implies the existence of a primitive class $z=(z_1,\dots,z_n)\in H_1(T;\Z)$ such that $\int_z\eta\neq0$.  Such a class can be represented by a cycle $Z$ given by a rational line in $[0,1]^n/\sim$. A choice of supplementary plane to $Z$ in $[0,1]^n/\sim$ yields a diffeomorphism $T\cong S^1\times T^{n-1}$, which in turn gives a symplectomorphism $$\phi:(T^*T,\w)\rightarrow (T^*(S^1\times T^{n-1}),\w_Z\oplus\w_{n-1}),$$ with $\w$, $\w_Z$, and $\w_{n-1}$ respectively denoting the standard symplectic forms on $T^*T$, $T^*S^1$, and $T^*T^{n-1}$. 
	
	Consider the $(n+1)$-area swept by the isotopy in $T^*T$, namely the locus $$\gamma:=\{(x,t\eta(x))\in T^*T:x\in T, t\in[0,1]\},$$ and its pushforward $$\phi_*\gamma=\{(x,t(\phi^{-1})^*\eta(x))\in T^*(Z\times T^{n-1})\},$$ which is just the $(n+1)$-area swept by the graph of the closed $1$-form $(\phi^{-1})^*\eta$ on $T^*S^1\times T^*T^{n-1}$.
	
	We wish to apply Lemma \ref{cob_implies_periods} by integrating over $\gamma$ the $(n+1)$-form $\w\wedge \phi^*vol\in\Omega^{n+1}(T^*T)$, where $vol\in\Omega^{n-1}(T^*(S^1\times T^{n-1}))$ is the pullback by the natural projection map $\pi:T^*(S^1\times T^{n-1})\rightarrow T^{n-1}$ of an integral volume form on $T^{n-1}$. Notice that, because $H^{n+1}(T^*T^n\times\C)$ is trivial, it is enough to show that this integral does not vanish. 
	\begin{align}
		\int_{\gamma}\w\wedge\phi^*vol&=\int_{\phi_*\gamma}(\w_Z\oplus\w_{n-1})\wedge vol \notag \\
		&=\int_{\phi_*\gamma}\w_Z\wedge vol. \notag
	\end{align}
	This last integral can be simplified as follows. Let $\tilde{\pi}:\phi_*\gamma\rightarrow T^{n-1}$ be the restriction to $\phi_*\gamma$ of the projection $\pi$ (which is surjective), and $y$ any point in $T^{n-1}$. Each fibre $$\tilde{\pi}^{-1}(y)=\{(s,y,t(\phi^{-1})^*\eta(s,y)):s\in Z, t\in[0,1]\}$$ is a compact cylinder parametrised by $s\in Z\simeq S^1$ and $t\in[0,1]$.  Applying fibrewise integration we get
	\begin{align}
		\int_{\phi_*\gamma}\w_Z\wedge vol&=\int_{T^{n-1}}\left(\int_{\tilde{\pi}^{-1}(y)}\w_Z \right)vol. \notag
	\end{align}
	Notice now that $\int_{\tilde{\pi}^{-1}(y)}\w_Z=\int_Z\eta$ is independent of $y\in T^{n-1}$, therefore $$\int_{\gamma}\w\wedge\phi^*vol=vol(T^{n-1})\int_Z\eta\neq0.$$
\end{proof}

\begin{remark}\label{global_tori_rationality_caveat}
	The result above does not hold for a general embedded Lagrangian torus $\iota:T^n\hookrightarrow M^{2n}$  in any symplectic manifold. By requiring that $$\iota^*:H^1(M;\R)\longrightarrow H^1(T^n;\R)$$ is onto, one can always find a non-zero integral over an $(n+1)$-chain bounding the Lagrangians $T^n$ and $T^n_\eta$ for some non-exact $\eta\in H^1(T^n;\R)$, as constructed in the proof of Lemma \ref{cotangent_torus_flux}. However, this no longer translates to a non-zero flux because one might have $H_{n+1}(M;\Z)\neq0$. 
	Counterexamples appear in recent work by Chassé-Leclercq \cite{chassé2024weinsteinexactnessnearbylagrangians}. They construct Lagrangian tori in any symplectic manifold of dimension $\geq6$ which admit arbitrarily Hausdorff-close disjoint elements in their Hamiltonian orbit. This follows from the characterisation of such orbits for product tori in $\C^n$ by Chekanov in \cite{Chekanov1996LagrangianTI}. Crucially, these Lagrangian tori are not \textit{rational} in the sense that $\w(H_2(M,L))\neq t\Z$ for any $t\geq0$.
	In dimension four, Brendel--Kim exhibit Hamiltonian isotopic Lagrangian tori $\iota:T\hookrightarrow (S^2\times S^2,\w\oplus t\w)$ for $t\in\R$ and $\iota':T'\hookrightarrow (S^2\times S^2,\w\oplus t\w)$ such that $T'\cong T_\eta$ for $\eta\neq0\in H^1(T;\R)$ in a Weinstein neighbourhood of $T$ \cite{brendel_kim}. 
\end{remark}

It is worth comparing this result to a result by Hicks--Mak \cite[Corollary 4.2]{hicks_mak_cute}, who prove that any two oriented Lagrangians which are Lagrangian isotopic are Lagrangian cobordant. In particular, in the setting of the Lemma \ref{cotangent_torus_flux}, $T_\eta$ is Lagrangian cobordant to $T_0$ for \textit{any} closed $1$-form $\eta$ on $T$, and our orientation assumption on the cobordism is crucial. In the remaining paragraphs of this section, we expand on why this extra rigidity is in fact desired from the perspective of mirror symmetry. 

In the interest of studying the Fukaya category of a symplectic manifold, one might want to consider \textit{Lagrangian branes}. These are Lagrangians equipped with a \textit{grading} in order for their morphism spaces to be $\Z$-graded, as well as (twisted) Pin structures which are used to orient the moduli spaces defining the $A_\infty$ operations, and are therefore required unless one is content to work with $\Z_2$-coefficients. 

The important implication here is that a graded Lagrangian automatically possesses a \textit{relative} orientation with respect to some local system $\xi\subset \Lambda^n_\C TM$. Details can be found in \cite[Section 11j]{seidel_book} (recall that defining such gradings required the assumption $2c_1(M)=0$). When furthermore $c_1(M)=0$, in particular for the tori we are concerned with in this paper, one can choose this local system to be trivial. Therefore these relative orientations are in fact honest orientations on the Lagrangians, and cobordisms of Lagrangian branes should preserve these orientations. 

\begin{remark}
	Of course in the graded setting but with $c_1(M)$ $2$-torsion, one would still have a \textit{relative cycle class map}
	\begin{align}
		cyc_\xi: \Cob(M)\longrightarrow& H_n(M;\xi) \notag \\
		L_1+\dots+L_k&\mapsto [L_1]_\xi+\dots+[L_k]_\xi, \notag
	\end{align}
	and flux obstructions could be constructed analogously in this setting using the theory developped in \cite{local_coefficients_steenrod}, through a \enquote{local coefficient version} of the morphism in Equation \eqref{flux_type_morphism}. 
\end{remark}

\section{Algebraic Lagrangian cobordisms}\label{algebraic_Lagrangian_cobordisms}

This Section uses some of the tropical background reviewed in Sections \ref{tropical_tori} and \ref{tropical_curves_Jac_Ceresa} concerning tropical tori and tropical curves. 

\subsection{Motivation and definitions}\label{def_alg_lag_cob}

As advertised in the introduction, our motivation for defining algebraic Lagrangian cobordisms is to provide an equivalence relation on Lagrangians mirror to algebraic equivalence of cycles. Because we use the language of Lagrangian correspondences to define and study algebraic Lagrangian cobordisms, we work with immersed Lagrangians. 

Let $Y$ be an algebraic variety over $\C$, and $Z_0$, $Z_\infty$ two algebraic $k$-cycles in $Y$. The most frequent definition found in the literature is that $Z_0$ and $Z_\infty$ are said to be \textit{algebraically equivalent} if they are two fibres of a flat family over a curve. However, an equivalent definition replaces the curve by any smooth projective variety; this is the version found in \cite[§10.3]{fulton_intersection_theory}. The definition which is most relevant to us replaces the curve by an abelian variety. Equivalence of these three definitions was proved by A. Weil in \cite[Lemme 9]{weil_criteres_equivalence} (translated in \cite[§III.1]{abelian_varieties_lang}), and in modern algebro-geometric language by Achter--Casalaina-Martin--Vial in \cite{parameter_spaces_alg_equiv}. 

\begin{definition}\label{alg_equiv_new_def}
	The cycles $Z_0$ and $Z_\infty$ are said to be \textit{algebraically equivalent} if they are fibres over points $p,q\in B$ of a $(k+g)$-cycle $\Gamma$ in $Y\times B$, flat over $B$, where $B$ is an abelian variety of dimension $g$ over $\C$. 
\end{definition}

 \begin{remark}\label{bertini_algequiv_over_curve}
 	Notice that applying Bertini's Theorem enough times to $\Gamma$ using that $B$ is projective, one readily obtains $Z_0$ and $Z_\infty$ as fibres over $p$ and $q$ of a flat family over a curve between these two points. 
 \end{remark}

We call \textit{symplectic polarised torus} a symplectic manifold $X(B):=T^*B/T^*_\Z B$ where $B$ is a polarised tropical torus (tropical tori and polarisations thereof are introduced in Section \ref{tropical_background}).  Recall that a Lagrangian correspondence $\Gamma$ between symplectic manifolds $M_1$ and $M_2$ is a Lagrangian in $\overline{M_1}\times M_2$. Given a Lagrangian $L_1$ in $M_1$, the \textit{geometric composition} $\Gamma(L_1)$ of $\Gamma$ with $L_1$ is $\pi_{M_2}(\Gamma\cap(L_1\times M_2))\subset M_2$. Possibly after perturbing $\Gamma$ or $L_1$ by a Hamiltonian isotopy, the intersection can be made transverse, and in this case $\Gamma(L_1)$ is a immersed Lagrangian (see e.g. \cite{ww_quilted_floer_cohomology}[Lemma 2.0.5]).

Similarly as for Lagrangian cobordisms, one can choose a set of suitable Lagrangians on which to define the equivalence relation. Furthermore, if one endows these Lagrangians with additional decorations (e.g. orientation, local systems, etc), one should compatibly decorate the algebraic Lagrangian cobordism. Very generally, we will denote by $\mathcal{L}(M)$ a set of \textit{suitable (decorated) Lagrangians} in the symplectic manifold $M$. As we have motivated at the end of the previous section, for our purposes this will mean closed \textit{oriented} Lagrangians.

\begin{definition}\label{alg_lag_cob}
	Let $(L_1,\dots,L_k)$ and $(L_1',\dots,L_{k'}')$ be two finite sets of Lagrangians in $\mathcal{L}(M)$. They are said to be \textit{algebraically Lagrangian cobordant} if there exists a symplectic polarised torus $X(B):=T^*B/T^*_\Z B$ with points $p,q\in B$, and a Lagrangian correspondence $\Gamma\in\mathcal{L}(\overline{X(B)}\times M)$ such that 
	\begin{itemize}
		\item $\Gamma(F_p)=L_1+\dots+L_k$ and $\Gamma(F_q)=L'_1+\dots+L'_{k'}$, where $F_p$ and $F_q$ are the fibres in $X(B)$ at $p$ and $q$ respectively;
		\item $\pi_B(\Gamma)=B$, where the map $$\pi_{B}:M\times X(B)\longrightarrow B$$ is the composition of the projection onto the second factor $M\times X(B)\rightarrow X(B)$ with the natural projection to the base $X(B)\rightarrow B$. 
	\end{itemize}
	Then we call $\Gamma$ an \textit{algebraic Lagrangian cobordism}, or an \textit{algebraic Lagrangian correspondence} between $(L_1,\dots,L_k)$ and $(L_1',\dots,L_{k'}')$ and denote this $\Gamma:(L_1,\dots,L_k)\stackrel{X(B)_{p,q}}{\rightsquigarrow}(L_1',\dots,L_{k'}')$ or simply $\Gamma:(L_1,\dots,L_k)\stackrel{X(B)}{\rightsquigarrow}(L_1',\dots,L_{k'}')$.
	From this we define the \textit{algebraic Lagrangian cobordism group} of $M$ as $$\algcob(M):=\oplus_{L\in\mathcal{L}}\Z\cdot L/\sim,$$ where $\sim$ denotes the equivalence relation generated by algebraic Lagrangian cobordism. 
\end{definition}

We will focus our attention on the \textit{oriented} Lagrangian cobordism group, which we denote $\algcob^{or}(M)$. Here the Lagrangians considered are oriented, and any algebraic Lagrangian cobordism is also oriented with equalities of the type $\Gamma(F)=L$ respecting the orientation (notice this requires a choice of orientation for fibres $F$ of $X(B)$). 

\begin{proposition}\label{algcob_cyc_class}
	There is a well-defined cycle class map $$\algcob^{or}(M)\longrightarrow H_n(M;\Z).$$
\end{proposition}
\begin{proof}
	 Because $\pi_B(\Gamma)=B$, given any path $P:[0,1]\longrightarrow B$ between $p$ and $q$, one can consider $\pi^{-1}_B(P([0,1]))\cap\Gamma$. Perturbing $\pi^{-1}_B(P([0,1]))$ to make the intersection transverse yields a $(n+1)$-chain whose boundary is $L$ and $L'$.
\end{proof}

\begin{remark}\label{cob_implies_algcob}
	Finite sets of Lagrangians which are (cylindrically) cobordant are algebraically Lagrangian cobordant, consistently with the comparison between rational and algebraic equivalence. Recall that a planar cobordism always defines a cylindrical cobordism by quotienting $\C$ by a large imaginary translation, therefore without loss of generality we assume there is a cylindrical cobordism $V\subset M\times\C^*$ between the finite sets of Lagrangians $(L_1,\dots,L_k)$ and $(L_1',\dots, L'_{k'})$. Then there is an obvious algebraic Lagrangian cobordism between those finite sets, obtained by gluing two copies of $V\subset M\times\C^*$ as shown below. 
 \begin{figure}[htb]
	\centering
	\includegraphics[width=2.1in]{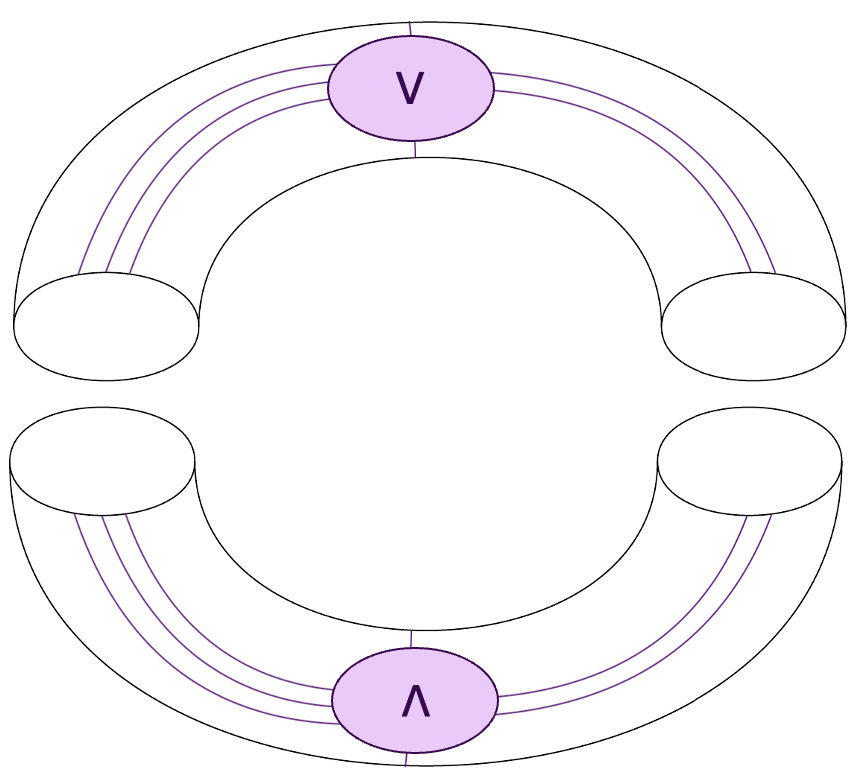} 
	\caption{Gluing cylindrical cobordisms into an algebraic Lagrangian cobordism.}
	\label{Lag_lift}
\end{figure}
	This can be seen as mirroring the fact that any rationally equivalent cycles are algebraically equivalent over an elliptic curve $E$ after a pullback by the double cover $$E\longrightarrow\proj^1.$$
	In particular, there is a natural projection $$\Cob(M)\stackrel{\pi}{\longrightarrow}\algcob(M),$$ and the cycle class maps from each of these groups coincide after factoring through $\pi$. 
\end{remark}

\begin{example}[Abelian varieties]
	Given a polarised tropical torus $B$, one can consider the oriented \textit{fibered} cobordism group $\Cob_{fib}(X(B))$, that is the group generated by fibres $F_b$ for $b\in B$ in which relations come from cobordisms all of whose ends are fibres. It follows from flux considerations that there cannot exist an oriented cobordism between $F_b$ and $F_{b'}$ for $b\neq b'$. In fact, Sheridan--Smith prove in \cite[Theorem 1.2]{nick_ivan1} that $\Cob_{fib}(X(B))_{hom}$  is divisible, which means that for every $L\in\Cob_{fib}(X(B))_{hom}$ and for any $k\in\mathbb{N}$, there exists some $L'\in\Cob_{fib}(X(B))_{hom}$ such that $L=k\cdot L'$\footnote{Note they work with \textit{cylindrical} Lagrangian cobordisms}. This result mirrors an algebraic result about divisibility of (tropical) Chow groups, namely \cite[Theorem 3.1]{bloch_cycles_on_abelian_varieties}, of which they prove a tropical version \cite[Proposition 3.25]{nick_ivan1}. On the other hand, the diagonal $\Delta\subset \overline{X(B)}\times X(B)$ provides an obvious oriented algebraic Lagrangian cobordism between any two fibres $F_b$ and $F_{b'}$, yielding $\algcob(X(B))_{fib}\cong\Z$. This is a symplectic manifestation of the fact that any two points on an abelian variety $A$ lie on a curve, hence $\Gr_0(A)\cong \Z$. Notice that while certain fluxes \textit{do} obstruct the existence of algebraic Lagrangian cobordisms, this specific type of flux between fibres does not. This is a manifestation of Remark \ref{no_trop_flux_0_cycles}. 
\end{example}

As noted in Remark \ref{bertini_algequiv_over_curve}, in the algebraic setting one can recover an algebraic equivalence of cycles over a curve from an algebraic equivalence of cycles over an abelian variety by considering the flat pullback to a curve between the two relevant points. We present a Lemma which is an analogous result for algebraic Lagrangian cobordisms:

\begin{lemma}\label{algcob_yields_Lagcorresp_over_curve}
	Let $\Gamma:L\stackrel{X(B)_{p,q}}{\longrightsquigarrow}L'$ be an algebraic Lagrangian cobordism. There exists a Lagrangian $\Gamma'$ in $M\times X(B)$ such that $\pi_B(\Gamma')\subset B$ is a tropical curve, and $\Gamma'(F_p)=L$, $\Gamma'(F_q)=L'$. 
\end{lemma}
\begin{proof}
	Because $B$ is polarised, it contains a tropical curve passing through any two given points. For our purposes, we want these curves to admit Lagrangian lifts. We use \cite[Lemma 3.28]{nick_ivan1}, which states that, given a point $p$ on a polarised tropical torus $B$, there is an open dense subset $U_p$ of points in $B$ such that, for any point $q\in U_p$, $p$ and $q$ lie on edges of some trivalent tropical curve. Therefore, possibly after perturbing $\Gamma$ by a hamiltonian isotopy, one can take $p$ and $q$ to lie on edges of some trivalent tropical curve $C$ in $B$. 
	The fact that trivalent tropical curves admit Lagrangian lifts follows independently from constructions of Matessi and Mikhalkin in \cite{matessi_lag_pair_of_pants,mikhalkin_trop_to_Lag_corresp}, and is also proved in \cite[Corollary 4.9]{nick_ivan1} with additional properties of the lift (e.g. it is graded and \textit{locally exact}). Let $L_C\subset X(B)$ denote such a lift. 
	We now recall a construction of A. Subotic \cite{subotic_monoidal_structure} known as fibrewise sum in a Lagrangian torus fibration with a distinguished Lagrangian section. In the case of a symplectic polarised torus $X(B)$, this is the operation of geometric composition with the correspondence 
		\begin{align}\label{connect_sum_tori}
		\Sigma:=\{(b,f_1,b,f_2,b,f_1+f_2): b\in B, f_1, f_2\in F_b\}\subset \overline{X(B)\times X(B)}\times X(B).  \notag 
	\end{align}
	We now consider the correspondence $\varphi_{13452}(\Delta_M\times\overline{\Sigma})\subset \overline{M}\times X(B)\times X(B)\times \overline{X(B)}\times M$, where $\Delta_M\subset \overline{M}\times M$ is the diagonal, and 
	\begin{align}
		\varphi_{13452}&:M\times M\times X(B)\times X(B)\times X(B)\longrightarrow M\times X(B)\times X(B)\times M\times X(B) \notag \\
		&\hspace{20mm}(z_1,z_2,z_3,z_4,z_5)\hspace{15mm}\mapsto \hspace{15mm}(z_1,z_3,z_4,z_5,z_2) \notag
	\end{align}
	is simply the map changing the order of the product components. This correspondence takes Lagrangians in $M\times \overline{X(B)\times X(B)}$ to Lagrangians in $\overline{X(B)}\times M$. From this we define $\Gamma'$ as  $\varphi_{13452}(\Delta_M\times\Sigma)(\Gamma\times L_C)$. 

	We check that $\Gamma'(F_p)=L$, and verifying $\Gamma'(F_q)=L'$ is analogous. Notice that, by definition of $\Gamma'$,
	\begin{align}
		\Gamma'(F_p)&=\pi_M(\Gamma'\pitchfork(M\times F_p))=\{m\in M: \exists f\in F_p, (m,f)\in \Gamma'\}\notag\\
		&=\{m\in M: \exists f,f_1,f_2\in F_p, (m,f_1)\in\Gamma, f_2\in L_C, f=f_1+f_2\}.\notag
	\end{align}
	It remains to notice that because $\Gamma(F_p)=L$, there exists $f_1\in F_p$ with $(m,f_1)\in\Gamma$ if and only if $m\in L$. For $m\in L$, one can take any $f_2\in F_p\cap  L_C$ to be the intersection of $F_p$ with the zero section in $X(B)$, then $f=f_1$ satisfies $(m,f)\in\Gamma'$. 
\end{proof}

\subsection{Flux maps from $\algcob(M)$}\label{flux_from_alglagcob}

In this Section, we prove that some of the flux maps $\Theta_\alpha$ defined in Section \ref{symplectic_flux} from $\Cob(M)$ (namely those with $\alpha$ of the form $\w\wedge\beta$) factor through $\Cob(M)\stackrel{\pi}{\longrightarrow}\algcob(M)$. Similarly as for Lagrangian cobordism groups, we denote $\algcob^{or}(M)_{hom}$ for the homologically trivial part of the oriented algebraic Lagrangian cobordism group (that is, the kernel of the cycle class map from Proposition \ref{algcob_cyc_class}).

\begin{theorem}\label{algcob_flux_theorem}
	For any closed form $\beta\in\Omega^{n-3}(M),$ there is a \textit{flux-type} morphism
	\begin{align}
		\tilde{\Theta}_\beta:&\hspace{2mm}\algcob^{or}(M)_{hom}\longrightarrow\R/\left([\w^2\wedge\beta]\cdot H_{n+1}(M)\right)\notag\\
		&\hspace{9mm}L\hspace{17mm}\mapsto\hspace{12mm} \int_{\gamma_0}\w^2\wedge\beta, \notag
	\end{align} 
	where $\gamma_0$ is an $(n+1)$-chain in $M$ with $\partial\gamma_0=L$.
\end{theorem}
\begin{proof}
	Assume there exists a Lagrangian correspondence $\Gamma\subset M\times \overline{X(B)}$ between the finite sets of Lagrangians $(L_1,\dots,L_k)$ and $(L'_1,\dots, L'_{k'})$. It suffices to show that, if $\gamma_0$ is an $(n+1)$-chain in $M$ with $\partial\gamma_0=L_1+\dots+L_k-L'_1-\dots-L'_{k'}$, then $$\int_{\gamma_0}\w^2\wedge\beta\in [\w^2\wedge\beta]\cdot H_{n+1}(M;\Z).$$
	
	We proceed by constructing an $(n+1)$-cycle $C$ in $M\times \overline{X(B)}$, and show on the one hand that the integral of $(\pi_M^*\w-\pi_{X(B)}^*\w_{X(B)})^2\wedge\pi^*_{M}\beta$ on $C$ is the same as an integral over an $(n+1)$-chain $\gamma_0$ as above, and on the other hand that $C$ can be continuously deformed to an $(n+1)$-cycle in $M\times F_p$. The result will follow.  Throughout the proof, we will assume all intersections we consider are transverse, which can always be achieved by a small hamiltonian perturbation of $\Gamma$. 
	
	We start by choosing a path $$P:[0,1]\longrightarrow B$$ such that $P(0)=p$ and $P(1)=q$. Because $\pi_B(\Gamma)=B$, for all $t\in[0,1]$, $\pi^{-1}_{B}(P(t))\cap\Gamma$ is non-empty. 
	We write $$C_1:=\Gamma\cap(\pi^{-1}_B(P([0,1]))).$$

	Recall that the flat connection determined by the affine structure on $B$ defines a canonical parallel transport map 
	\begin{align}
		\varphi_\gamma:\hspace{1mm}&[0,1]\times\{q\}\times F_q\longrightarrow \pi_B^{-1}(\gamma([0,1]))\subset X(B)\notag \\
		&\hspace{9mm}(t,y)\hspace{12mm}\mapsto\hspace{1mm}(\gamma(t),\tilde{\varphi}(t,y)) \notag
	\end{align}
	for any path $\gamma:[0,1]\rightarrow B$ with $\gamma(0)=q$, where $\tilde{\varphi}(t,y)\in F_{\gamma(t)}$ is the parallel transport along $\gamma([0,t])$ of $y\in F_q$.
	From this, we define $C_2$ as the image of the embedding
	\begin{align}
		&[0,1]\times(\Gamma\cap(M\times F_q))\longrightarrow M\times X(B) \notag\\
		&\hspace{14mm} (t,x,q,y)\hspace{10mm}\mapsto \hspace{1mm}(x,\varphi_{\overline{P}}(t,y)), \notag
	\end{align}
	where $\overline{P}$ denotes the path $P$ with the orientation reversed, and $x\in M$ while $y\in F_q$.
	Notice that $C_1\cup C_2$ is an $(n+1)$-chain in $M\times X(B)$ with boundary in $M\times F_p$, and more precisely $\pi_M(\partial(C_1\cup C_2))=L_1+\dots+L_k-L'_1-\dots-L'_{k'}\subset M$. The fact that the boundary component in $M\times F_q$ is empty simply follows from the fact that the boundaries of $\partial C_1$ and $\partial C_2$ cancel out by construction of $C_2$.

	Our next step is to construct an $(n+1)$-chain $C_3$ inside $M\times F_p$ with $\partial C_3=\partial(C_1\cup C_2)\subset M\times F_p$. Here we use parallel transport again from the fibres $F_{P(t)}$ to $F_p$, this time sending $\Gamma\cap(M\times F_{P(t)})$ to $M\times F_p$. More explicitly, we take $C_3$ to be the image of the map
	\begin{align}
		C_1\hspace{6mm}&\longrightarrow \hspace{4mm}M\times F_p\notag\\
		(x,P(t),f)&\hspace{1mm}\mapsto (x,p=P(0),\tilde{\varphi}_{\tilde{P}_t}(1,f)), \notag
	\end{align}
	where $x\in M$, $f\in F_{P(t)}$, and $\tilde{P}_t:[0,1]\rightarrow B$ is defined by $\tilde{P}_t(s):=P(t(1-s))$. From this construction it follows that $C_3$ is $(n+1)$-dimensional except in the trivial case where $C_1=C_2$ as a set, in which case it is $n$-dimensional, and in particular this would imply $L_1+\dots+L_k=L'_1+\dots+L'_{k'}$. 
	Putting all this together, we have constructed an $(n+1)$-cycle $$C:=C_3-C_1-C_2$$ inside $M\times \overline{X(B)}$. We now compute the integral
	\begin{equation}\label{integral_over_C}
		\int_{C}\left(\pi_M^*\w-\pi_{X(B)}^*\w_{X(B)}\right)^2\wedge\pi_M^*\beta
	\end{equation}
	in two different ways. 
	The first way involves noticing that the only contributions to this integral come from integrating over $C_3$. The fact that the contributions from $C_1$ vanish follows from the fact that $$\left(\pi_M^*\w-\pi_{X(B)}^*\w_{X(B)}\right)\vert_{C_1}=0$$ because $TC_1\subset T\Gamma$ and $\Gamma$ is Lagrangian in $M\times\overline{X(B)}$. Although the same doesn't necessarily hold for $C_2$, it is true that $$\left(\pi_M^*\w-\pi_{X(B)}^*\w_{X(B)}\right)^2\vert_{C_2}=0.$$ 
	
	To see this, one can decompose the tangent space to $C_2\subset M\times \overline{X(B)}$ at any point $(x,P(t),f_{P(t)})\in C_2$. The component along $M$ lives in an isotropic subspace of $T_xM$, because $C_2$ is parallel transport of $\Gamma\cap(M\times F_q)$.  By considering the trivialisation $$t_b: X(B)\longrightarrow B\times F_b$$ given by parallel transport, one sees that the component along $T_{P(t)}B$ is one-dimensional, therefore $\w_{X(B)}^2$ necessarily vanishes along the $X(B)$ direction. 
	
	We have now reduced the integral from Equation \eqref{integral_over_C} to
	\begin{align}
		\int_{C_3}\left(\pi_M^*\w-\pi_{X(B)}^*\w_{X(B)}\right)^2\wedge\pi_M^*\beta= \int_{C_3}\pi_M^*(\w^2\wedge\beta)=\int_{(\pi_M)_*C_3}\w^2\wedge\beta. \notag
	\end{align}
	Here the first equality follows from the fact that $C_3$ is contained in $M\times F_p$ and $F_p$ is Lagrangian in $X(B)$. Notice that wa have exhibited an $(n+1)$-chain in $M$, namely $(\pi_M)_*C_3$, whose boundary is $L_1+\dots+L_k-L'_1-\dots-L'_{k'}$, and over which $\w^2\wedge\beta$ integrated to something which lies in $$\left[\left(\pi_M^*\w-\pi_{X(B)}^*\w_{X(B)}\right)^2\wedge\pi_M^*\beta\right]\cdot H_{n+1}(M\times X(B);\Z).$$ 
	The aim of the second computation from Equation \eqref{integral_over_C} is to show that the integral over $C$ can actually be reduced to an integral of $\w^2\wedge\beta$ over an $(n+1)$-cycle in $M$. That is, the integral lies in $$[\w^2\wedge\beta]\cdot H_{n+1}(M;\Z)\subset\left[\left(\pi_M^*\w-\pi_{X(B)}^*\w_{X(B)}\right)^2\wedge\pi_M^*\beta\right]\cdot H_{n+1}(M\times X(B);\Z).$$
	This follows from the trivial fact that the composition of paths $\overline{P}\star P$ is contractible. Denote by $$r:[0,1]^2\rightarrow B$$ a homotopy to $\{p\}$, that is $r(0,s)=p=r(1,s)$, $r(s,0)=	\overline{P}\star P(s)$ and $r(s,1)=p$. This induces a strong deformation retraction $$R(1): \pi_B^{-1}(P([0,1]))\longrightarrow M\times F_p,$$ again obtain  by parallel transport through the following family of maps parametrised by $s\in[0,1]$: 
	\begin{align}
		R(s):&\hspace{1mm}\pi_B^{-1}(P([0,1]))\longrightarrow \pi_B^{-1}(r([0,1],s)) \notag \\
		&\hspace{3mm}(x,P(t),f)\hspace{3mm}\mapsto (x,r(t,s),\varphi_{r(t,\cdot)}(f)).\notag
	\end{align} 
	Noticing that $C\subset \pi_B^{-1}(P([0,1]))$, this allows us to compute the integral in Equation \eqref{integral_over_C}:
	\begin{align}
		\int_{C}\left(\pi_M^*\w-\pi_{X(B)}^*\w_{X(B)}\right)^2\wedge\pi_M^*\beta&=\int_{R(1)_*C}\left(\pi_M^*\w-\pi_{X(B)}^*\w_{X(B)}\right)^2\wedge\pi_M^*\beta \notag \\
		&=\int_{R(1)_*C}\pi_M^*(\w^2\wedge\beta)\notag \\
		&= \int_{(\pi_M\circ R(1))_*C}\w^2\wedge\beta.\notag 
	\end{align}
	It remains to observe that $(\pi_M\circ R(1))_*C$ is an $(n+1)$-cycle in $M$, therefore we have shown that $$	\int_{C}\left(\pi_M^*\w-\pi_{X(B)}^*\w_{X(B)}\right)^2\wedge\pi_M^*\beta=\int_{(\pi_M)_*C_3}\w^2\wedge\beta$$ lies in the periods of $\w^2\wedge\beta$, from which we conclude the proof. 
\end{proof}

The following is immediate by comparing the definition of the flux maps from $\Cob^{or}(M)$ and $\algcob^{or}(M)$.

\begin{corollary}	
	For $\alpha=\w\wedge\beta$, the flux morphisms $\Theta_\alpha$ and $\tilde{\Theta}_\beta$ defined from $\Cob^{or}(M)$ and $\algcob^{or}(M)$ factor through the projection map as follows
	\begin{center}
			\begin{tikzcd}
			\Cob^{or}(M) \arrow[rd, "\Theta_{\w\wedge\beta}"'] \arrow[rr, "\pi"] &                                        & \algcob^{or}(M) \arrow[ld, "\tilde{\Theta}_{\beta}"] \\
			& \R/\left({[\w^2\wedge\beta]\cdot H_{n+1}(M;\Z)}\right). &                                                     
		\end{tikzcd}
	\end{center}
\end{corollary}

\subsection{A Lefschetz Jacobian}\label{Lefschetz_jacobian}

In this Section, we reformulate the flux maps constructed in Sections \ref{flux_type_morphism} and \ref{flux_from_alglagcob} as \enquote{symplectic Abel-Jacobi maps} to a \textit{Lefschetz Jacobian}. These can be thought of as a symplectic analogue of intermediate Jacobians. We work under the assumption that $M$ is a Lefschetz manifold, that is the maps $$\cup[\w]^k:H^{n-k}(M;\R)\longrightarrow H^{n+k}(M;\R)$$ are isomorphisms for all $0\leq k\leq n$.

Letting $\mathcal{L}^{or}$ be a set of suitable oriented Lagrangians in $M$, define $\mathbb{L}^{or}:=\oplus_{L\in\mathcal{L}}\Z\cdot L$. This comes with an obvious cycle class map $$cyc:\mathbb{L}^{or}\longrightarrow H_n(M;\Z),$$ whose kernel we denote $\mathbb{L}^{or}_{hom}$.

There is a generalised flux map map
\begin{align}
	\Theta:\hspace{6mm}\mathbb{L}^{or}_{hom}\hspace{6mm}&\longrightarrow 
	 LJ(M):=H^{n+1}(M;\R)^{\vee}/H_{n+1}(M;\Z)\notag \\
	L\hspace{9mm}&\mapsto \hspace{14mm}
	\left([\eta]\mapsto \int_{\gamma}\eta\right), \notag
\end{align}
where $\partial\gamma=L$. This is independent of the choice of such $\gamma$ because we view $H_{n+1}(M;\Z)$ as a lattice inside $H^{n+1}(M;\R)^{\vee}$ through the usual integration map over $(n+1)$-cycles. 
 To see that $\Theta$ is well-defined at the level of cohomology, notice that an exact $(n+1)$-form $d\nu$ must be of the form $\w\wedge d\mu=d(\w\wedge\mu)$ for some $(n-2)$-form $\mu$. Therefore $$\int_{\gamma}d\eta=\int_{L}\w\wedge\mu=0.$$

Given an algebraic Lagrangian cobordism $\Gamma\subset M\times X(B)$, and given a choice of basepoint $b_0\in B$, one can consider the map 
\begin{align}
	\varphi_\Gamma:B&\longrightarrow LJ(M)\notag \\
	b&\mapsto \Theta(\Gamma(F_b-F_{b_0})). \notag 
\end{align}

A special instance of this is $$\varphi_\Delta:B\longrightarrow LJ(X(B))$$ induced by the diagonal $\Delta\subset X(B)\times X(B)$, which provides an algebraic Lagrangian cobordism between any pair of fibres $F_p$, $F_q$ in $X(B)$. Below we show that the real torus $LJ(X(B))$ can be endowed with a tropical structure making $\varphi_\Delta$ into a tropical embedding. 

It follows from Theorem \ref{cob_implies_algcob} and by the fact that $H^{n+1}(M;\R)\cong [\w]\cup H^{n-1}(M;\R)$ that nullcobordant elements in $\mathbb{L}^{or}_{hom}$ lie in the kernel of $\Theta$. In other words, $\varphi_\Gamma$ factors as
\begin{center}
	\begin{tikzcd}
		B \arrow[rd, "{b\mapsto [\Gamma(F_b)-\Gamma(F_{b_0})]}"', shift left] \arrow[rr, "\varphi_\Gamma"] &                                                                          & LJ(M) \\
		& \Cob^{or}(M)_{hom}. \arrow[ru, "\tilde{\varphi}_\Gamma"', shift left] &         
	\end{tikzcd}
\end{center}

Similarly, Theorem \ref{flux_from_alglagcob} implies that algebraically nullbordant elements in $\mathbb{L}^{or}_{hom}$ are mapped to the subset 
\begin{equation}\label{hey}
	\left(H^{n+1}(M;\R)/([\w]^2\cup H^{n-3}(M;\R))\right)^\vee\subset \left(H^{n+1}(M;\R)\right)^\vee.
\end{equation}

We now describe a tropical structure for $LJ(M)$ by choosing a lattice in $(H^{n+1}(M;\R))^\vee$ of those maps $H^{n+1}(M;\R)\longrightarrow\R$ which evaluate to integers on a certain sublattice $H^{n+1}(M;\R)_\Z$ of $H^{n+1}(M;\R)$. To define $H^{n+1}(M;\R)_\Z$, we invoke the Lefschetz decomposition $$H^{n+1}(M;\R)\cong\bigoplus_{i}[\w]^i\cup \ker([\w]^{2i}:H^{n-2i+1}(M;\R)\longrightarrow H^{n+2i+1}(M;\R)).$$ From this we set
\begin{equation}\label{tropical_structure_lefschetz_jacobian}
	H^{n+1}(M;\R)_\Z:=\bigoplus_{i}[\w]^i\cup\ker([\w]^{2i}:H^{n-2i+1}(M;\Z)\longrightarrow H^{n+2i+1}(M;\R)).
\end{equation}

Unless otherwise stated,  from now on when referring to $LJ(M)$ as a tropical object we endow it with the aforementioned tropical structure.

\begin{proposition}\label{lefschetz_embedding}
	A polarised tropical torus $B$ embeds tropically
	into the Lefschetz Jacobian $LJ(X(B))$. 
\end{proposition}
\begin{proof}
	Let $\Delta\subset X(B)\times X(B)$ be the diagonal, which as we have seen, provides an algebraic Lagrangian cobordism between any pair of fibres $F_p$, $F_q$ in $X(B)$. We will show that the map $$\varphi_\Delta:B\longrightarrow LJ(X(B))$$ is a 
	embedding, where we have chosen a basepoint $b_0\in B$. 
	For any $b_0\neq b\in B$, we consider a line segment $\{b_0+tv\}_{t\in[0,1]}$ between $b_0$ and $b$, where $v\in TB\cong\R^n$. This traces out a Lagrangian isotopy $\{F_{b_0+tv}\}_{t\in[0,1]}$ between $F_{b_0}$ and $F_b$, and the image $\varphi_\Delta(b)\in LJ(X(B))$ can be seen as integration of $(n+1)$-forms along this isotopy, modulo integration over $(n+1)$-cycles. 
	
	We start by understanding the differential $$(\varphi_\Delta)_*:TB\longrightarrow TLJ(X(B)),$$ using the observation above. Recall that up to a global coordinate transformation, $B\cong B(Q):=\R^n/Q\cdot\Z^n$ for some $Q\in GL(n;\R)$, where $\R^n$ is considered with the tropical structure given by the integer lattice $\Z^n\subset\R^n$. Given $v\in TB\cong \R^n$, there is a corresponding element $[v]\in H_1(B;\R)$ obtained by identifying $H_1(B;\R)\cong (Q\cdot\Z^n)\otimes\R\cong TB$. Fixing a point $b\in B$, a fundamental class $[F_b]\in H_n(F_b;\Z)$, and using the Künneth isomorphism
	\begin{equation}
		\psi: \bigoplus_{k=1}^n H^k(B;\R)\otimes H^{n+1-k}(F_b;\R)\stackrel{\cong}{\longrightarrow}  H^{n+1}(X(B);\R), \notag
	\end{equation}
	 the differential is simply $$v\mapsto \psi([v]\otimes[F_b]),$$ which is injective.
	Injectivity of $\varphi_\Delta$ follows from the fact that it acts linearly with respect to scaling the vector $v$, from which we can deduce that $\varphi_\Delta(b)\neq 0 \mod H_{n+1}(X(B);\Z)$ unless $v\in H_1(B;\Z)$. 
	
	The fact that the tropical structure defined in Equation \eqref{tropical_structure_lefschetz_jacobian} makes $\varphi_\Delta$ into a tropical map follows from the observation that if $v\in T_\Z B$, then $[v]\otimes[F_b]\in Q^{-1}\cdot H_1(B;\Z)\otimes H_{n}(F_b;\Z).$ This corresponds to maps $H^{n+1}(X(B);\R)\longrightarrow\R$ which evaluate to integers on the lattice $$\psi\left(Q^{t}\cdot H^1(B;\Z)\otimes H^n(F_b;\Z)\right)\subset H^{n+1}(X(B);\R).$$ Now notice that $$[\w]\in \psi\left(Q^{t}\cdot H^1(B;\Z)\otimes H^1(F_b;\Z)\right)$$
	as it evaluates to integers on homology classes corresponding to $$Q^{-1}\cdot H_1(B;\Z)\otimes H_1(F_b;\Z).$$ In particular, $$\psi\left(Q^{t}\cdot H^1(B;\Z)\otimes H^n(F_b;\Z)\right)\subset [\w]\cup \ker([\w]^2:H^{n-1}(X(B);\Z)\longrightarrow H^{n+1}(X(B);\R)),$$ from which we conclude.
\end{proof}

\begin{remark}
	It follows from Equation \eqref{hey} that the proof of tropicality of a map $\varphi_\Gamma:B\longrightarrow LJ(M)$ induced by an algebraic Lagrangian cobordism $\Gamma\subset \overline{X(B)}\times M$ only depends on the choice of tropical structure in first component of the Lefschetz decomposition, that is in $$[\w]\cup \ker([\w]^2:H^{n-1}(M;\R)\longrightarrow H^{n+3}(M;\R)).$$ 
	Additionally, in this instance where $\Gamma=\Delta\subset \overline{X(B)}\times X(B)$, the only component relevant to the tropical structure is  $$[\w]\cup\psi\left(H^0(B;\R)\otimes H^{n-1}(F_b;\R))\right)\subset [\w]\cup \ker([\w]^2:H^{n-1}(M;\R)\longrightarrow H^{n+3}(M;\R)).$$ This is related to the fact that all of the Lagrangians considered here are fibres; one expects that the terms contained in $$[\w]\cup\psi\left(\bigoplus_{i\geq1}H^{i}(B;\R)\otimes H^{n-i-1}(F_b;\R))\right)$$ would become relevant if one were to consider e.g. Lagrangians lifts in $X(B)$ of tropical $k$-cycles in $B$ with $k\geq1$. 
\end{remark}

\section{Relating tropical and symplectic fluxes}\label{relating_trop_symp_flux}

Recall from the Introduction that our proof of Theorem \ref{intro_main_theorem} involves using Zharkov's tropical flux $\int_{\gamma_0}\Omega_0$ between $C$ and $C^-$ in $J(C)$ to exhibit a non-trivial symplectic flux $\int_{\hat{\phi}(\gamma_0)}\w^2$ between $L_C$ and $L_C^-$ in $X(J(C))$. Here $\Omega_0$ is the notation used by Zharkov to refer to the determinantal form we call $\Omega^3_{2,1}$ in Section \ref{determinantal_forms_section}, and $\gamma_0$ is a tropical chain in $C_2(J(C);T_\Z J(C))$ with $\partial\gamma_0=\overline{C}-\overline{C^-}$. The way to translate this into a symplectic flux in $X(J(C))$ is by a conormal-type construction taking $\gamma_0$ to $\hat{\phi}(\gamma_0)$, and a dual construction taking $\Omega_0$ to $\w^2$. The duality between these construction is what provides the equality $$\int_{\gamma_0}\Omega_0=\int_{\hat{\phi}(\gamma_0)}\w^2.$$

In this Section we formalise this mechanism by which tropical fluxes in a tropical torus $B$ can be translated into symplectic fluxes in the corresponding Lagrangian torus fibration $X(B)$. 

Because $B$ is a torus, this fibration is trivial, and because furthermore the local system $T_\Z^*B$ is constant, it carries a canonical trivialisation
\begin{equation}\label{trivialisation_ltf}
	t_b:X(B)\stackrel{\cong}{\longrightarrow} B\times F_b,
\end{equation} 
associated to an arbitrary choice of point $b\in B$, where $F_b:=(T^*B)_b/(T^*_\Z B)_b$. This is given by parallel transporting elements in any fibre $F_{b'}$ for $b'\in B$ to the fibre $F_b$ by the flat connection determined by the affine structure on $B$. 
In fact, recall that the local systems $T_\Z^{(*)}B$ are constant, because the linear part of the monodromy of the sheaf of affine functions $\textit{Aff}_B\subset C^\infty(B)$ vanishes (i.e., the slope of these functions is preserved by the monodromy). This implies that sections of this local system are all global, and can be identified with points in $(T^*_\Z B)_b$. With this in mind, we will repeatedly use the isomorphisms 
\begin{align}
	&H^j(B;\Lambda^kT^*_\Z B)\cong H^j(B;\Z)\otimes\Lambda^{k}(T^*_\Z B)_b & \text{and} & &H_j(B;\Lambda^{k}T_\Z B)\cong H_j(B;\Z)\otimes\Lambda^{k}(T_\Z B)_b. &  \notag
\end{align}

The product decomposition above implies a Künneth decomposition of the homology of $X(B)$ into terms of the form $H_j(B)\otimes H_k(F_b)$. Each of these is in fact isomorphic to the tropical homology groups $H_j(B;\Lambda^{n-k}T_\Z B)\cong H_j(B;\Z)\otimes\Lambda^{n-k}(T_\Z B)_b$ through 
\begin{equation}\label{chain_of_identificiations}
	(\Lambda^{n-k}T_\Z B)_b\cong (\Lambda^{n-k}T_\Z^*B)_b^*\cong H_{n-k}(F_b;\Z)^*\cong H^{n-k}(F_b;\Z)\cong H_k(F_b;\Z),
\end{equation}
where the last isomorphism is Poincaré duality, and they are all canonical after a choice of orientation on the fibres. An identical reasoning holds at the level of cohomology, where we find $$H^j(B)\otimes H^k(F_b)\cong H^j(B;\Lambda^{n-k}T_\Z^* B)$$ through the chain of isomorphisms dual to that of Equation \eqref{chain_of_identificiations}, that is
\begin{equation}\label{second_chain_identification}
	\Lambda^{n-k} (T_\Z^*B)_b\cong H_{n-k}(F_b;\Z)\cong H^k(F_b;\Z).
\end{equation}

These isomorphisms of (co)homology groups are the key to relating tropical fluxes in $B$ in the sense of Section \ref{tropical_flux} to symplectic fluxes in $X(B)$ in the sense of Section \ref{symplectic_flux}. Nevertheless, they are not sufficient, and one needs to explicitly realise them at the level of (co)chains. One reason for this is to understand what differential forms on $X(B)$ the determinantal forms are mapped to. Another is to show that tautological tropical cycles $\overline{Z}$ are mapped to something resembling a PL Lagrangian lift of $Z$, as well as understand the behaviour of chains of the form $\tilde{W}$ constructed in Section \ref{construction_tropical_chain_equivalence} under these maps.

Our first aim is thus to construct (a family of) (co)chain maps 
\begin{align}
	\hat{\phi}&:C_\bullet(B;\Lambda^kT_\Z B)\longrightarrow C_{\bullet+n-k}(X(B);\Z)\notag \\
	\hat{\psi}&:C^\bullet(B;\Lambda^kT^*_\Z B)\longrightarrow C^{\bullet+n-k}(X(B);\Z) \notag
\end{align}
whose induced maps at the level of (co)homology, 
\begin{align}
	\phi&:H_\bullet(B;\Lambda^kT_\Z B)\longrightarrow H_{\bullet+n-k}(X(B);\Z)\notag \\
	\psi&:H^\bullet(B;\Lambda^kT^*_\Z B)\longrightarrow H^{\bullet+n-k}(X(B);\Z), \notag
\end{align}
realise the Künneth-type isomorphisms described above.

\begin{remark}
	One can wonder how much of our construtions extend to more general settings, for instance when $B$ is singular, or the local system $T_\Z B$ has non-trivial monodromy. In the former situation, the constructions from this section do not translate, and our arguments are no longer valid globally. On the other hand, in a Lagrangian torus fibration which is non-singular, but with non-trivial monodromy for the local systems $T_\Z^{(*)}B$, we expect that our construction can be generalised by upgrading our maps to maps from (co)homology with \textit{local} coefficients. However, the more complicated global topology of $X(B)$ implies the maps on homology and cohomology constructed in the following sections would no longer be isomorphisms, hence the correspondence between tropical and symplectic would no longer be one-to-one.
\end{remark}

\subsection{Map on homology}\label{map_homology}

We construct a family of chain maps $$C_\bullet(B;\Lambda^kT_\Z B)\longrightarrow C_{\bullet+n-k}(X(B);\Z),$$ each depending on arbitrary choices but inducing the same map on homology. The arbitrary choice involved in our construction is a map
\begin{equation}\label{kappa}
	\kappa:H_\bullet(T^n;\Z)\longrightarrow C_\bullet(T^n;\Z) 
\end{equation}
with values in cycles satisfying $[\kappa(x)]=x$. In other words,  a splitting of the short exact sequence
\begin{equation}
	\text{im}(d:C_\bullet(F_b;\Z)\rightarrow C_\bullet(F_b;\Z)) \hookrightarrow \text{ker}(d:C_\bullet(F_b;\Z)\rightarrow C_\bullet(F_b;\Z)) \twoheadrightarrow H_\bullet(F_b;\Z), \notag
\end{equation}
which exists because $F_b\cong T^n$. Here and throughout, the point $b$ is an arbitrary point chosen in $B$.

By Equation \eqref{chain_of_identificiations}, we have isomorphisms $$C_\bullet(B;\Lambda^k T_\Z B)\cong C_\bullet(B;\Z)\otimes \Lambda^k (T_\Z B)_b\cong C_\bullet(B;\Z)\otimes H_{n-k}(F_b;\Z).$$

Therefore, given a $\kappa$ as above and the trivialisation $t$ from Equation \eqref{trivialisation_ltf}, we can define a chain map
\begin{align}
	\hat{\phi}:C_\bullet(B;\Lambda^kT_\Z B)\longrightarrow C_{\bullet+n-k}(X(B);\Z) \notag
\end{align}
as the following composition: $$C_\bullet(B;\Z)\otimes H_\bullet(F_b;\Z)\stackrel{(id,\kappa)}{\longrightarrow}C_\bullet(B;\Z)\otimes C_\bullet(F_b;\Z)\longrightarrow C_\bullet(B\times F_b;\Z)\stackrel{(t^{-1})_*}{\longrightarrow}C_\bullet(X(B);\Z),$$
where the middle arrow is the Eilenberg-Zilber map, which is a quasi-isomorphism.

The fact that this is a chain map with the differential on $C_\bullet(B;\Lambda^kT_\Z B)$ given by restricting the coefficients for the usual singular differential is a straightforward verification, and the fact that the induced map on homology doesn't depend on the choice of $\kappa$ follows from the condition that $[\kappa(x)]=x$. Therefore we have a well-defined map $$\phi:H_\bullet(B;\Lambda^kT_\Z B)\longrightarrow H_{\bullet+n-k}(X(B);\Z).$$

It is instructive to understand this map more explicitly. For this, we construct a choice of $\kappa$, and study the associated map $\hat{\phi}$ at the level of chain complexes. 
Let $\beta$ be an element in $\Lambda^k (T_\Z B)_b\cong H_{n-k}(F_b;\Z)\cong\Lambda^{n-k}(T_\Z^*B)_b$. From this perspective, it can be written (although not uniquely) as
\begin{equation}\label{choice_decomposition}
	\beta=\sum_{i=1}^se_{i,1}\wedge\dots\wedge e_{i,n-k} \in \Lambda^{n-k}(T_\Z^*B)_b
\end{equation}
for some $s\in\mathbb{N}$ and covectors $e_{i,l}\in (T^*_\Z B)_b$, where $l\in\{1,\dots,n-k\}$. Choosing such a decomposition amounts to assigning to each element in $H_{n-k}(F_b;\Z)$ a set of $s$ finite sets of $n-k$ covectors. To these $s$ finite sets, we associate $s$ parallelepipeds in $T^*B$, defined for $i\in\{1,\dots,s\}$ as $$E_i:=\{(b,t_1e_{i,1}+t_2e_{i,2}+\dots+t_{n-k}e_{i,n-k}):(t_1,\dots,t_{n-k})\in[0,1]^{n-k}\}\subset  (T^* B)_b.$$ Passing to the quotient, this yields a cycle $E_\beta$ in $F_b$ whose homology class is $\beta$. Therefore such a choice of decomposition from Equation \eqref{choice_decomposition} yields a map $\kappa$ as in Equation \eqref{kappa}. 
Let $\sigma:\Delta^j\longrightarrow B$ be a  $j$-simplex in $B$. At the chain level, $\hat{\phi}$ maps $\sigma\otimes\beta$ to $(t^{-1})_*(\sigma\times E_\beta)$ in $C_\bullet(X(B);\Z)$.

\begin{example}
	We give a simple example to show that despite its technical construction, the geometry of the maps $\hat{\phi}$ - and hence of $\phi$ - is quite simple. Consider a tropical $2$-torus $B$ with basepoint $b\in B$, and let $\sigma$ be any $1$-simplex in $B$.
	Let $v_1$ and $v_2$ be two generators of $(T_\Z B)_b$, and $v_1^*$ and $v_2^*$ their dual generators of $(T_\Z^*B)_b$. From the identification $H_1(F_b;\Z)\cong (T_\Z^*B)_b$, we can choose two $1$-cycles $\gamma_1$ and $\gamma_2$ such that$[\gamma_1]=v_1^*$ and $[\gamma_2]=v_2^*$. 
	 From this set $\kappa$ to be 
	\begin{align}
		\kappa:H_\bullet(F_b;\Z)&\longrightarrow C_\bullet(F_b;\Z) \notag \\
		[\gamma_1]\hspace{4mm}&\hspace{1mm}\mapsto \hspace{6mm}\gamma_1 \notag \\
		[\gamma_2]\hspace{4mm}&\hspace{1mm}\mapsto \hspace{6mm}\gamma_2. \notag
	\end{align}
	Now take 
	\begin{align}
		x_0&:=\sigma\in C_1(B;\Lambda^0T_\Z B)\cong C_1(B;\Z)\notag \\
		x_1&:= v_1\otimes\sigma \notag \in C_1(B;\Lambda^1T_\Z B) \\
		x_3&:= v_1\wedge v_2\otimes\sigma \in C_1(B;\Lambda^2T_\Z B).\notag
	\end{align}
	We compute the images $\hat{\phi}(x_0)$, $\hat{\phi}(x_1)$, and $\hat{\phi}(x_2)$. Through the isomorphism $$\Lambda^*T_\Z B\cong \Lambda^{n-*}T_\Z^*B\cong H_{n-*}(F_b;\Z),$$ $1\in\Lambda^0T_\Z B$ is mapped to $[\gamma_1]\wedge[\gamma_2]\in H_2(F_b;\Z)$, $v_1\in\Lambda^1 T_\Z B$ is mapped to $[\gamma_2]\in H_1(B;\Z)$, and $v_1\wedge v_2\in\Lambda^2T_\Z B$ is mapped to $1\in H_0(F_b;\Z)$. Therefore, with our choice of $\kappa$, we have
		\begin{align}
		\hat{\phi}(x_0)&:=(t^{-1})_*\left(\sigma\times(\gamma_1\wedge\gamma_2)\right)\in C_3(X(B);\Z) \notag \\
		\hat{\phi}(x_1)&:=(t^{-1})_*\left(\sigma\times\gamma_2\right) \in C_2(X(B);\Z) \notag \\
		\hat{\phi}(x_3)&:=(t^{-1})_*\left(\sigma\right)\in C_1(X(B);\Z). \notag
	\end{align}
\end{example}

\subsection{Map on cohomology}\label{map_cohomology}

Similarly, we construct families of cochain maps
$$\hat{\psi}:C^\bullet(B;\Lambda^kT_\Z^*B)\longrightarrow C^{\bullet+n-k}(X(B);\Z).$$ Again, our arbitrary choice is that of a splitting $$\kappa': H^\bullet(F_b;\Z)\longrightarrow C^\bullet(F_b;\Z)$$ with $[\kappa'(x)]=x$. This is equivalent to assigning to each element in $\Lambda^k(T_\Z^*B)_b$ a linear combination of closed $(n-k)$-forms $dx_{i_1}\wedge\dots\wedge dx_{i_{n-k}}$, where $\{x_1,\dots,x_n\}$ are coordinates on $B$.

Given such a choice, we define the cochain map $\hat{\psi}$ as the composition $$C^{\bullet}(B;\Z)\otimes H^{\bullet}(F_b;\Z)\stackrel{(id,\kappa')}{\longrightarrow}C^{\bullet}(B;\Z)\otimes C^{\bullet}(F_b;\Z)\stackrel{\text{quism}}{\longrightarrow}C^{\bullet}(B\times F_b;\Z)\stackrel{t^*}{\longrightarrow}C^{\bullet}(X(B);\Z).$$ 

Again one readily verifies that this is a cochain map (using that $\kappa'$ takes values in closed forms), and that the induced map on cohomology, which we denote $\psi$, does not depend on our choice of $\kappa'$.

Very concretely, let $\alpha$ be a $j$-form on $B$, and $\beta\in\Lambda^k (T_\Z^*B)_b$. Then using Equation \eqref{second_chain_identification}, $\beta$ can be represented by an $(n-k)$-form $\kappa'(\beta)$ on $F_b$. It follows that
\begin{equation}\label{cochain_map}
	\hat{\psi}(\beta\otimes\alpha)=t^*(\kappa'(\beta)\wedge\alpha) \in C^{j+n-k}(X(B);\Z).\notag
\end{equation}

\subsection{The pairing is preserved}

We prove that the tropical fluxes we compute on the base $B$ lift to symplectic fluxes in $X(B)$. It follows from the fact that the chains of isomorphisms in Equations \eqref{chain_of_identificiations} and \eqref{second_chain_identification} are dual to one another that the maps $\phi$ and $\psi$ preserve the pairings in the sense that given $a$ in $H_j(B;\Lambda^k T_\Z B)$ and $\alpha$ in $H^j(B;\Lambda^kT_\Z^*B)$, $$\alpha(a)=\psi(\alpha)(\phi(a)).$$ Nevertheless, we need the stronger result that an integration pairing at the level of (co)chains is also preserved.

\begin{proposition}\label{pairing_preserved}
	Let $c\in C_j(B;\Lambda^kT_\Z B)$ and $\delta\in C^j(B;\Lambda^kT_\Z^*B)$. For any choices of $\kappa$ and $\kappa'$ as in Equations \eqref{map_homology} and \eqref{map_cohomology}, the corresponding maps $\hat{\phi}$ and $\hat{\psi}$ satisfty $$\int_c \delta=\int_{\hat{\phi}(c)}\hat{\psi}(\delta).$$
\end{proposition}
\begin{proof}
	Consider an element of $C_j(B;\Lambda^kT_\Z B)$ of the form $\beta\otimes\sigma$, where $\sigma$ is a singular $j$-chain in $B$, and $\beta$ is viewed as an element of $H_{n-k}(F_b;\Z)$, to which $\kappa$ associates a choice of representative by parallelepipeds in the fibre direction as in Section \ref{map_homology}. Similarly, let $\alpha\otimes\eta$ be an element in $C^j(B;\Lambda^kT_\Z^*B)$, where $\eta$ is a $j$-form on $B$, and $\alpha$ is viewed as an element of $H^{n-k}(F_b;\Z)$, to which $\kappa'$ associates a particular choice of closed $(n-k)$-form on the fibre $F_b$.
 	The tropical pairing of $\alpha\otimes\eta$ with $\beta\otimes\sigma$ is $$\beta(\alpha)\int_\sigma\eta.$$ From fiberwise integration over $E_\beta$, it follows that this coincides with the singular pairing $$\int_{E_\beta}\beta\wedge\eta.$$ Here we have reused the notation from Section \ref{map_homology}, where $E_\beta\in C_{j+n-k}(B\times F_b;\Z)$ is the product of $\sigma$ with the quotient by $(T_\Z^*B)_b$ of the parallelepipeds in $(T^*B)_b$. 
 	It is also true that the pairing $\alpha(\beta)$ of $(n-k)$ (co)homology classes, coincides with the pairing between $\Lambda^kT_\Z^*B$ and $\Lambda^kT_\Z B$ through the identifications from Equations \eqref{chain_of_identificiations} and \eqref{second_chain_identification}. These are dual through the canonical isomorphism $$\Lambda^kT_\Z^*B\cong(\Lambda^kT_\Z B)^*,$$ therefore the pairing is preserved. 
 	One readily verifies this is independent of the choice of $\kappa$ and $\kappa'$. 
\end{proof}

\subsection{Image of the determinantal forms and tautological tropical cycles}

Here we show that determinantal forms in $C^\bullet(B;\Lambda^k T^*_\Z B)$ are mapped through $\hat{\psi}$ to differential forms in $C^{\bullet+n-k}(X(B);\Z)$ that precisely detect flux-type obstructions to the existence of Lagrangian cobordisms in the sense of Section \ref{symplectic_flux}, and that tautological tropical cycles are mapped through $\hat{\phi}$ to something resembling a conormal bundle in $X(B)$.
Once again, we fix a basis $e_1,\dots,e_n$ of $\Gamma_2$, and the corresponding coordinates $x_1,\dots,x_n$ (unique up to translation) on $B$. Let us prescribe a map $\kappa'$ by setting
\begin{align}
	\kappa': H^\bullet(F_b;\Z)\hspace{1mm}\longrightarrow &\hspace{2mm}C^\bullet(F_b;\Z) \notag \\
	[dp_{j_1}\wedge\dots\wedge dp_{j_r}]&\mapsto dp_{j_1}\wedge\dots\wedge dp_{j_r}, \notag
\end{align}
where $p_i$ is the cotangent coordinate associated to $x_i$ in $X(B)$. In particular, composing with the isomorphism $\Lambda^k (T_\Z^*B)_b\cong H^{n-k}(F_b;\Z)$, elements of the form $$e_{j_1}^*\wedge\dots\wedge e_{j_k}^*$$ are mapped to the differential form $$(-1)^{sgn(\mu)}dp_{j_{k+1}}\wedge\dots\wedge dp_{j_{n}},$$ where $\mu$ is the permutation taking $\{1,\dots,n\}$ to $\{j_1,\dots,j_k,j_{k+1},\dots,j_{n}\}$. 
 
Notice that in each term appearing in Definition \ref{determinantal_forms}, the $n-k$ indices $\{j_{k+1},\dots,j_{n}\}$ would correspond to the union of the $j$ indices $\{s_1,\dots,s_j\}$ and the $n-(j+k)$ (ordered) indices $$\{i_{j+k+1},\dots,i_n\}:=\{1,\dots,n\}\backslash\{i_1,\dots,i_{j+k}\}\cong\{1,\dots,n\}\backslash\{j_1,\dots,j_k,s_1,\dots,s_j\}.$$ Without loss of generality, we can therefore order $\{j_{k+1},\dots,j_{n}\}$ as $\{i_{j+k+1},\dots,i_n, s_1,\dots,s_j\}$.

We denote the wedge product $dp_{j_{k+1}}\wedge\dots\wedge dp_{j_{n-j}}$ by $vol_{n-(j+k)}$, or $vol_{n-(j+k)}^{(i_1,\dots,i_{j+k})}$ if we need to emphasize the dependence on the indices $\{i_1,\dots,i_{j+k}\}$. With these notations we have $$(-1)^{sgn(\mu)}dp_{j_{k+1}}\wedge\dots\wedge dp_{j_{n}}= (-1)^{sgn(\mu)} vol_{n-(j+k)}\wedge dp_{s_1}\wedge\dots\wedge dp_{s_j}.$$

The following proposition then follows from a straightforward verification.

\begin{proposition}
	With the above notations and conventions, we have $$\hat{\psi}(\Omega^n_{j,k})=\sum_{(i_1<\dots<i_{j+k})\in\{1,\dots,n\}^{j+k}}(-1)^{sgn(\eta)+1+j(n-k)}\frac{1}{j!}vol_{n-(j+k)}^{(i_1,\dots,i_{j+k})}\wedge\w^j\in C^{j+n-k}(X(B);\Z),$$
	where $\eta$ is the permutation taking $\{1,\dots,n\}$ to $\{i_1,\dots,i_n\}$.
\end{proposition}

\begin{example}\label{image_of_Zharkov_determinantal_form}
	With this prescribed map, the image of Zharkov's determinantal form $$\Omega_{2,1}^3:=e_1^*\otimes dx_2\wedge dx_3 - e_2^*\otimes dx_1\wedge dx_2+e_3^*\otimes dx_1\wedge dx_2$$ is simply $\hat{\psi}=\frac{1}{2}\w^2$ in $X(J(C))$.
\end{example}

We can similarly prescribe a map $\kappa$ as
\begin{align}
	\kappa: H_\bullet(F_b;\Z)\hspace{1mm}\longrightarrow &\hspace{2mm}C_\bullet(F_b;\Z) \notag \\
	[e^*_{j_1}\wedge\dots\wedge e^*_{j_r}]\mapsto&\hspace{1mm} e^*_{j_1}\wedge\dots\wedge e^*_{j_r}. \notag
\end{align}
Given an algebraic $k$-cycle on the base, every $k$-cell $\sigma$ in its support endowed with its tautological framing is mapped through $\hat{\phi}$ to the quotient by the lattice $T^*_\Z B$ of the conormal bundle over $\sigma$. This follows from the isomorphism $$\Lambda^k T_\Z B\cong\Lambda^{n-k}T^*_\Z B,$$ which maps $v_{i_1}\wedge\dots\wedge v_{i_k}\in \Lambda^k T_\Z B$ to $(-1)^{sgn(\mu)}v^*_{i_{k+1}}\wedge\dots\wedge v^*_{i_n}\in\Lambda^{n-k}T_\Z^*B$, where $\mu$ is the permutation $\mu(\{1,\dots,n\})=\{i_1,\dots,i_n\}$. This feature is crucial to our construction because it establishes a close relation between the image through $\hat{\phi}$ of $k$-cycles on the base, and their Lagrangian lifts in $X(B)$.

\begin{example}
	 Consider a $1$-cell $e$ in $B$ with $\partial e= b_1-b_2$; this defines an element in $C_1(B;\Z)$ whose image through $\hat{\phi}$ is an $(n+1)$-cycle $\hat{\phi}(e)\cong T^n\times[0,1]$ with boundary $F_{b_1}-F_{b_2}$. In this situation, the appropriate determinantal form is $$\Omega^n_{1,0}=\sum_{i\in\{1,\dots,n\}}dx_i,$$ which is mapped through $\hat{\psi}$ to $$\sum_{i\in\{1,\dots,n\}}vol_{n-1}^i\wedge\w,$$ where $vol_{n-1}^i=dp_1\wedge\dots\wedge dp_{i-1}\wedge dp_{i+1}\wedge\dots\wedge dp_n$.
	 Because $e$ is tropical, a vector $v=(v_1,\dots,v_n)$ in $TB$ tangent to $e$ has rational slope. Therefore, we can (non-canonically) decompose the base $B\cong T^n$ as $T^{n-1}\times S^1$ in such a way that $v$ tangent to the $S^1$-direction. We denote by $\theta$ the coordinate on this $S^1$-factor. Then the tropical pairing $\int_e\Omega^n_{1,0}$ corresponds precisely to $\int_{(T^{n-1}\times S^1)\times[0,1]}vol_{n-1}\wedge d\theta\wedge dt,$ with $t$ being the coordinate along $[0,1]$ and $vol_{n-1}$ a volume form on the $T^{n-1}$-component; this matches the symplectic flux used in Proposition \ref{cotangent_torus_flux}.
\end{example}

\begin{remark}
	From Remark \ref{no_trop_flux_0_cycles}, we know that tropical flux does not obstruct algebraic equivalence for $0$-cycles. It does, however, yield symplectic flux obstructing Lagrangian cobordisms between the corresponding fibres.
	In the case of Remark \ref{no_trop_flux_0_cycles}, the corresponding symplectic flux between $F_p$ and $F_q$ in $X(S^1)\cong T^2$ is  $$\int_C\w,$$ where $C$ is a $2$-cylinder with $\partial C=F_p-F_q$. It implies that $F_p$ and $F_q$ are not oriented Lagrangian cobordant unless $p=q$.
\end{remark}

\section{The Lagrangian Ceresa cycle}\label{Lag_Ceresa_section}

Recall the tropical Ceresa cycle of a tropical curve $C$ with basepoint $p\in C$ is $$AJ_p(C)-(-1)_*AJ_p(C)\in \Gr_1(J(C)),$$ and when there is no ambiguity from the choice of $p$ we often abuse notation and simply write $C-C^-$. If they exist, we denote by $L_C$ the Lagrangian lift of $C$, and by $L_C^-$ the Lagrangian lift of $C^-$. We then define the \textit{Lagrangian Ceresa cycle} as the Lagrangian $L_C-L_C^-$ inside $X(B)$. As mentioned in the introduction, for $C$ hyperelliptic, choosing $p$ to be a Weierstrass point implies $AJ_p(C)=(-1)_*AJ_p(C)$, therefore $L_C=L_C^-$. This means that the simplest non-trivial case to study the Lagrangian Ceresa is a non-hyperelliptic curve $C$ with $g(C)=3$, and here we restrict ourselves to the type $K4$ curve from Figure \ref{K4_basis}. 

 \begin{figure}[h!]
	\centering
	\includegraphics[width=2in]{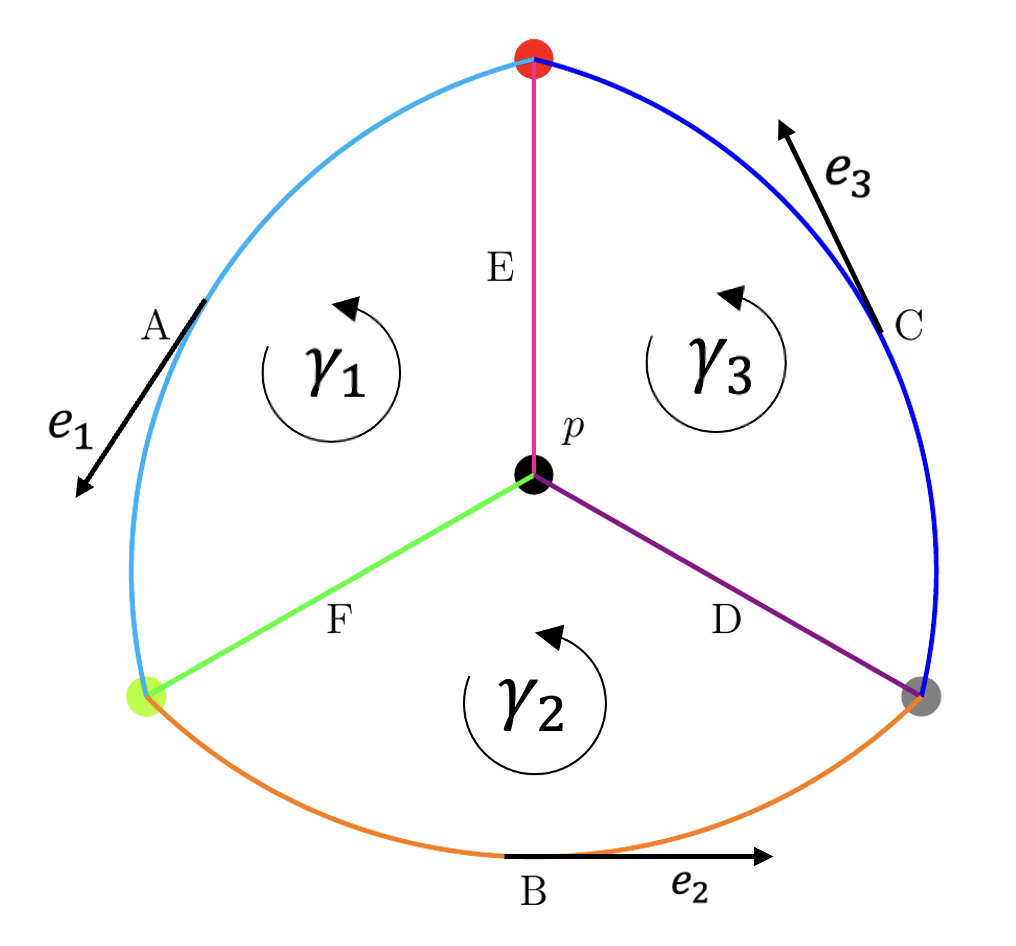} 
	\caption{A curve of type $K4$ with basis $\{e_1,e_2,e_3\}$ of $T_\Z C$ and $\{\gamma_1,\gamma_2,\gamma_3\}$ of $H_1(C;\Z)$.}
	\label{K4_basis}
\end{figure}

\subsection{The polarised symplectic torus $X(J(C))$}

Generally, given any tropical torux $B=B(Q)$ for some matrix $Q\in GL(n;\R)$, the periods of the symplectic form on  $X(B):=T^*B/T^*_\Z B$ are given by the entries of the matrix $Q$. To see this, consider global coordinates $\{x_1,x_2,x_3,p_1,p_2,p_3\}$ on $X(B)$, where the $p_i$ are cotangent coordinates associated to the $x_i$. The nine non-vanishing periods of $\w$ correspond to integrals over generators of $$H_1(B;\Z)\otimes H_1(F_b;\Z)\subset H_2(X(B);\Z),$$ where $F_b$ is a torus fibre at any point $b\in B$, and we have implicitly used the isomorphism $X(B)\cong B\times F_b$. It follows from the construction of $B(Q)$ that generators of $H_1(F_b;\Z)$ can be identified with unit vectors $\{e_1,e_2,e_3\}$ in $T^*_\Z B\cong\R^3$, while generators of $H_1(B;\Z)$ are identified with $\{Q\cdot e_1,Q\cdot e_2,Q\cdot e_3\}=\{\gamma_1,\gamma_2,\gamma_3\}$. The integral of  $\w$ over the generator $Q\cdot e_i\otimes e_j$ of $H_1(B;\Z)\otimes H_1(F_b;\Z)$ is thus the symplectic area of the surface in $X(J(C))$ tangent to the vectors $(Q\cdot e_i,0,0,0)$ and $(0,0,0,e_j)$, which is simply $(Q\cdot e_i)_j$.

In the particular case where $B=J(C)$ is the Jacobian of a tropical curve, Remark \ref{periods_jacobians_polarisation_form} tells us that the matrix $Q$ is the polarisation matrix of the curve. In particular, it is symmetric and positive-definite. For the $K4$ curve from Figure \ref{K4_basis}, the periods of the symplectic form on $X(J(C))$ are therefore integral combinations of the six different entries of the polarisation matrix 
\begin{equation}
	Q=	\begin{pmatrix}
		a+e+f & -f & -e \\ -f & b+d+f & -d \\ -e & -d & c+d+e
	\end{pmatrix}.\notag
\end{equation}
Explicitly, in terms of the dual basis $\{\gamma_1^*,\gamma_2^*,\gamma_3^*\}$ of $H^1(B;\Z)$ and $\{e_1^*,e_2^*,e_3^*\}$ of $H^1(F_b;\Z)$, the class of the symplectic form is
\begin{align}
	[\w]=&(a+e+f)\gamma_1^*\otimes e_1^*+(b+d+f)\gamma_2^*\otimes e_2^*+(c+d+e)\gamma_3^*\otimes e_3^* \notag \\&-f(\gamma_2^*\otimes e_1^*+\gamma_1^*+e_2^*)-e(\gamma_3^*\otimes e_1^*+\gamma_1^*\otimes e_3^*)-d(\gamma_2^*\otimes e_3^*+\gamma_3^*\otimes e_2^*)\in H^2(X(B);\R). \notag
\end{align}

\subsection{The $3$-manifold $L_C-L_C^-$}\label{Lag_ceresa_topology}

Because the $K4$ curve $C$ is trivalent, we know from constructions of Mikhalkin and Matessi \cite{matessi_lag_pair_of_pants,mikhalkin_trop_to_Lag_corresp} that it admits a Lagrangian lift $L_C$ to $X(J(C))$. Therefore $L_C-L_C^-$ exists, and is a Lagrangian $3$-manifold inside a symplectic $6$-torus $X(J(C))$. 

Such a lift is not unique; we construct one after choosing small open balls $U_i$ around the vertices of $C$. Outside these open subsets around each vertex, the edges simply lift to $(0,1)\times T^2$ topologically. Because the balancing condition implies that each trivalent vertex is contained locally in a $2$-plane in an open subset of $\R^3$, the Lagrangian lift around a trivalent vertex is topologically $S^1\times S^2_3$, where $S^2_3$ denotes a $2$-sphere with three punctures. 

Figure \ref{Lag_lift_oh} shows a part of the Lagrangian lift of $C$, namely the part corresponding to the generator $\gamma_1$ of $H_1(C;\Z)$. The rest can be assembled in the same way. We have chosen coordinates ${x_1,x_2,x_3}$ on $J(C)$ induced by the vectors $\{e_1,e_2,e_3\}$ from Figure \ref{K4_basis}, and ${p_1,p_2,p_3}$ their associated cotangent coordinates in $T^*(J(C))$. Then $\{x_1,x_2,x_3,p_1,p_2,p_3\}$ canonically induce global coordinates for $X(J(C))\cong T^6$. 

 \begin{figure}[htb]
	\centering
	\includegraphics[width=6.6in]{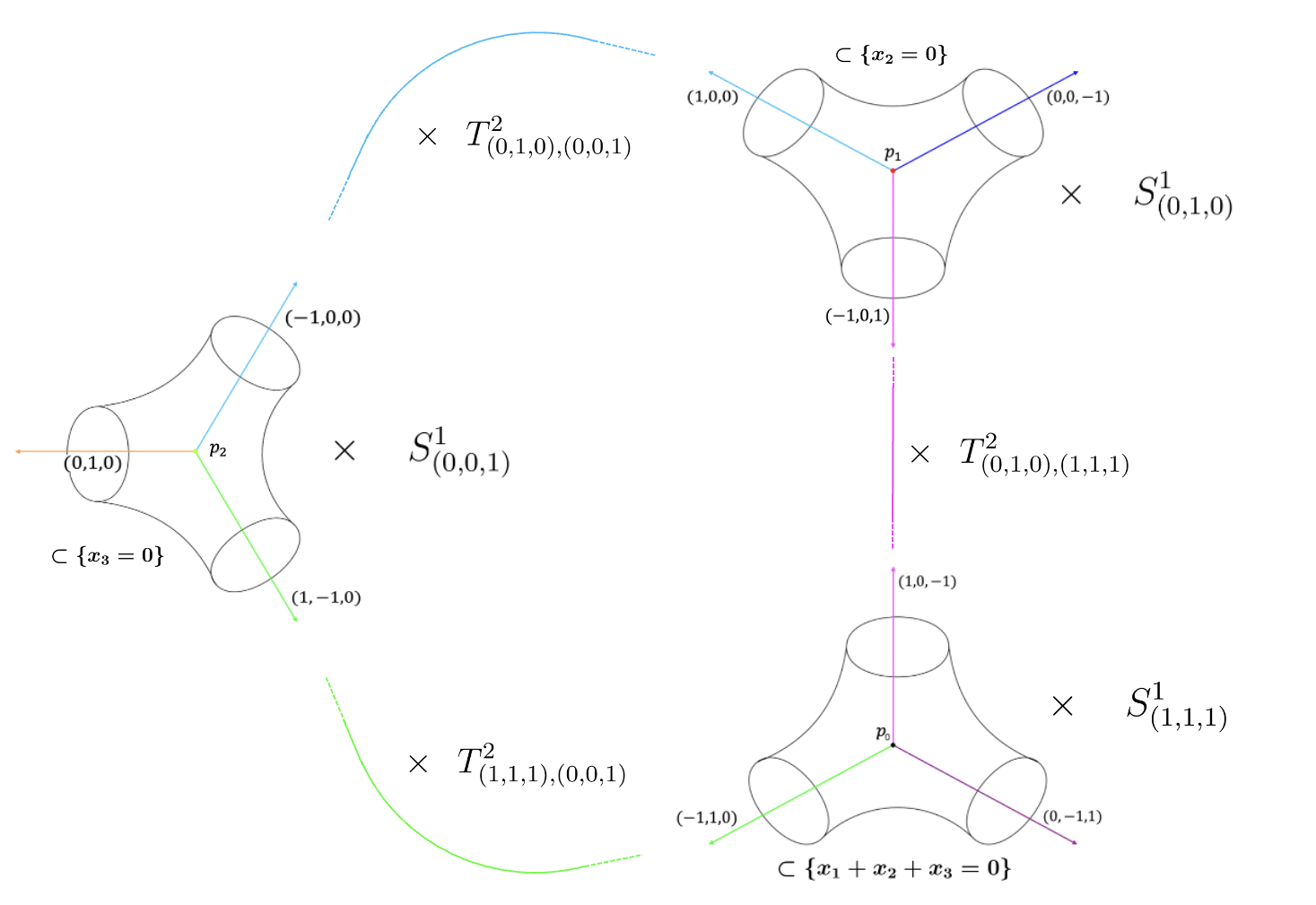} 
	\caption{A representation of a part of the Lagrangian lift of $C$ inside $X(B)\cong B\times T^3$. Each of the pairs of pants is associated to a trivalent vertex contained in a $2$-plane inside $B$ indicated on the figure. $S^1_{(p_1,p_2,p_3 )}$ denotes the $S^1$-factor in the fibre direction generated by the line $(p_1,p_2,p_3)$ in $(T^*B)_b$. Similarly, $T^2_{(p_1,p_2,p_3),(q_1,q_2,q_3)}$ denotes the $2$-torus in a fibre generated by the linear span of the vectors $(p_1,p_2,p_3)$ and $(q_1,q_2,q_3)$ in $(T^*B)_b$.}
	\label{Lag_lift_oh}
\end{figure}

The homology of $L_C$ can be computed by iteratively applying Mayer-Vietoris using the cover represented in Figure \ref{Lag_lift_oh}, from which one finds $b_1(L_C)=b_2(L_C)=6$ and $H_3(L_C)\cong H_0(L_C)\cong\Z$. In particular, $L_C$ is orientable. 
% Notice that the generators of the first homology of each open set can be represented by loops in the cotangent directions - that is, loops which are constant along the base direction in the decomposition $X(J(C))\cong J(C)\times T^3$. Because these generators are generators of $(T_\Z J(C))$ which has no monodromy, from the choice of $3$ generators in $H_1(T^3)$ one can assign orientations to all of the degree $1$ generators appearing in the homology computation. 

The Lagrangian lift $L_C^-$ of $C_C^-$ is simply the image of $L_C$ by a global symplectomorphism. Indeed, $C^-$ is the image of $C$ by the automorphism $(-1)$ of $B$. Because $(-1)$ is tropical, the induced symplectomorphism on $T^*J(C)$ yields a symplectomorphism $\varphi^-$ of $T^*J(C)/T^*_\Z J(C)$. Notice $L_C$ and $L_C^-$ are not disjoint, therefore the Lagrangian Ceresa cycle is only immersed. 

Although we know the Ceresa cycle is nullhomologous (see Remark \ref{Ceresa_nullhomologous}), it may not yet be clear that the Lagrangian Ceresa cycle is. This fact will follow from our proof of Theorem \ref{symplectic_Ceresa} (see Remark \ref{i_swear_this_theorem_is_not_trivial}), nevertheless we can already exhibit some supporting evidence. Tropically, $(-1)$ acts as $-\Id$ on $H_1(J(C);\Z)$ and $T_\Z^*J(C)\cong H_2(F;\Z)$, therefore it acts trivially on $H_1(J(C))\otimes H_2(F;\Z)$. We know this homology group contains $\hat{\phi}(\overline{C})$, and it will follow from the proof of Theorem \ref{symplectic_Ceresa} that it also contains $L_C$.

\subsection{Main result}

Here we prove the following Theorem:
\begin{theorem}\label{symplectic_Ceresa}
	For a generic tropical curve $C$ of type $K4$ with $g(C)=3$, the Lagrangian Ceresa cycle $L_C-L_C^-\in\algcob^{or}(X(J(C)))_{hom}$ has infinite order.
\end{theorem}

From which the following Corollary is immediate from Remark \ref{cob_implies_algcob}:

\begin{corollary}\label{symplectic_Ceresa_weaker_cobordism_statement}
	For a generic tropical curve $C$ of type $K4$ with $g(C)=3$, the Lagrangian Ceresa cycle $L_C-L_C^-\in\Cob^{or}(X(J(C)))_{hom}$ has infinite order.
\end{corollary}

This result builds on Zharkov's proof of the corresponding statement about the tropical Ceresa cycle $C-C^-\in \Gr_1(J(C))$ being non-torsion (Theorem \ref{Zharkov_theorem}). There, Zharkov constructs a tropical $2$-chain $\gamma_0$ in $C_2(J(C),T_\Z J(C))$ satisfying $\partial\gamma_0=\overline{C}-\overline{C^-}$. He shows that $$\int_{\gamma_0}\Omega^3_{2,1}=-ad,$$ which for a generic choice of edge lengths $\{a,b,c,d,e,f\}$ is not torsion modulo the period lattice computed in Example \ref{periods_computation}.

To be more precise, he constructs a $2$-chain $\gamma_0'$ in $C_2(J(C);T_\Z J(C))$ between $AJ_p(C)$ and $AJ_{p'}(C)^-$ in $J(C)$, where $p$ and $p'$ are chosen so that the images of the edge $C$ in $AJ_p(C)$ and $AJ_{p'}(C)^-$ coincide. For instance, take $p=p_1$ and $p'=p_3$ (see Figure \ref{K4_basis}). 
Notice that $AJ_p(C)$ and $AJ_{p'}(C)$ are algebraically equivalent for any $p$ and $p'$ in $C$: this is realised by the equivalence $$W:=\{p,AJ_p(C)\}\subset C\times J(C),$$ which one verifies is a tropical cycle.

With this in mind, when we refer to the tropical chain $\gamma_0$, we will really mean $\tilde{W}\cup\gamma'_0$, where $\tilde{W}\in C_2(J(C);T_\Z J(C))$ is the tropical chain constructed in \ref{construction_tropical_chain_equivalence} from $W$, which satisfies $$\int_{\tilde{W}}\Omega_0=0$$ by Lemma \ref{determinantal_form_vanishes}.
Therefore $$\int_{\gamma_0:=\gamma'_0\cup\tilde{W}}\Omega_0=\int_{\gamma'_0}\Omega_0.$$

The tropical chain $\gamma'_0\in C_2(J(C);T_\Z J(C))$ is supported on the following $5$ parallelograms in $J(C)$:

\begin{itemize}
	\item $E_1$ with edges $AJ_p(A)$, $AJ_p(B)$, $(-1)_*AJ_{p'}(A)$, and $(-1)_*AJ_{p'}(B)$;
	\item $E_2$ with edges $(-1)_*AJ_{p'}(B)$, $(-1)_*AJ_{p'}(F)$, $(-1)AJ_{p'}(B)$ translated by $(-1)_*(AJ_{p'}(p_0)-AJ_{p'}(p_2))$, and $(-1)_*AJ_{p'}(F)$ translated by $(-1)_*(AJ_{p'}(p_2)-AJ_{p'}(p_3))$;
	\item $E_3$ with edges $AJ_p(B)$, $(-1)_*AJ_{p'}(E)$, $AJ_p(B)$ translated by $(-1)_*(AJ_{p'}(p_0)-AJ_{p'}(p_1))$, and $(-1)_*AJ_{p'}(E)$ translated by $AJ_p(p_2)-AJ_p(p_3)$. 
	\item $E_4$ with edges $AJ_p(E)$, $AJ_p(F)$, $AJ_p(E)$ translated by $AJ_p(p_2)-AJ_p(p_0)$, and $AJ_p(F)$ translated by $AJ_p(p_1)-AJ_p(p_0)$. 
	\item $E_5$ with edges $AJ_p(E)$, $(-1)_*AJ_{p'}(E)$, $AJ_p(D)$, and $(-1)_*AJ_{p'}(D)$.
\end{itemize}
These are parallelograms in $J(C)$ because we have chosen $p$ and $p'$ so that the image of $C$ by $AJ_p$ and $(-1)_*AJ_{p'}$ coincide, therefore cancel out in $C-C^-$.  Furthermore, there is a choice of orientation for these parallelepipeds which we do not describe here, but is found in Zharkov's original construction \cite{zharkov_tropical_ceresa}.
The $T_\Z J(C)$-coefficients are given as follows, in the basis $\{e_1,e_2,e_3\}$: $(-1,1,0)$ for $E_2$, $E_3$ and $E_4$, $(-1,0,0)$ for $E_1$, and $(0,1,-1)$ for $E_5$. One readily verifies (given the orientations) that $\partial\gamma'_0=\overline{AJ_p(C)}-\overline{(-1)_*AJ_{p'}(C)}$. Notice furthermore that the $T_\Z J(C)$-coefficients are tangent to the parallelograms for $E_1$, $E_2$, $E_4$ and $E_5$. Hence Proposition \ref{pre_determinantal_form_vanishes} implies that $$\int_{\gamma'_0}\Omega_0=\int_{E_3}\Omega_0=-ad,$$ where the minus sign follows from the orientation.

Recall from Example \ref{image_of_Zharkov_determinantal_form} that the determinantal form $\Omega^3_{2,1}$ is mapped through the appropriate choice of $\hat{\psi}$ to $$\frac{1}{2}\w^2\in \Omega^4(X(J(C))),$$ where $\w$ is the canonical symplectic form on $X(J(C))$ induced by the one on $T^*J(C)$. Therefore by Proposition \ref{pairing_preserved}, $$\int_{\gamma_0}\Omega^3_{2,1}=\frac{1}{2}\int_{\hat{\phi}(\gamma_0)}\w^2.$$

\begin{proof}[Proof of Theorem \ref{symplectic_Ceresa}]
	To prove Theorem \ref{symplectic_Ceresa}, we apply Lemma \ref{cob_implies_periods} with $\alpha:=\frac{1}{2}\w$, integrating $\frac{1}{2}\w^2$ over an $4$-chain between $L_C$ and $L_C^-$. Notice that $\hat{\phi}(\gamma_0)$ is a $4$-chain between $\hat{\phi}(\overline{C})$ and $\hat{\phi}(\overline{C^-})$, on which $\frac{1}{2}\w^2$ integrates to something which is not in its periods. 
	In Matessi's construction of the Lagrangian pair-of-pants, he shows (Corollary 3.19 in \cite{matessi_lag_pair_of_pants}) that they come in families which converge to his piecewise-linear Lagrangian lift. 
	By inspection of \cite[Section 2]{matessi_lag_pair_of_pants} and by Section \ref{image_of_Zharkov_determinantal_form}, the piecewise-linear Lagrangian lift of $C$ coincides with $\hat{\phi}(C)$.
	This family gives an $4$ chain $\gamma$ connecting $L_C$ and $\hat{\phi}(C)$, on which $\w^2$ necessarily vanishes (as on any $4$-chain given by a $1$-dimensional family of Lagrangians). The same argument yields an $4$-chain $\gamma^-$ connecting $L_C^-$ and $\hat{\phi}(C^-)$ on which $\w^2$ vanishes. Therefore $$\int_{\gamma\cup\hat{\phi}(\gamma_0)\cup\gamma^-}\frac{1}{2}\w^2=\int_{\hat{\phi}(\gamma_0)}\frac{1}{2}\w^2=\int_{\gamma_0}\Omega_{2,1}^3=-ad.$$ 
	 Because the maps $\phi$ and $\psi$ are isomorphisms, the periods of $\frac{1}{2}\w^2$ are exactly the periods of $\Omega^3_{2,1}$. 
	 Therefore for a generic choice of edge lengths, $-ad$ is not torsion modulo the periods of $\frac{1}{2}\w^2$, and we conclude.
\end{proof}

\begin{remark}\label{i_swear_this_theorem_is_not_trivial}
	While $\phi$ being an isomorphism already implied that $\hat{\phi}(C)$ and $\hat{\phi}(C^-)$ are homologous, the $(n+1)$-chains provided by Matessi's converging family of pairs-of-pants allows us to conclude that $L_C$ and $L_C^-$ are homologous in $X(J(C))$.
\end{remark}

\begin{remark}\label{remark_torsion_case}
	While for \textit{generic} choices of edge lengths, the flux $-ad$  is not torsion in $\mathcal{P}^3_{2,1}$, it is in special cases e.g. taking $a=d=e=f$. Although no analogous results have been brought to light tropically, classically \cite{Beauville_torsion_Ceresa_class,Beauville_torsion_Ceresa_cycle} exhibit an example of a non-hyperelliptic curve which nevertheless has torsion Ceresa cycle. 
\end{remark}

\begin{remark}\label{infinite_generators}
	It is worth noting that a few years after Ceresa proved the inifinite order statement for the Ceresa cycle in $\Gr_1(J(C))$, Bardelli proved that this group is in fact infinitely generated. He proceeds by constructing generators as follows: for every $k\geq1$, the cycles $C_k-C_k^-\in\Gr_1(J(C))$ are pushforwards of Ceresa cycles of curves $C_k$ for which $J(C_k)$ has a $k$-to-$1$ map to $J(C)$. It is expected that one could prove a similar statement symplectically, only with our methods this would require a tropical construction of these infinitely many generators, which to the best of our knowledge has not appeared in the literature. There has been, however, progress on the tropical Schottky problem; for instance an algorithm to produce a tropical curve $C_k$ from the data of its Jacobian (and a theta divisor) is suggested in \cite[Section 7]{bolognese2017curvestropicaljacobians}, and in \cite{Chua2017SchottkyAC}. 
\end{remark}

\subsection{Speculation for higher genus curves}

While we have thus far restricted to studying the genus three $K4$ graph, we expect the methods used to be able to produce statements about Lagrangian Ceresa cycles of any higher genus tropical curve containing $K4$ as a subgraph. 

\begin{conjecture}
	Let $C$ be a generic finite compact trivalent tropical curve of genus $\geq3$ containing $K4$ as a subgraph. Then $L_C-L_C^-\in \algcob^{or}(X(J(C)))$, its Lagrangian Ceresa cycle, has infinite order. 
\end{conjecture}

This Conjecture is motivated by the following tropical statement by Zharkov: 
\begin{theorem}\cite[Theorem 8]{zharkov_tropical_ceresa}
	Let $C$ be a generic tropical curve of genus $\geq 3$ containing $K4$ as a subgraph. Then $C$ is not algebraically equivalent to $C^-$ in $J(C)$. 
\end{theorem}
The argument is sketched as follows. Let $C$ be a genus $g\geq3$ curve containing $K4$ as a subgraph. We proceed by successively removing edges $e$ of $C$ to obtain a new curve $C_e$, until we are left with the $K4$ graph. Notice that this can be done in a way that at each iteration, $C_e$ is connected (up to removing edges which have constant image in the Jacobian). Furthermore, we only consider \textit{ compact} curves whose underlying graph is finite, therefore there are no $1$-valent vertices (this property is inherited by $C_e$ as two metric graphs are said to be \textit{equivalent} if one is obtained from another by inserting a $2$-valent vertex in the interior of an edge). 
With these precautions, we observe that after each iteration $g(C_e)=g(C)-1$. Zharkov proves \cite[Propositions 6. and 7.]{zharkov_tropical_ceresa} that this comes with a tropical projection map $J(C)\rightarrow J(C_e)$ along the direction of the edge $e$ in $J(C)$, which maps $C\subset J(C)$ to $C_e\subset J(C_e)$. 
Because an algebraic equivalence between $C$ and $C^-$ would survive this tropical projection map, iterating this process would eventually lead to a contradiction.

A natural approach to proving a symplectic version of this would be the following. The graph of the tropical projection map above is a cycle of dimension $g$ in $J(C)\times J(C_e)$, which induces a map  $$CH_i(J(C))\longrightarrow CH_i(J(C_e))$$ between tropical Chow groups. If we choose a set of coordinates $(x_1,\dots,x_g)$ for $J(C)$ such that the tangent vector to $e$ in $J(C)$ is parallel to $(0,\dots,0,1)$, then this $g$-cycle is $$\{(x_1,\dots,x_g,x_1,\dots,x_{g-1})\}\subset J(C)\times J(C_e).$$ 
Using the associated cotangent coordinates, this motivates us to consider its Lagrangian lift as a Lagrangian correspondence $$P:=\{(x_1,p_1,\dots,x_g,0,x_1,p_1,\dots,x_{g-1},p_{g-1})\}\subset \overline{X(J(C))}\times X(J(C_e)).$$ 

Our result would follow from the following Conjecture:

\begin{conjecture}\label{functoriality_algcob_conjecture}
	Let $Y$ be an oriented Lagrangian correspondence $Y\subset\overline{M}\times M'$ between symplectic manifolds $M$ and $M'$. Then $Y$ induces a well-defined map
	\begin{align}
		Y:&\hspace{1mm}\algcob(M)\longrightarrow\algcob(M')\notag \\
		& \hspace{9mm}L\hspace{10mm}\mapsto \hspace{7mm}Y(L). \notag
	\end{align} 
\end{conjecture}

One reason to expect and hope for this statement to be true, is that this would be the symplectic manifestation of the fact that algebraic equivalence is an \textit{adequate equivalence relation}, in particular it behaves well under correspondences. A more explicit piece of evidence is hinted in the following Lemma:

\begin{lemma}\label{functoriality_algcob_partialstatement}
	Let $Y$ be an oriented Lagrangian correspondence $Y\subset\overline{M}\times M'$ between symplectic manifold $M$ and $M'$ which satisfies $\pi_M(Y)=M$. Then $Y$ induces a well-defined map
	\begin{align}
		Y:&\hspace{1mm}\algcob(M)\longrightarrow\algcob(M')\notag \\
		& \hspace{9mm}L\hspace{10mm}\mapsto \hspace{7mm}Y(L). \notag
	\end{align} 
\end{lemma}
\begin{proof}
	Let $L_1$ and $L_2$ be oriented Lagrangians in $M$. Assume there is an oriented Lagrangian correspondence $$\Gamma:L_1\stackrel{A_{p,q}}{\rightsquigarrow}L_2,$$ and set $L'_i:=Y(L_i)$ for $i=1,2$.
	Denote by $\Delta_A\subset \overline{A}\times A$ be the diagonal, and by 
	\begin{align}
		\varphi_{1324}:&\hspace{1mm}M\times M'\times A\times A\longrightarrow M\times A\times M'\times A \notag \\
		&\hspace{3mm}(x_1,x_2,x_3,x_4)\hspace{4.5mm}\mapsto\hspace{4mm}(x_1,x_3,x_2,x_4)\notag
	\end{align}
	the map interchanging the factors. Then $\varphi_{1324}(Y\times\Delta_A)$ is an oriented Lagrangian correspondence between $M\times A$ and $M'\times A$. Define $\Gamma'\subset M'\times A$ as the image of $\Gamma$ under this correspondence. Observe
	\begin{align}
		\Gamma'(F_q)&=\pi_{M'}(\Gamma'\pitchfork(M'\times F_q))=\pi_{M'}(\varphi_{1324}(Y\times\Delta_A)\pitchfork(M'\times F_q))\notag \\
		&=\pi_{M'}\left(\pi_{M'\times A}(\varphi_{1324}(Y\times\Delta_A)\pitchfork(\Gamma\times M'\times A)\pitchfork(M'\times F_q))\right)\notag \\ 
		&= \pi_{M'}\left(\varphi_{1324}(Y\times\Delta_A)\pitchfork(\Gamma\times M'\times F_q)\right) = L'_2. \notag
	\end{align}
	Similarly, we have $\Gamma'(F_p)=L_1'$. 
	Finally, we verify that $\pi_B(\Gamma')=B$, in other words that $\forall b\in B$, $\Gamma'(F_b)$ is non-empty. This can be rephrased as the requirement that for all $b\in B$, there should exist an $x\in M$, an $x'\in M'$, and an $f\in F_b$ such that $(x,x')\in Y$ and $(x,f)\in\Gamma$. We know that $\pi_B(\Gamma)=B$, therefore we are guaranteed the existence of some $x\in M$ and $f\in F_b$ such that $(x,f)\in \Gamma$. The existence of an associated $x'\in M'$ such that $(x,x')\in Y$ follows from our condition that $\pi_M(Y)=M$. 
\end{proof}

Here the condition that $\pi_M(Y)=M$ is used to ensure that $\pi_B(\Gamma')=B$. Very generally, we expect that this condition can be relaxed while still yielding an equivalent definition for algebraic Lagrangian cobordisms. Recall from Section \ref{def_alg_lag_cob} that the classical notion of algebraic equivalence can be equivalently formulated in terms of cycles living in flat families over a curve, an abelian variety, or any smooth projective variety \cite{parameter_spaces_alg_equiv}. It is reasonable to expect a similar statement tropically, and by extension for our definition of algebraic Lagrangian cobordisms to be equivalent to the existence of a Lagrangian correspondence $\Gamma'$ supported on a tropical curve for instance, as in Lemma \ref{algcob_yields_Lagcorresp_over_curve}.

Notice that in our specific setting of the correspondence $P\subset\overline{X(J(C))}\times X(J(C_e))$ above, the condition $\pi_B(P)$ fails to be satisfied in a very tangible way. Namely, letting $H:=\{p_g=0\}$ be the hyperplane $p_g=0$ in $X(J(C))$,  $$\pi_B(P)=\pi_B(\Gamma\cap(H\times A)),$$ which generically has dimension $\dim(B)+g-1\geq\dim(B)$.

\section{Some additional remarks}\label{further_remarks}

\subsection{Tropical curves of hyperelliptic type}\label{hyperelliptic_types_section}

The failure of the tropical Torelli theorem implies that some non-hyperelliptic curves can have a Jacobian isomorphic to that of a hyperelliptic curve; these are said to be of \textit{hyperelliptic type}. Readers are referred to \cite{Chan_tropical_hyperelliptic_curves} for background on tropical hyperelliptic curves, and \cite{Corey_tropical_curves_hyperelliptic_type} for background on tropical curves of hyperelliptic type. In the latter, it is shown that a curve being of hyperelliptic type depends only on its underlying graph. 

Throughout this paper, we have been concerned with a genus $3$ tropical curve of type $K4$ as represented in \ref{K4}. There are however, four additional types of generic genus $3$ tropical curves; see Figure \ref{hyperelliptic_types}. By generic we mean that all other genus three tropical curves can be obtained as degenerations of these (in particular they are trivalent). 

 \begin{figure}[htb]
	\centering
	\includegraphics[width=4.8in]{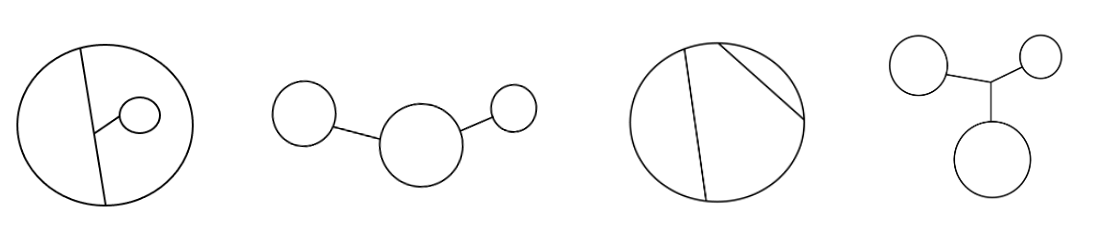} 
	\caption{The four generic types of genus three tropical curves of hyperelliptic type.}
	\label{hyperelliptic_types}
\end{figure}

The curves from Figure \ref{hyperelliptic_types} are of hyperelliptic type. Figure \ref{hyperelliptic_nonhyperelliptic} gives an example of two curves which have isomorphic Jacobians, yet the one on the left is hyperelliptic while the one on the right isn't. 

 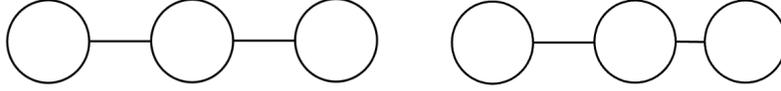
\begin{figure}[htb]
	\centering
	\begin{tikzpicture}[scale=0.6]
		\pgfdeclarelayer{nodelayer} 
		\pgfdeclarelayer{edgelayer}
		\pgfsetlayers{main,nodelayer,edgelayer}
		\begin{pgfonlayer}{nodelayer}
			\node (0) at (-12, 1) {};
			\node (1) at (-13, 0) {};
			\node (2) at (-12, -1) {};
			\node (3) at (-11, 0) {};
			\node (4) at (-8, 1) {};
			\node (5) at (-9, 0) {};
			\node (6) at (-8, -1) {};
			\node (7) at (-7, 0) {};
			\node (8) at (-4, 1) {};
			\node (9) at (-5, 0) {};
			\node (10) at (-4, -1) {};
			\node (11) at (-3, 0) {};
			\node (12) at (4, 1) {};
			\node (13) at (3, 0) {};
			\node (14) at (4, -1) {};
			\node (15) at (5, 0) {};
			\node (16) at (9.75, 1) {};
			\node (17) at (8.75, 0) {};
			\node (18) at (9.75, -1) {};
			\node (19) at (10.75, 0) {};
			\node (20) at (12, 1) {};
			\node (21) at (11, 0) {};
			\node (22) at (12, -1) {};
			\node (23) at (13, 0) {};
		\end{pgfonlayer}
		\begin{pgfonlayer}{edgelayer}
			\draw [thick, bend left=45] (0.center) to (3.center);
			\draw [thick, bend left=45] (3.center) to (2.center);
			\draw [thick, bend left=45] (1.center) to (0.center);
			\draw [thick, bend right=45] (1.center) to (2.center);
			\draw [thick, bend left=45] (4.center) to (7.center);
			\draw [thick, bend left=45] (7.center) to (6.center);
			\draw [thick, bend left=45] (5.center) to (4.center);
			\draw [thick, bend right=45] (5.center) to (6.center);
			\draw [thick, bend left=45] (8.center) to (11.center);
			\draw [thick, bend left=45] (11.center) to (10.center);
			\draw [thick, bend left=45] (9.center) to (8.center);
			\draw [thick, bend right=45] (9.center) to (10.center);
			\draw [thick, bend left=45] (12.center) to (15.center);
			\draw [thick, bend left=45] (15.center) to (14.center);
			\draw [thick, bend left=45] (13.center) to (12.center);
			\draw [thick, bend right=45] (13.center) to (14.center);
			\draw [thick, bend left=45] (16.center) to (19.center);
			\draw [thick, bend left=45] (19.center) to (18.center);
			\draw [thick, bend left=45] (17.center) to (16.center);
			\draw [thick, bend right=45] (17.center) to (18.center);
			\draw [thick, bend left=45] (20.center) to (23.center);
			\draw [thick, bend left=45] (23.center) to (22.center);
			\draw [thick, bend left=45] (21.center) to (20.center);
			\draw [thick, bend right=45] (21.center) to (22.center);
			\draw [thick] (3.center) to (5.center);
			\draw [thick] (7.center) to (9.center);
			\draw [thick] (15.center) to (17.center);
			\draw [thick] (19.center) to (21.center);
		\end{pgfonlayer}
	\end{tikzpicture}
	\caption{Two curves with isomorphic Jacobians, however the left one is hyperelliptic and the right one isn't.}
	\label{hyperelliptic_nonhyperelliptic}
\end{figure}

Recently, C. Ritter has shown \cite[Theorem A.]{ceresa_period_tropical} that, if $C$ is any genus $g$ tropical curve, the $\Omega^g_{2,1}$-flux obstruction for its Ceresa cycle $C-C^-$ vanishes if and only if $C$ is of hyperelliptic type. In particular, this implies that the symplectic flux vanishes for the Lagrangian Ceresa cycles associated to curves of hyperelliptic type. On the other hand, given a nullhomologous Lagrangian $L$ in $M$ with \enquote{zero flux}, i.e. whose image $\Theta(L)$ in the Lefschetz Jacobian $LJ(M)$ introduced in Section \ref{Lefschetz_jacobian} is zero, there is no easy way of 

In fact, in cases where mirror symmetry holds, one would expect rational equivalence to be mirror to \textit{unobstructed} (cylindrical) Lagrangian cobordism. We note that while flux detects when Lagrangians are not \textit{oriented} nullcobordant, it may not detect when Lagrangians are oriented cobordant, but not \textit{unobstructed} cobordant. 

Classically in algebraic geometry, for a nullhomologous cycle to have \enquote{zero flux} is manifested in the fact that its image in the corresponding intermediate Jacobian is zero. However it is an open question whether 
 However, the mirror correspondence between rational equivalence and Lagrangian cobordisms suggests that they are \textit{not} generally nullcobordant. These limitations of flux as an obstruction to Lagrangian cobordisms could already be noticed from the fact that it does not allow us to detect Lagrangians which are oriented cobordant, but not \textit{unobstructed} cobordant.

\subsection{Flux and characters of symplectic Torelli-type groups}

A concept closely related to our notion of symplectic flux appeared in a paper by Reznikov \cite[Section 4]{Reznikov_characteristic_classes_in_symplectic_topology}, where it is used to construct characters of a symplectic Torelli-type group $\Symp_{0}(M)/\ham(M)$. Here $\Symp_{0}(M)$ is the kernel of the map $$\Symp(M)\longrightarrow \Aut(H_{odd}(M;\Z)),$$ and $\ham(M)\subset \Symp(M)$ is the subgroup of Hamiltonian isotopies of $M$. 

Given $f\in\Symp_{0}(M)$ and $z\in H_{2k-1}(M;\Z)$, define $$\chi(f,z):=\int_b\w^k\mod\Z,$$ with $b$ is a $2k$-chain satisfying $\partial b=c-f(c)$, and $c$ a chain representing $z$. This yields a well-defined map \cite[Theorem 4.1]{Reznikov_characteristic_classes_in_symplectic_topology}:
\begin{align}
	\chi:\Symp_{0}(M)/\ham(M)\longrightarrow\hom(H_{odd}(M;\Z),\R/\left([\w^k]\cdot H_{2k}(M;\Z)\right). \notag
\end{align}

Now let $M$ be a symplectic manifold of dimension $2(2k-1)$, and fix a class $z\in H_{2k-1}(M;\Z)$ which admits a Lagrangian representative $L$. Then $$\chi(\cdot, z):\Symp_{0}(M)/\ham(M)\longrightarrow\R/\left([\w^k]\cdot H_{2k}(M;\Z)\right)$$ fits into a commutative diagram

\begin{tikzcd}
	\Symp_{0}(M)/\ham(M) \arrow[rrdd, "{\chi( \cdot ,z)}"'] \arrow[rrrr, "\varphi\mapsto L-\varphi(L)"] &  &                                           &  & Cob^{or}(M)_{hom} \arrow[lldd, "\w^k- flux"] \\
	&  &                                           &  &                                                          \\
	&  & \R/\left([\w^k]\cdot H_{2k}(M;\Z)\right). &  &                                                         
\end{tikzcd}

From this perspective, Theorem \ref{symplectic_Ceresa} can be translated into a non-triviality result for Reznikov's character in $X(J(C))$. Indeed, recall that $L^-=\varphi^-(L)$ for some global symplectomorphism $\varphi^-$ which is in $\Symp_{hom}(X(J(C)))$. However, it is important to notice that in general the horizontal map from $\Symp_{hom}(M)/\ham(M)$ to $Cob^{or}(M)_{hom}$ is \textit{not} a group homomorphism, but a \textit{crossed} homomorphism satisfying $$\varphi_2\circ\varphi_1\mapsto (L-\varphi_1(L))+(\varphi_1(L)-\varphi_2\circ\varphi_1(L)).$$ In particular, in the case of the Ceresa cycle, $\varphi^-$ is of order two, and is mapped to $L-L^-$ of infinite order in $\Cob^{or}(X(J(C)))_{hom}$.

	\hspace{3mm}
	
\printbibliography

@article{abouzaid_fukaya_cat_higher_genus,
	title = {On the Fukaya categories of higher genus surfaces},
	journal = {Advances in Mathematics},
	volume = {217},
	number = {3},
	pages = {1192-1235},
	year = {2008},
	author = {Mohammed Abouzaid}
}

@article{Ceresa,
  title = {C Is Not {{Algebraically Equivalent}} to {{C-}} in Its {{Jacobian}}},
  author = {Ceresa, Giuseppe},
  date = {1983},
  journaltitle = {Annals of Mathematics},
  volume = {117},
  number = {2},
  pages = {285--291},
  publisher = {Annals of Mathematics}
}

@article{Ceresa_class_corey_et_al, title={The Ceresa class and tropical curves of hyperelliptic type}, volume={12}, DOI={10.1017/fms.2024.36}, journal={Forum of Mathematics, Sigma}, author={Corey, Daniel and Li, Wanlin}, year={2024}, pages={e54}}

@article{zharkov_tropical_ceresa,
	title={C Is Not {{Algebraically Equivalent}} to {{C-}} in Its {{Jacobian}}: A tropical point of view},
	author={Zharkov, Ilia},
	date={2013},
	journal={International Mathematics Research Notices},
	volume={2015},
	issue={3},
	pages={817-829}
}

@article{nick_ivan1,
  title = {Lagrangian Cobordism and Tropical Curves},
  author = {Sheridan, Nick and Smith, Ivan},
  date = {2021},
  journaltitle = {Journal für die reine und angewandte Mathematik (Crelles Journal)},
  volume = {2021},
  number = {774},
  pages = {219--265},
  publisher = {De Gruyter}
}

@article{nick_ivan_K3_mirror_symmetry,
  title = {Symplectic Topology of {{K3}} Surfaces via Mirror Symmetry},
  author = {Sheridan, Nick and Smith, Ivan},
  date = {2020},
  journaltitle = {Journal of the American Mathematical Society},
  shortjournal = {J. Amer. Math. Soc.},
  volume = {33},
  number = {3},
  eprint = {1709.09439},
  eprinttype = {arXiv},
  eprintclass = {math},
  pages = {875--915}
}

@article{nick_ivan_K3_rational_equivalence,
  title = {Rational Equivalence and {{Lagrangian}} Tori on {{K3}} Surfaces},
  author = {Sheridan, Nick and Smith, Ivan},
  date = {2020},
  journaltitle = {Commentarii Mathematici Helvetici},
  shortjournal = {Comment. Math. Helv.},
  volume = {95},
  number = {2},
  eprint = {1809.03892},
  eprinttype = {arXiv},
  eprintclass = {math},
  pages = {301--337}
}

@book{fulton_intersection_theory,
  title = {Intersection {{Theory}}},
  publisher={Springer New York, NY},
  year={1998},
  doi={10.1007/978-1-4612-1700-8},
  author = {Fulton, William},
  langid = {english}
}

@incollection {fukaya_family_floer_progress_report,
	AUTHOR = {Fukaya, Kenji},
	TITLE = {Floer homology for families---a progress report},
	BOOKTITLE = {Integrable systems, topology, and physics ({T}okyo, 2000)},
	SERIES = {Contemp. Math.},
	VOLUME = {309},
	PAGES = {33--68},
	PUBLISHER = {Amer. Math. Soc., Providence, RI},
	YEAR = {2002},
	ISBN = {0-8218-2939-4},
	DOI = {10.1090/conm/309/05341}
}

@article{Arnold_Lag_Leg_cobordisms,
  title = {Lagrange and Legendre Cobordisms. {{I}}},
  author = {Arnol'd, Vladimir},
  date = {1980-07-01},
  journaltitle = {Functional Analysis and Its Applications},
  shortjournal = {Funct Anal Its Appl},
  volume = {14},
  number = {3},
  pages = {167--177},
  langid = {english}
}

@incollection{gillet_Ktheory_intersection_theory,
  title = {K-{{Theory}} and {{Intersection Theory}}},
  booktitle = {Handbook of {{K-Theory}}},
  author = {Gillet, Henri},
  editor = {Friedlander, Eric M. and Grayson, Daniel R.},
  date = {2005},
  pages = {235--293},
  publisher = {Springer},
  location = {Berlin, Heidelberg},
  langid = {english}
}

@online{mikhalkin_trop_to_Lag_corresp,
  title = {Examples of Tropical-to-{{Lagrangian}} Correspondence},
  author = {Mikhalkin, Grigory},
  date = {2019},
  volume={5},
  pages={1033-1066},
  journal={European Journal of Mathematics}
}

@book{johnnyevans_ltf,
  title = {Lectures on {{Lagrangian}} Torus Fibrations},
  author = {Evans, Jonathan David},
  date = {2022-10-28},
  eprint = {2110.08643},
  eprinttype = {arXiv}
}

@online{unob_lag_cob_gps_surfaces,
  title = {Unobstructed {{Lagrangian}} Cobordism Groups of Surfaces},
  author = {Rathel-Fournier, Dominique},
  date = {2023},
  eprint = {2307.03124},
  eprinttype = {arXiv},
  eprintclass = {math},
  pubstate = {prepublished}
}

@online{alvaro_bielliptic,
  title = {Lagrangian Cobordisms and {{K-theory}} of Symplectic Bielliptic Surfaces},
  author = {Muñiz-Brea, Álvaro},
  date = {2024},
  eprint = {2403.17098},
  eprinttype = {arXiv},
  eprintclass = {math},
  pubstate = {prepublished}
}

@book{tropical_book,
  title = {Tropical {{Geometry}}},
  publisher={Book in preparation},
  year={2015},
  author = {{Mikhalkin, Grigory} and {Rau, Johannes}}
}

@article{Mak_Ruddat_tropical_Lag_quintic_threefolds, title={Tropically constructed Lagrangians in mirror quintic threefolds}, volume={8}, DOI={10.1017/fms.2020.54}, journal={Forum of Mathematics, Sigma}, author={Mak, Cheuk Yu and Ruddat, Helge}, year={2020}, pages={e58}}

@article{jeff_lag_cob_wall_crossing,
	title={Wall-crossing from Lagrangian cobordisms},
	author={Jeffrey Hicks},
	journal={Algebraic \& Geometric Topology},
	volume={24},
	year={2024},
	issue={6},
	pages={3069--3138}
}

@article{Clemens_Griffiths_infinitely_generated,
  title = {Homological Equivalence, modulo Algebraic Equivalence, Is Not Finitely Generated},
  author = {Clemens, Herbert},
  date = {1983},
  journaltitle = {Publications Mathématiques de l'Institut des Hautes Études Scientifiques},
  shortjournal = {Publications Mathématiques de l’Institut des Hautes Études Scientifiques},
  volume = {58},
  number = {1},
  pages = {19--38},
  langid = {english}
}

@article{Biran_Cornea_2,
  title = {Lagrangian Cobordism and {{Fukaya}} Categories},
  author = {Biran, Paul and Cornea, Octav},
  date = {2014},
  journaltitle = {Geometric and Functional Analysis},
  shortjournal = {Geometric and Functional Analysis},
  volume = {24},
  number = {6},
  pages = {1731--1830},
}

@article{matessi_lag_pair_of_pants,
  title = {Lagrangian {{Pairs}} of {{Pants}}},
  author = {Matessi, Diego},
  date = {2021},
  journaltitle = {International Mathematics Research Notices},
  volume = {2021},
  number = {15},
  pages = {11306--11356}
  }

@online{[zotero]g,
  title = {Affine Structures and Non-Archimedean Analytic Spaces},
  author = {Kontsevich, Maxim and Soibelman, Yan},
  date = {2004-06-28},
  eprint = {math/0406564},
  eprinttype = {arXiv},
  langid = {english},
  pubstate = {prepublished}
}

@book{Abelian_varieties_book,
  title = {Curves and {{Abelian Varieties}}},
  author = {Alexeev, Valery and Beauville, Arnaud and Clemens, C. Herbert and Izadi, Elham},
  date = {2008},
  series = {Contemporary {{Mathematics}}},
  volume = {465},
  publisher = {American Mathematical Society},
  langid = {english}
}

@article{Mikhalkin_Zharkov_trop_curves_jac_theta_funct,
  title = {Tropical Curves, Their {{Jacobians}} and Theta Functions},
  author = {Mikhalkin, Grigory and Zharkov, Ilia},
  editor = {Alexeev, Valery and Beauville, Arnaud and Clemens, C. Herbert and Izadi, Elham},
  date = {2008},
  journaltitle = {Contemporary Mathematics},
  volume = {465},
  pages = {203--230},
  publisher = {American Mathematical Society},
  location = {Providence, Rhode Island},
  langid = {english}
}

@article{local_coefficients_steenrod,
  title = {Homology {{With Local Coefficients}}},
  author = {Steenrod, Norman},
  date = {1943},
  journaltitle = {The Annals of Mathematics},
  shortjournal = {The Annals of Mathematics},
  volume = {44},
  number = {4},
  pages = {610-627}
}

@article{[zotero]ad,
  title = {The {{Ceresa}} Class and Tropical Curves of Hyperelliptic Type},
  author = {Corey, Daniel and Li, Wanlin},
  date = {2024},
  journaltitle = {Forum of Mathematics, Sigma},
  shortjournal = {Forum of Mathematics, Sigma},
  volume = {12}
}

@article{Beauville_torsion_Ceresa_class,
  title = {A Non-Hyperelliptic Curve with Torsion {{Ceresa}} Class},
  author = {Beauville, Arnaud},
  date = {2021-09-17},
  journaltitle = {Comptes Rendus. Mathématique},
  volume = {359},
  number = {7},
  pages = {871--872},
  langid = {english}
}

@article{Beauville_torsion_Ceresa_cycle,
  title = {A {{Non-Hyperelliptic Curve}} with {{Torsion Ceresa Cycle Modulo Algebraic Equivalence}}},
  author = {Beauville, Arnaud and Schoen, Chad},
  date = {2023-03-01},
  journaltitle = {International Mathematics Research Notices},
  volume = {2023},
  number = {5},
  pages = {3671--3675},
  langid = {english}
}

@article{Chan_tropical_hyperelliptic_curves,
  title = {Tropical Hyperelliptic Curves},
  author = {Chan, Melody},
  date = {2013},
  journaltitle = {Journal of Algebraic Combinatorics},
  volume = {37},
  number = {2},
  pages = {331--359}
}

@article{Corey_tropical_curves_hyperelliptic_type,
  title = {Tropical Curves of Hyperelliptic Type},
  author = {Corey, Daniel},
  date = {2021},
  journaltitle = {Journal of Algebraic Combinatorics},
  volume = {53},
  issue = {4},
  pages = {1215--1229},
  langid = {english}
}

@online{ceresa_period_tropical,
  title = {The {{Ceresa}} Period from Tropical Homology},
  author = {Ritter, Caelan},
  date = {2024},
  eprint = {2405.07402},
  eprinttype = {arXiv},
  eprintclass = {math},
  pubstate = {prepublished}
}

@article{Reznikov_characteristic_classes_in_symplectic_topology,
  title = {Characteristic Classes in Symplectic Topology},
  author = {Reznikov, Alexander},
  date = {1997},
  journaltitle = {Selecta Mathematica},
  shortjournal = {Sel. math., New ser.},
  volume = {3},
  number = {4},
  pages = {601--642},
  langid = {english}
}

@article{bosshard_Lag_cob_Liouville,
  title = {Lagrangian Cobordisms in {{Liouville}} Manifolds},
  author = {Bosshard, Valentin},
  date = {2024},
  journaltitle = {Journal of Topology and Analysis},
  shortjournal = {J. Topol. Anal.},
  volume = {16},
  number = {05},
  pages = {777--831},
  langid = {english}
}

@article{bc1,
  title = {Lagrangian {{Cobordism}}. {{I}}},
  author = {Biran, Paul and Cornea, Octav},
  date = {2013},
  journaltitle = {Journal of the American Mathematical Society},
  volume = {26},
  number = {2},
  eprint = {43302800},
  eprinttype = {jstor},
  pages = {295--340},
  publisher = {American Mathematical Society}
}

@inproceedings{hicks_mak_cute,
  title = {Some Cute Applications of {{Lagrangian}} Cobordisms towards Examples in Quantitative Symplectic Geometry},
  author = {Hicks, Jeffrey and Mak, Cheuk Yu},
  date = {2022-08-30}
}

@book{seidel_book,
  title = {Fukaya {{Categories}} and {{Picard-Lefschetz}} Theory},
  author = {{Seidel Paul}},
  date = {2008-06-26},
  publisher = {European Mathematical Society},
  langid = {english},
  pagetotal = {334}
}

@article{haug_T2,
  title = {The {{Lagrangian}} Cobordism Group of $T^2$},
  author = {Haug, Luis},
  date = {2015},
  journaltitle = {Selecta Mathematica},
  volume={21},
  issue={3},
  pages = {1021--1069},
  langid = {english}
}

@article{haug_antisurgery,
  title = {Lagrangian Antisurgery},
  author = {Haug, Luis},
  date = {2020},
  journaltitle = {Mathematical Research Letters},
  volume = {27},
  number = {5},
  pages = {1423--1464},
  langid = {english}
}

@article{trop_intersection,
  title = {First Steps in Tropical Intersection Theory},
  author = {Allermann, Lars and Rau, Johannes},
  date = {2010-03-01},
  journaltitle = {Mathematische Zeitschrift},
  shortjournal = {Math. Z.},
  volume = {264},
  number = {3},
  pages = {633--670},
  langid = {english}
}

@article{tropical_homology_ikmz,
	TITLE = {{Tropical Homology}},
	AUTHOR = {Itenberg, Ilia and Katzarkov, Ludmil and Mikhalkin, Grigory and Zharkov, Ilia},
	JOURNAL = {{Mathematische Annalen}},
	PUBLISHER = {{Springer Verlag}},
	YEAR = {2019},
	HAL_ID = {hal-03983044},
	volume={374},
	HAL_VERSION = {v1},
	pages={963-1006}
}

@article{weil_criteres_equivalence, 
	title={Sur les critères d'équivalence en géométrie algébrique}, 
	author={Weil, André}, 
	date={1954}, 
	journaltitle={Math. Ann.}, 
	volume={128}, 
	pages={95-127}
	}

@article{parameter_spaces_alg_equiv,
	title={Parameter spaces for algebraic equivalence},
	author={Jeff Achter and Sebastian Casalaina-Martin and Charles Vial},
	journal={arXiv: Algebraic Geometry},
	year={2016},
	url={https://api.semanticscholar.org/CorpusID:119307498}
}

@phdthesis{subotic_monoidal_structure, title={A monoidal structure for the Fukaya category}, author={Aleksandar Subotic}, date={2010}, school={Harvard University}, type={PhD thesis}}

@book{abelian_varieties_lang, author={Serge Lang}, title={Abelian varieties},publisher={Springer}, date={1983}}

@misc{bolognese2017curvestropicaljacobians,
	title={From Curves to Tropical Jacobians and Back}, 
	author={Barbara Bolognese and Madeline Brandt and Lynn Chua},
	year={2017},
	eprint={1701.03194},
	archivePrefix={arXiv},
	primaryClass={math.AG}
}

@article{Chua2017SchottkyAC,
	title={Schottky algorithms: Classical meets tropical},
	author={Lynn Chua and Mario Kummer and Bernd Sturmfels},
	journal={Math. Comput.},
	year={2017},
	volume={88},
	pages={2541-2558}
}

@article{bloch_cycles_on_abelian_varieties, 
	title={Some elementary theorems about algebraic cycles on Abelian varieties}, 
	author={Spencer Bloch}, 
	journal={Inventiones Mathematicae}, 
	volume={37}, 
	issue={3},
	pages={215-228}, 
	date={1976}, 
	ISSN={1432-1297}}

@misc{chassé2024weinsteinexactnessnearbylagrangians,
	title={Weinstein exactness of nearby Lagrangians and the Lagrangian $C^{0}$ flux conjecture}, 
	author={Jean-Philippe Chassé and Rémi Leclercq},
	year={2024},
	eprint={2410.04158},
	archivePrefix={arXiv},
	primaryClass={math.SG}
}

@article{Chekanov1996LagrangianTI,
	title={Lagrangian tori in a symplectic vector space and global symplectomorphisms},
	author={Yu. Chekanov},
	journal={Mathematische Zeitschrift},
	year={1996},
	volume={223},
	pages={547-559}
}

@article{ww_quilted_floer_cohomology,
	title={Quilted Floer Cohomology},
	author={Wehrheim, Katrin, and Chris T. Woodward},
	journal={Geometry \& Topology},
	year={2010},
	volume={14.2},
	pages={833–902}
}

@misc{brendel_kim,
	title={Lagrangian split tori in $S^2 \times S^2$ and billiards}, 
	author={Joé Brendel and Joontae Kim},
	year={2025},
	eprint={2502.03324},
	archivePrefix={arXiv},
	primaryClass={math.SG},
}
\nocite{nick_ivan_K3_mirror_symmetry,nick_ivan_K3_rational_equivalence,Abelian_varieties_book}
\nocite{Clemens_Griffiths_infinitely_generated,Mak_Ruddat_tropical_Lag_quintic_threefolds,[zotero]ad}
\nocite{[zotero]g,fulton_intersection_theory,jeff_lag_cob_wall_crossing,fukaya_family_floer_progress_report}

\end{document}